\documentclass{amsart}
\usepackage{amssymb, amsfonts, lacromay}
\usepackage{mathrsfs,color,comment}
\usepackage{url}

\usepackage[all,arc]{xy}
\usepackage{enumerate}
\usepackage{mathrsfs}

\newtheorem{thm}{Theorem}[section]

\newtheorem{prop}[thm]{Proposition}
\newtheorem{lem}[thm]{Lemma}
\newtheorem{conj}[thm]{Conjecture}

\theoremstyle{definition}
\newtheorem{defn}[thm]{Definition}

\newtheorem{noexmp}[thm]{Counterexample}
\newtheorem{exmp}[thm]{Example}

\newtheorem{rem}[thm]{Remark}

\newtheorem{sch}[thm]{Scholium}

\makeatletter
\let\c@equation\c@thm
\makeatother
\numberwithin{equation}{section}

\makeatletter
\ifx\SK@label\undefined\let\SK@label\label\fi
 \let\your@thm\@thm
 \def\@thm#1#2#3{\gdef\currthmtype{#3}\your@thm{#1}{#2}{#3}}
 \def\mylabel#1{{\let\your@currentlabel\@currentlabel\def\@currentlabel
  {\currthmtype~\your@currentlabel}
 \SK@label{#1@}}\label{#1}}
 \def\myref#1{\ref{#1@}}
\makeatother

\newcommand{\coldbit}[1]{\textcolor{orange}{#1}}

\renewcommand{\coldbit}[1]{}

\begin{document}


\title[Equivariant iterated loop spaces and permutative $G$-categories]{Equivariant iterated loop space theory and permutative $G$-categories}

\author{B.J. Guillou}
\address{Department of Mathematics, The University of Kentucky, Lexington, KY, 40506}
\email{bertguillou@uky.edu}
\author{J.P. May}
\address{Department of Mathematics, The University of Chicago, Chicago, IL 60637}
\email{may@math.uchicago.edu}



\begin{abstract}  We set up operadic foundations for equivariant iterated
loop space theory.   We start by building up from a discussion of the approximation theorem
and recognition principle for $V$-fold loop $G$-spaces to several avatars of a recognition principle for
infinite loop $G$-spaces.  We then explain what genuine permutative $G$-categories are and, more 
generally, what $E_{\infty}$ $G$-categories are, giving examples showing how they arise.   As an 
application, we prove the equivariant Barratt-Priddy-Quillen theorem as a statement
about genuine $G$-spectra and use it to give a new, categorical, proof
of the tom Dieck splitting theorem for suspension $G$-spectra.
Other examples are geared towards equivariant algebraic $K$-theory.
\end{abstract}

\maketitle

\tableofcontents



\section*{Introduction}

Let $G$ be a finite group. We will develop equivariant infinite loop space
theory in a series of papers. In this introductory one, we focus on the operadic equivariant infinite
loop space machine.  This is the most topologically grounded machine, as we illustrate
by first focusing on its relationship to $V$-fold deloopings for $G$-representations $V$.  Genuine
permutative $G$-categories and, more generally, $E_{\infty}$ $G$-categories are also 
defined operadically. They  provide the simplest categorical input needed to construct genuine 
$G$-spectra from categorical input. 

For background, naive $G$-spectra are just spectra with actions by $G$.
They have their uses, but they are not adequate for serious work in equivariant
stable homotopy theory.  The naive suspension $G$-spectra of spheres $S^n$ with 
trivial $G$-action are invertible in the naive equivariant stable homotopy category. In contrast, 
for all real orthogonal $G$-representations $V$,
the genuine suspension $G$-spectra of $G$-spheres $S^V$ are invertible in the
genuine equivariant stable homotopy category, where $S^V$ is the one-point compactification
of  $V$.  Naive $G$-spectra represent $\bZ$-graded cohomology
theories, whereas genuine $G$-spectra represent 
cohomology theories graded on the real representation ring $RO(G)$.  The $RO(G)$-grading is essential for Poincar\'e duality
and, surprisingly, for many nonequivariant applications.

The zeroth space $E_0=\OM^{\infty}E$ of a naive $\OM$-$G$-spectrum is an infinite loop
$G$-space in the sense that it is equivalent to an $n$-fold loop $G$-space 
$\OM^n E_n$ for each $n\geq 0$.  The zeroth space $E_0$ of a genuine $\OM$-$G$-spectrum 
$E$ is an infinite loop $G$-space in the sense that it is  equivalent to a $V$-fold
loop $G$-space $\OM^VE(V)$ for all real representations $V$. The essential point
of equivariant infinite loop space theory is to construct $G$-spectra from space or category level
data. Such a result is called a recognition principle since it allows us to recognize infinite 
loop $G$-spaces when we see them.  A functor that constructs $G$-spectra from $G$-space or 
$G$-category level input is called an equivariant infinite loop space machine.  

As we shall see, a recognition principle for naive $G$-spectra is obtained 
simply by letting $G$ act in the obvious way on the input data familiar
from the nonequivariant theory. One of our main interests is to construct and apply an 
equivariant infinite loop space machine that constructs genuine $G$-spectra from
categorical input.

A permutative category is a symmetric strictly associative and unital 
monoidal category, and any symmetric monoidal category is equivalent to a 
permutative category.  The
classifying space of a permutative category $\sA$ is rarely an infinite loop space, 
but infinite loop space theory constructs an $\OM$-spectrum  $\bK \sA$ whose zeroth
space is a group completion of the classifying space $B\sA$.  A naive permutative
$G$-category is a permutative category that is a $G$-category with equivariant
structure data. It is a straightforward  adaptation of the nonequivariant theory to 
construct naive $G$-spectra $\bK \sA$ from naive permutative $G$-categories $\sA$ in 
such a way that $\bK_0\sA$ is a group completion of $B\sA$, meaning that $(\bK_0\sA)^H$ is
a nonequivariant group completion of $B(\sA^H)$ for all subgroups $H$ of $G$.

In this paper, we explain what genuine
permutative $G$-categories are and what $E_{\infty}$ $G$-categories are, and 
we explain how to construct a genuine $G$-spectrum $\bK_G\sA$ from a genuine
permutative $G$-category $\sA$ or, more generally, from an $E_{\infty}$ $G$-category $\sA$. 
A genuine $G$-spectrum has an underlying naive $G$-spectrum, and the underlying naive 
$G$-spectrum of $\bK_G\sA$ will be $\bK \sA$.  Therefore we still have the crucial
group completion property relating $B\sA$ to the zeroth $G$-space of $\bK_G\sA$.

We use this theory to show how to construct suspension $G$-spectra from categorical
data, giving a new equivariant version of the classical Barratt-Priddy-Quillen (BPQ)
theorem for the construction of the sphere spectrum from symmetric groups. In \cite{GMMOMult},
we shall use this version of the BPQ theorem as input to a  proof of the results from
equivariant infinite loop space theory that were promised in \cite{GM2}, where we described the category 
of $G$-spectra as an easily understood category of spectral presheaves.  Here we use this 
version of the BPQ theorem to give a new categorical proof of the tom Dieck splitting 
theorem for the fixed point spectra of suspension $G$-spectra. The new proof is
simpler and gives more precise information than the classical proof by induction up orbit types.

A complementary interest is to understand the geometry of $V$-fold loop $G$-spaces.  
As we shall explain in this paper, these interests lead to quite different perspectives.
They are manifested in point set level distinctions that would be invisible to a more
abstract approach.   One way of pinpointing these differences is to
emphasize the distinction between the role played by $E_V$ operads for representations $V$,
which are the equivariant generalizations of $E_n$ operads, and the role played by (genuine)
$E_{\infty}$-operads of $G$-spaces.

An $E_V$-space is a $G$-space with an action by an $E_V$-operad. We here develop a 
machine that constructs $V$-fold loop $G$-spaces from 
$E_V$-spaces.  For future perspective, we envision 
the possibility of an equivariant version of factorization homology in which $E_V$ operads will govern local 
structure of $G$-manifolds in analogy  with the role played by $E_n$ operads in the existing non-equivariant theory.  
For such a theory, $E_{\infty}$ operads would be essentially irrelevant.

In contrast, for infinite loop space theory $E_V$ operads serve merely as scaffolding used to build a machine that 
constructs genuine $G$-spectra from $E_{\infty}$ $G$-spaces, which are spaces with an action by 
some $E_{\infty}$ operad.  The classifying $G$-spaces of genuine permutative $G$-categories are examples of 
$E_{\infty}$ $G$-spaces with actions by a particular $E_{\infty}$ operad $\sP_G$, but $E_{\infty}$ $G$-spaces with 
actions by quite different $E_{\infty}$ operads abound.  We concentrate on such an operadic machine in this paper.   The machine 
we concentrate on in the sequels \cite{GMMOAdd, GMMOMult, MMO} makes no use of $E_V$-operads and does not 
recognize $V$-fold loop $G$-spaces, but it allows a level of categorical power and multiplicative control that is 
unobtainable with the machine built here.

This paper offers a  number of variant perspectives on the topics it studies.  We give recognition
principles for $V$-fold loop spaces (\myref{Vrec}), for orthogonal $G$-spectra (\myref{KStarrec}
and \myref{Einforth}) and, preserving space level structure invisible in orthogonal $G$-spectra, for Lewis-May 
$G$-spectra (\myref{EinfLM} and \myref{Vrec2}). 
  The geometric input data for \myref{Vrec} consists 
of algebras over the little disks or Steiner operad $\sD_V$ or $\sK_V$.  For \myref{KStarrec}, it consists 
of compatible algebras over the $\sK_V$ for all finite dimensional $V$.  

In both Definitions \ref{Einforth} and \ref{EinfLM}, the input data consists of algebras over an $E_\infty$ operad 
of $G$-spaces.   These algebras may come 
by applying the classifying space functor $B$ to algebras over an $E_{\infty}$ operad of $G$-categories.
The orthogonal spectrum machine and the Lewis-May spectrum machine are shown to 
be equivalent by comparing them both to a machine landing in the $S_G$-modules of EKMM \cite{EKMM, MM}.  
In effect, the machines landing in Lewis-May $G$-spectra and in $S_G$-modules provide highly structured fibrant
approximations of the machine landing in orthogonal $G$-spectra.   In retrospect, such fibrant approximation
is central to nonequivariant calculational understanding, and one can hope that the same will eventually
prove true equivariantly.

The variants have alternative and contradictory good features, which become particularly apparent and relevant
when specialized to free $E_{\infty}$ algebras, where they are all viewed as giving variants of the equivariant BPQ theorem.    
Thinking unstably and geometrically, \myref{unstableBPQ} shows how the
machine recognizes $V$-fold suspensions $\SI^V X$ and shows that the recognition is precisely compatible with 
the evident $G$-homeomorphisms $\SI^VX \sma \SI^W Y \iso \SI^{V\oplus W} Y$.   Thinking stably and geometrically,
Theorems  \ref{BPQ2} and \ref{BPQ3} show how the machine recognizes orthogonal or Lewis-May suspension 
$G$-spectra $\SI^{\infty}_GX$. In both cases, the recognition is precisely compatible with the standard
$G$-isomorphisms $\SI^{\infty}_G X\sma \SI^{\infty}_G Y \iso \SI^\infty_G (X\sma Y)$.  However, the meaning of $\SI^{\infty}_G$ is quite
different in the two cases.  For orthogonal $G$-spectra, $\SI^{\infty}_GX$ is cofibrant if $G$ is cofibrant as 
a $G$-space, but it is never fibrant.  For Lewis-May or EKMM $G$-spectra, $\SI^{\infty}_GX$ is always fibrant
and often bifibrant. 

Theorems \ref{SpThree} and \ref{GBPQ} show how the machine recognizes suspension 
$G$-spectra from two variant categorical inputs.  Here we do {\em not} have precise compatibility with
smash products, a failure that will be rectified with a hefty dose of $2$-category theory in the sequel \cite{GMMOMult},
but instead we have structure that allows our new proof of the tom Dieck splitting theorem.

As already mentioned, there are three sequels to this paper.  The first, \cite{MMO}, develops a new version of the 
Segal-Shimakawa infinite loop space machine and proves among other things that it is equivalent both to the original
Segal-Shimakawa machine and to the machine landing in orthogonal $G$-spectra that we develop here. That requires a 
generalization of the present machine from operads to categories of operators, about which we say nothing here.  The second and third
\cite{GMMOAdd, GMMOMult} give a more categorically sophisticated machine, the first purely additive and 
the second building in multiplicative structure.  These start with more general categorical input than we deal with here and 
give new information even nonequivariantly. 

\subsection*{Outline}

We begin with a machine for recognizing iterated equivariant loop spaces in \S\ref{sec:VLoops}. 
All versions of our iterated loop space machine are based on use  of the Steiner operads, whose
equivariant versions have not previously appeared.  We define them and compare them to the 
little disks operads in \S\ref{sec:DisksSteiner}.  All versions are also based on an approximation theorem, which is 
explained in \S\ref{approxSec}.  We use a strengthened version due to Rourke and Sanderson 
\cite{RourSan2}, and that allows us to obtain slightly stronger versions of the recognition principle 
than might be expected.  The compatibility with smash products of the geometric versions of the
recognition principle is based on pairings between Steiner operads that are defined in 
\S\ref{SecSteinerPair}; the relevant definition of a pairing is recalled in \S\ref{SecPair}.  The 
promised variants of the recognition principle starting from space level input data are given in 
\S\S \ref{sec:RecogPrin}, \ref{GeoRec}, \ref{sec:MachineGS}, and \ref{sec:MachineGSp}. 

Section \ref{Sec5} gives our machines for recognizing infinite loop $G$-spaces.  After recalling the notion of 
$E_\infty$ $G$-operad and giving some examples in \S2.1, the orthogonal and Lewis-May machines are 
defined and compared in \S2.2--\S2.4.  Examples of $E_\infty$ $G$-spaces are given in \S\ref{sec:Exmps}. 
General properties that must hold for any equivariant infinite loop space machine are 
described in \S\ref{sec:MachineProps}.  A recognition principle for naive $G$-spectra, with $G$ not 
necessarily finite, is given in \S\ref{SecNaive}.   An interesting detail  there shows how to use the 
recognition principle to construct change of  universe functors on the space level.  The 
proof uses a double bar construction described in \S\ref{DoubleTrouble}.  

The recognition principle starting from categorical imput is given in \S\ref{SecKGB}.  It is preceded
by preliminaries about equivariant universal bundles and equivariant $E_{\infty}$ operads in \S\ref{sec:Prelim}
and by a discussion of operadic definitions of naive and genuine permutative $G$-categories in \S\ref{Sec2}.  
In the brief and parenthetical \S\ref{SecAlgKThy}, we point out how these ideas and our prequel \cite{GMM} with 
Merling specialize to give a starting point for equivariant algebraic $K$-theory \cite{DK, FHM, Kuku, Merling}.  
We give an alternative and equivalent starting point in the case of $G$-rings $R$ in \S\ref{SecKGBtoo}.

We give a precise description of the $G$-fixed $E_{\infty}$ categories of 
free $\sP_G$-categories in \S\ref{Sec3}. This is a precursor of our first
categorical version of the BPQ theorem, which we prove in \S\ref{Sec5.1}, 
and of the tom Dieck splitting theorem for suspension $G$-spectra, which we reprove
in \S\ref{Sec5.2}.

Changing focus, in \S\ref{SecPQR} and \S\ref{sec:Einf} we give three interrelated examples of $E_{\infty}$ $G$-operads, denoted 
$\sV_G$, $\sV^\times_G$, and $\sW_G$, and give examples of their algebras.
This approach to examples is more intuitive than the approach based on
genuine permutative $G$-categories, and it has some technical advantages.  It is new and illuminating even nonequivariantly.  It gives a more intuitive categorical hold on the BPQ theorem than does the treatment starting from genuine permutative $G$-categories, as we explain in \S\ref{subSecbait}.  It also gives a new starting point for multiplicative infinite loop space theory, both equivariantly and nonequivariantly, but that is work in progress.  

\subsection*{Notational preliminaries}

A dichotomy between Hom objects with $G$-actions 
and Hom objects of equivariant morphisms, often denoted using a $G$ in front, is omnipresent.  
We start with an underlying category $\sV$. A $G$-object $X$ in $\sV$ 
can be defined to be a group homomorphism $G\rtarr \Aut X$. We have the category 
$\sV_G$ of $G$-objects in $\sV$ and all morphisms in $\sV$ between them, with $G$ acting by conjugation.  
We denote the morphism objects of $\sV_G$ simply by $\sV(X,Y)$.\footnote{In  \cite{MM} and elsewhere,
we used the notation $\sV_G(X,Y)$ instead of $\sV(X,Y)$, but some readers found that misleadingly analogous to $\Hom_G(X,Y)$.}
We also have the category $G\sV$ of $G$-objects in $\sV$ 
and $G$-maps in $\sV$. Since objects are fixed by $G$, $G\sV$ is in fact the 
$G$-fixed category $(\sV_G)^G$, although we shall not use that notation.  Thus the
hom object $G\sV(X,Y)$ in $\sV$ of $G$-morphisms between $G$-objects $X$ and $Y$
is the fixed point object $\sV(X,Y)^G$.

One frequently used choice of $\sV$ is $\sU$, the category of unbased (compactly generated) spaces.  We let $\sT$ denote the category of based spaces. We assume once and for all that the basepoints $\ast$ of all given based $G$-spaces $X$ (or spaces $X$ when $G=e$) are nondegenerate. This means that $\ast\rtarr X$ is a $G$-cofibration (satisfies the $G$-HEP). It follows that $\ast\rtarr X^H$ is a cofibration for all $H\subset G$.

By an equivalence $f\colon X\rtarr Y$ of $G$-spaces we understand a $G$-map whose fixed point
maps $f^H\colon X^H\rtarr Y^H$ are weak homotopy equivalences for all subgroups $H$ of $G$.
When $X$ and $Y$ have the homotopy types of $G$-CW complexes, such an $f$ is a $G$-homotopy equivalence.

By a topological category $\sC$ we understand a category internal to $\sU$; thus it has an object space and a morphism space such that the structural maps $I$, $S$, $T$, and $C$ are continuous.  This is more structure than a topologically enriched category, which would have a discrete space of objects.  We also have the based variant of categories internal to $\sT$, but $\sU$ will be
the default.  

We let $\sC\!at$ denote the category of (small) topological categories. As above, starting from
$\sC\!at$, we obtain the concomitant categories $G\sC\!at$ and $\sC\!at_G$ of $G$-categories.  
A $G$-category is a topological category equipped with an action of $G$ through natural isomorphisms.  This is the same structure as a category internal to $G\sU$.  Similarly, a based $G$-category, is a category internal to $G\sT$.  That is, an action of $G$ on a topological category $\sC$ is given by actions of $G$ on both the object space and the morphism space such that $I$, $S$, $T$, and $C$ are $G$-maps. In particular, $G$ can and often will act non-trivially on the space of objects.  That may be unfamiliar (as the referee noted), but in many of our examples it is essential for proper behavior on passage to $H$-fixed subcategories for $H\subset G$.

For brevity of notation, we shall often but not always write $|-|$ for the composite classifying space functor $B = |N-|$ from topological categories through simplicial spaces to spaces.
It works equally well to construct $G$-spaces from topological $G$-categories. We assume that the reader is familiar 
with operads (as originally defined in \cite{MayGeo}) and especially with the fact 
that operads can be defined in any symmmetic monoidal category $\sV$. 
Brief modernized expositions are given in \cite{MayOp1, MayOp2}. Since it is 
product-preserving, the functor $|-|$ takes operads in $\sC\!at$ or in $G\sC\!at$ 
to operads in $\sU$ or in $G\sU$, and it takes algebras over an operad $\sC$ in 
$\sC\!at$ or in $G\sC\!at$ to algebras over the operad $|\sC|$ in $\sU$ or in $G\sU$.

To avoid proliferation of letters, we shall write $\bO_G$ for the monad on based 
$G$-categories constructed from an operad $\oO_G$ of $G$-categories. We shall write 
$\mathbf{O}_G$ for the monad on based $G$-spaces constructed from the operad $|\sO_G|$ of $G$-spaces. 
More generally, for an operad $\sC_G$ of unbased $G$-spaces, we write
$\mathbf{C}_G$ for the associated monad on based $G$-spaces.

\subsection*{Acknowledgements} The first author thanks Nat Stapleton for
very helpful discussions leading to the rediscovery of the operad $\sP_G$,
which was in fact first defined, but not used, by Shimakawa \cite[Remark, p.\,255]{Shim1}.
The second author thanks Mona Merling for many conversations and questions 
that helped clarify ideas.  We both thank the referee for helpful suggestions.  
We also
thank Anna Marie Bohmann and Ang\'{e}lica Osorno for pointing out a mistake in the 
original version.  That led to a reworking of this paper and to much of the
work in the sequels \cite{GMMOAdd, GMMOMult, MMO}.   It also led to the long 
delay in the publication of this paper, which is entirely due to the authors and not at
all to the referee or editors. We thank them for their patience.

\section{$E_V$ operads and $V$-fold loop $G$-spaces}\label{sec:VLoops}

In this geometrically focused chapter, we first define $E_V$ operads  and give two
examples. We then relate $E_V$-spaces to $V$-fold loop $G$-spaces via the equivariant 
approximation theorem and recognition principle.  The approximation theorem shows how to approximate
``free'' $V$-fold loop $G$-spaces $\OM^V\SI^V X$ in terms of free algebras  $\bD_V X$ or
$\bK_V X$ over the $E_V$ operad $\sD_V$ or $\sK_V$.
The recognition principle shows how to construct $V$-fold loop spaces from $E_V$-algebras. 
 We elaborate multiplicatively by showing how machine-built 
pairings relate to evident pairings between iterated loop $G$-spaces.
We then give a geometric version 
of a concrete spacewise infinite loop $G$-space machine that does not use $E_{\infty}$ operads 
and is new even nonequivariantly.  This gives a geometric precursor of the BPQ theorem that relates 
well to smash products.   As already noted, we envision that the theory here can provide the local data for an
as yet undeveloped equivariant factorization homology theory.

\subsection{The little disks and Steiner operads}\label{sec:DisksSteiner}

\begin{defn} Let $D(V)$ be the open unit disk in $V$.  A little $V$-disk is a map $d\colon D(V)\rtarr D(V)$ 
of the form $d(v) = rv + v_0$ for some $r\in [0,1)$ and some $v_0\in V$; $c(d) = v_0$ is the
center point of $d$ and $r$ is the radius. For $g\in G$, $(gd)(v) = rv + gv_0$.  
Define $\sD_V(j)$ to be the $G$-space of (ordered) $j$-tuples of little $V$-disks whose images
have empty pairwise intersections.
With the evident structure maps determined by disjoint union and composites of little disks, the $\sD_V(j)$ form 
an operad $\sD_V$, called the little disks operad.
\end{defn}

For a $G$-space $V$, let $F(V,j)\subset V^j$ be the configuration space of (ordered)
$j$-tuples of distinct points of $V$, with $G$ acting by restriction of the 
diagonal action on $V^j$.   By convention, $F(V,0)$ is a point, the empty $0$-tuple 
of points in $V$.  We are interested in the special case when $V$
is a real representation of $G$, by which we understand an orthogonal action of $G$
on a real inner product space.  In contrast to the nonequivariant case, very
little is known about the (Bredon) homology and cohomology of the $G$-spaces
$F(V,j)$, but we have the following result.

\begin{lem}\label{HtpyLittleDisks}  There is a $(G\times \SI_j)$-homotopy equivalence $\sD_V(j) \rtarr F(V,j)$ for each $j\geq 0$.
\end{lem} 

\begin{proof} Choose a decreasing rescaling homeomorphism $\ze\colon [0,\infty) \rtarr [0,1)$ and also denote
by $\ze$ the rescaling homeomorphism $V\rtarr D(V)$ that sends $v$ to $\ze(|v|/|v|)v$, where $D(V)$ is the open
unit disc in $V$.
Then $\ze$ induces a rescaling homeomorphism $\ze\colon F(V,j) \rtarr F(D(V),j)$. Define a map
$c\colon \sD_V(j)\rtarr F(D(V),j)$
by sending little disks to their
center points.   For a point $ \ul v =(v_1, \cdots, v_j)$ in $F(D(V),j)$,  define 
$$ \de(\ul v) = 1/2 \, \text{min}\,  \{| v_i-v_j|, \  i\neq j \}. $$
Define $s\colon F(D(V),j)\rtarr \sD_V(j)$ by $s(\ul v) = (d_1,\cdots, d_j)$, where $d_j(v) = \de( \ul v) v + v_i $.
Then $s$ and $c$ are $(G\times \SI_j)$-maps, $c\com s = \id$,  and there is a  $(G\times \SI_j)$-homotopy 
$h\colon s\com c\htp \id$.  If  $\ul d = (d_1,\cdots, d_j)\in \sD(j)$, where $d_i(v) = r_i v + v_i$,   then $c(\ul d) = \ul v$ and
$h(\ul d, t)$ has $i$th little $V$-disk $d_i(t)$ given by $d_i(t)(v)= \big{(}(1-t)\de(\ul v) + t r_i)\big{)}v + v_i$. 
\end{proof}

The following definition is the equivariant generalization of the usual definition of an $E_n$-operad.
We say that a map of operads of $G$-spaces is a weak equivalence if its $j$th map is a weak
$(G\times \SI_j)$-equivalence.

\begin{defn} An operad $\sC_G$ of $G$-spaces is an $E_V$-operad if there
is a chain of weak equivalences of operads connecting $\sC_G$ to $\sD_V$.   
\end{defn}

Of course, we could use any operad weakly equivalent to $\sD_V$ as a reference operad in the definition.
As explained in \cite[\S3]{Rant1},  for inclusions $V\subset W$ of inner product spaces, there is no map of operads 
$\sD_V\rtarr \sD_W$ that is compatible with suspension, so that use of the little disks operads is inappropriate for iterated
loop space theory.   The Steiner operads remedy the defect and will be used in \cite{MMO} to compare the operadic
and Segalic equivariant  infinite loop space machines.  Their equivariant definition is little different from their
nonequivariant definition given in \cite{Rant1}, following Steiner \cite{St}. 

\begin{defn}
Let $E_V$ be the space of embeddings $V\rtarr V$, with $G$ acting by
conjugation, and let $\text{Emb}_V(j)\subset E_V^j$ be the 
$G$-subspace of (ordered) $j$-tuples of embeddings with pairwise disjoint images.  
Regard such a $j$-tuple as an embedding  $^jV\rtarr V$, where $^jV$ 
denotes the disjoint union of $j$ copies of $V$ (where $^0V$ is empty).  
The element $\id$ in $\text{Emb}_V(1)$ is the identity embedding, the group 
$\SI_j$ acts on $\text{Emb}_V(j)$ by 
permuting embeddings, and the structure maps 
\[
\ga\colon \text{Emb}_V(k)\times \text{Emb}_V(j_1)\times \cdots \times
\text{Emb}_V(j_k)\rtarr \text{Emb}_V(j_1+\cdots + j_k) 
\]
are defined by composition and disjoint union in the evident way \cite[\S3]{Rant1}.
This gives an operad $\text{Emb}_V$ of $G$-spaces.
 
Define $R_V\subset E_V=\text{Emb}_V(1)$ to be 
the sub $G$-space of distance reducing embeddings $f\colon V\rtarr V$.  This
means that  $|f(v)-f(w)|\leq |v-w|$ for all $v,w\in V$.  Define a 
Steiner path to be a map $h\colon I\rtarr R_V$ such that $h(1)=\id$
and let $P_V$ be the $G$-space of Steiner paths, with action of $G$
induced by the action on $R_V$.  Define 
$\pi\colon P_V\rtarr R_V$ by evaluation at $0$, $\pi(h) = h(0)$.  
 
Define $\sK_V(j)$ to be the $G$-space of (ordered) $j$-tuples $(h_1,\cdots,h_j)$ of
Steiner paths such that the $\pi(h_i)$ have disjoint images. The element 
$\id$ in $\sK_V(1)$ is the constant path at the identity embedding,
the group $\SI_j$ acts on $\sK_V(j)$ by permutations, and the structure
maps $\ga$ are defined pointwise in the same way as those of $\text{Emb}_V$.
This gives an operad of $G$-spaces, and application of $\pi$ to Steiner paths 
gives a map of operads $\pi\colon \sK_V\rtarr \text{Emb}_V$.  Evaluation of embeddings 
at $0\in V$ gives center point $(G\times \SI_j)$-maps $c\colon \text{Emb}_V(j)\rtarr  F(V,j)$.
\end{defn}

The Steiner operads $\sK_V$ are reduced, meaning that $\sK_V(0)$ is a point, and 
$\sK_0$ is  the trivial operad with $\sK_0(1)= \id$ and $\sK_0(j) = \emptyset$ for $j>1$.
By pullback along $\pi$, any space with an action by $\text{Emb}_V$ inherits 
an action by $\sK_V$.  As in \cite[\S5]{MayGeo}, \cite[VII\S2]{MQR}, or 
\cite[\S3]{Rant1}, $\text{Emb}_V$ acts naturally on $\OM^VX$ for  based $G$-spaces $X$.

\begin{prop}[\cite{St}]\label{SteinerHtpy}  There is a weak equivalence of operads $\io\colon \sD_V \rtarr \sK_V$.
\end{prop}
\begin{proof}
For each $j$, we have a composite $(G\times \SI_j)$-map
\[ c\com \pi\colon \sK_V(j) \rtarr \mathrm{Emb}_V(j) \rtarr F(V,j).\]
Steiner's nonequivariant proof that $c\com \pi$ 
is a $\SI_j$-homotopy equivalence
applies to prove that it is a $(G\times \SI_j)$-homotopy equivalence. The argument is a clever and non-trivial
variant on the proof above for $\sD_V$, but for us the essential point is that it uses the metric on $V$ and the contractibility
of $I$ and $V$ in such a way that the construction is clearly $G$-equivariant. 

For a little disk $d(v) = rv + v_0$, define a path of little disks from $d$ to the identity map of $D(V)$ by sending 
$s\in I$ to the little disk
$$d(s)(v) = (s-rs + r)v + (1-s) v_0.$$
Conjugating $d$ by the rescaling $\ze$ of Lemma~\ref{HtpyLittleDisks} gives a distance reducing embedding $\ze^{-1}d\ze\colon V\rtarr V$, and conjugating paths pointwise 
gives an embedding $\io$ of $\sD_V$ as a suboperad of $\sK_V$.   Composing the inverse $(G\times \SI_j)$-homotopy equivalence
$F(V,j)\rtarr \sD_V(j)$ with $\io\colon \sD_V(j)\rtarr \sK_V(j)$ gives an inverse $(G\times \SI_j)$-homotopy equivalence
to $c\com \pi$, by Steiner's proof, and it follows that $\io$ is a $(G\times \SI_j)$-homotopy equivalence.
\end{proof}

Again, one key advantage of the Steiner operads over the little disks operads is that, for an inclusion $V\subset W$ of $G$-inner product
spaces,  there is an induced inclusion
 $\sK_V \rtarr \sK_W$ of $G$-operads such that the map
\[ \OM^V \et \colon \OM^V X \rtarr  \OM^V\OM^{W-V}\SI^{W-V} X \iso \OM^{W} \SI^{W-V} X \]
is a map of $\sK_V$-spaces for any $G$-space $X$.  Here $W-V$ is the orthogonal complement of $V$ in $W$.  If $f:V\rtarr V$ is a distance-reducing embedding, then $f\oplus \id_{W-V}:W \rtarr W$ is also distance-reducing, and this construction induces the inclusion.

\subsection{The approximation theorem}\label{approxSec}

Write $\BK_V$ for the monad on based $G$-spaces associated to the operad $\sK_V$.   For a $G$-space $X$,
$\BK_VX = \coprod \sK_V(j) \times_{\SI_j} X^j/ (\sim)$.  If $\si_i \colon \sK_V(j) \rtarr \sK_V(j-1)$ deletes the 
$i$th Steiner path and $s_i \colon X^{j-1}\rtarr X^j$ inserts the basepoint in the $i$th position, then
$(\si^i k, y) \sim (k, s_i y)$ for $k\in \sK_V(j)$ and $y\in X^{j-1}$.   The monad $\BD_V$ arising from the operad $\sD_V$ is defined the same way.

The unit $\et\colon \Id\rtarr \OM^V\SI^V$ of the monad $\OM^V\SI^V$ and the action $\tha$ of 
$\mathbf{K}_V$ on the $G$-spaces $\OM^V\SI^V X$ induce a composite natural map 
\[ \xymatrix@1{ \al_V\colon \mathbf{K}_VX \ar[r]^-{\mathbf{K}_V \et} &  \mathbf{K}_V\OM^V\SI^V X \ar[r]^-{\tha} &\OM^V\SI^V X,\\} \]
and $\al_V \colon \mathbf{K}_V\rtarr \OM^V\SI^V$ is a map of monads whose adjoint defines a right action of $\mathbf{K}_V$ on the 
functor $\SI^V$, just as in \cite{MayGeo}.   The restriction to $\mathbf{D}_V$ gives the corresponding map $\al_V\colon \mathbf{D}_V X\rtarr \OM^V\SI^V X$. 

The heart of the operadic recognition principle is the approximation theorem that says that $\al_V$ is a group
completion.  However, already nonequivariantly, we have two variants of what it means for a map $ X\rtarr Y$ 
to be a group completion.  Recall that Hopf spaces are spaces with a product with a two-sided unit
element up to homotopy.

\begin{defn}  A Hopf space $Y$ is grouplike if $\pi_0(Y)$ is a group.  Let $X$ and $Y$ be homotopy associative and 
commutative Hopf spaces, where $Y$ is grouplike, and let $f\colon X\rtarr Y$ be a Hopf map.  
Then $f$ is a group completion if $f_*\colon \pi_0 (X)\rtarr \pi_0(Y)$ is the Grothendieck construction converting a commutative 
monoid to an abelian group and if, for any field of coefficients $k$, the map of commutative $k$-algebras 
$H_*(X)[\pi_0(X)^{-1}] \rtarr H_*(Y)$ induced by $f_*$ is an isomorphism.   
\end{defn}

The second version of group completion drops the commutativity assumption and lives in the setting of $A_\infty$-spaces. For us, 
an $A_\infty$-space will mean a space with an action of the Steiner operad $\sK_\bR$.  An $A_{\infty}$-map will mean either a map homotopic to a
map of $\sK_{\bR}$-spaces or the homotopy inverse of a map of  $\sK_{\bR}$-spaces that is an underlying homotopy equivalence.

\begin{defn}  An  $A_{\infty}$-map $f\colon X \rtarr Y$ of $\mathbf{K}_\bR$-spaces is a weak group 
completion if it is equivalent under a chain of $A_{\infty}$-maps to the natural map $\eta\colon M \rtarr \OM B M$ for some topological monoid $M$.
\end{defn}

The following classical result has several proofs; see \cite[\S 15]{MayClass} for discussion in slightly greater generality.

\begin{thm}\mylabel{gpcom} If a topological monoid $M$ is homotopy commutative, then the natural map $\et\colon M\rtarr \OM B M$ 
is a group completion.
\end{thm}

Returning to our equivariant context, we have the following definition.

\begin{defn} A Hopf $G$-space $Y$ is grouplike if each $\pi_0(Y^H)$ is a group.  Let $X$ and $Y$ be homotopy associative and 
commutative Hopf $G$-spaces, where $Y$ is grouplike, and let $f\colon X\rtarr Y$ be a Hopf $G$-map.  
Then $f$ is a group completion if $f^H\colon X^H\rtarr Y^H$ is a group completion for all subgroups $H$ of $G$.  
\end{defn}

For the equivariant notion of weak group completion, note that if $X$ is a $\mathbf{K}_\bR$-$G$-space and 
$H\subset G$ is a subgroup, then $X^H$ inherits an action of $\mathbf{K}_\bR$.

\begin{defn} A map $f\colon X \rtarr Y$ of $\mathbf{K}_\bR$-$G$-spaces is a weak group completion if $f^H$ is a weak group completion for all $H$.
By \myref{gpcom}, $f$ is then a group completion if $X$ and $Y$ are homotopy commutative. 
\end{defn}

In the weak case we require no compatibility  between the monoids $M(H)\htp X^H$ as $H$ varies.
Recall that we understand equivalences of $G$-spaces to mean maps that induce (weak) equivalences
on passage to fixed points and observe that a group completion is an equivalence if $X$ is grouplike, 
for example if $X$ is $G$-connected in the sense that each $X^H$ is (path) connected.

\begin{thm}[The approximation theorem]\mylabel{approx}  Let $V$ be a representation of $G$.  If $X$ is $G$-connected,
then $\al_V:\mathbf{K}_V X \rtarr \OM^V \SI^V X$ is an equivalence.  If $V$ contains a copy of the trivial representation $\bR$, then $\al_V$ is a 
weak group completion.  Therefore, if $V$ contains a copy of $\bR^2$, then $\al_V$ is a group completion.
\end{thm}

We shall not give a proof, only a commentary on the existing proofs.   The group completion version was 
first proven by Hauschild in his unpublished Habilitationschrift \cite{Haus1}, but the shorter published 
version \cite{Haus2} restricts to the case $X= S^0$, remarking that the proof in the general case is essentially the
same.  Assuming that $V$ contains $\bR^{\infty}$ and not just $\bR^2$, Caruso and Waner \cite[1.18]{CarWaner} 
gave a shorter proof in a paper that concentrated on compact Lie groups $G$, rather than just finite groups.

Nonequivariantly, there is a proof by direct calculation due to Fred Cohen \cite{LMS} and a geometric
proof due to Segal \cite{SegCon}.   Starting from Segal's proof, Rourke and Sanderson 
\cite{RourSan1a, RourSan1b, RourSan1c} gave an 
elegant proof using their ``compression theorem''.   Following up a suggestion of May, they generalized that 
proof to give the stated version of the theorem in \cite{RourSan2}.  However, their
notations are quite different from ours.   They never work equivariantly and focus instead on
$G$-fixed point spaces.  They use the notation $\OM^V\SI^V X$ for the $G$-fixed point space
$(\OM^V\SI^V X)^G$.  One can replace $G$ by a subgroup $H$ in their proof, and it works just as well.

All known proofs are manifold theoretic in nature and start with the $G$-space $\BF_VX$ of (unordered) configurations 
of points in $V$ with labels in $X$.  More precisely, $\BF_VX = \coprod F(V,j) \times_{\SI_j} X^j/ (\sim)$ is defined in the same
way as $\BK_V X$.  In the notation of \cite{RourSan2}, their $C_V X$ is our $(\BF_VX)^G$. They work with little disks,
and their $C_V^o X$ is our $(\BD_VX)^G$.   Their map $j_V$ is the restriction to $\BF_V(X)^G$ of our map $\al_V^G$.  

Translated to our notations, \cite[Theorem 1]{RourSan2} proves the first statement of \myref{approx}, taking $X$ to be
$G$-connected; here there are no Hopf $G$-space structures in sight.
When $W = V\oplus \bR$, Rourke and Sanderson observe that $(\BD_WX)^G$ is equivalent to a monoid, and their \cite[Theorem 2]{RourSan2} proves that its classifying space is weak homotopy equivalent to $(\OM^V\SI^{W}X)^G$.  The approximation theorem as stated follows by applying $\OM$, as in \cite[Corollary 1]{RourSan2}.  

\subsection{The recognition principle for $V$-fold loop spaces}\label{sec:RecogPrin}

We explain how $K_V$-spaces, which are based spaces with an action of $\sK_V$,  give rise
to $V$-fold loop spaces.  For fixed $V$, we can work equally well with $\sD_V$.  For compatibility
as $V$ varies, $\sK_V$ is required.  The two-sided monadic bar construction is described in \cite{MayGeo, Rant1} 
and works exactly the same way equivariantly as nonequivariantly.\footnote{In particular, Reedy cofibrancy 
(or properness) works the same way; see \cite{MMO}.}  The adjoint of $\al_V$ gives a right action
$\tilde{\al}_V\colon \SI^V\BK_V \rtarr \SI^V$ of the monad $\BK_V$ on the functor $\SI^V$. 

\begin{defn} Let $Y$ be a $\sK_V$-space.  We define 
$$\bE_V Y = B(\SI^V,\BK_V, Y).$$ 
\end{defn}

We have the diagram of $\sK_V$-spaces and $\sK_V$-maps
\begin{equation}\label{myold}
\xymatrix@1{ Y & \ar[l]_-{\epz} \ar[r]^-{\overline{\al}_V} 
B(\mathbf{K}_V,\mathbf{K}_V,Y) & 
B(\OM^V\SI^V,\mathbf{K}_V,Y) \ar[r]^-{\zeta} & 
\OM^{V} B(\SI^{V},\mathbf{K}_V,Y), \\}
\end{equation}
where $\overline{\al}_V = B(\al_V,\id,\id)$ and $\zeta$ will be defined in the following sketch proof, 
which is based on arguments in \cite{CostWan, MayGeo, MayPerm}.  

\begin{thm}[From $\sK_V$-spaces to $V$-fold loop spaces]\mylabel{Vrec}  The following statements 
hold relating a $\sK_V$-space $Y$ to its ``$V$-fold delooping'' $\bE_V Y$.
\begin{enumerate}[(i)]
\item The map $\epz$ is a $G$-homotopy equivalence with a natural homotopy inverse $\nu$.
\item The map $\overline{\al}_V $ is an equivalence when $Y$ is $G$-connected and is a weak
group completion when $V\supset \bR$.
\item The map $\ze$ is an equivalence.
\end{enumerate}
Therefore the composite 
\begin{equation}\label{xi}
\xi = \ze\com \overline{\al}_V \com \nu\colon Y\rtarr \OM^V \bE_V Y
\end{equation}
is an equivalence if $Y$ is $G$-connected, a weak group completion if $V\supset \bR$, and a group completion if $V\supset \bR^2$.
\end{thm}
\begin{proof} The proof of (i) uses an ``extra degeneracy argument''  explained in \cite[9.8]{MayGeo};
note that the homotopy equivalence $\nu$ is not a $\sK_V$-map.   For (ii), it is shown nonequivariantly in  \cite[Theorem 2.3]{MayPerm}, 
that $\overline{\al}_V$ is an equivalence when $Y$ is connected and is a group completion when $V = \bR^n$ with $n\geq 2$.  
We use \myref{approx} to improve on that equivariantly. Geometric realization of simplicial $G$-spaces commutes with 
passage to $H$-fixed points, so we can work nonequivariantly, one fixed point space at a time.  If $Y$ is $G$-connected, 
each $(\bK_V^q Y)^H$ is connected, 
hence $\overline{\al}^H$ is the realization of a levelwise equivalence of simplicial spaces and hence an equivalence.  

Now assume $V\supset \bR$ and let $\sK = \sK_{\bR}$, with associated monad $\BK$.  We then have an inclusion of the 
nonequivariant $A_{\infty}$ operad $\sK$ in $\sK_V$ and can regard $Y$ and each $(\bK_V^q Y)^H$ as a $\sK$-space. 
From here we combine arguments from  \cite[Section 13]{MayGeo} and the proof of \cite[Theorem 2.3]{MayPerm} with
the Rourke-Sanderson proof of the approximation theorem.
Let $\sM$ be the associativity operad that defines monoids; we have a weak equivalence of ($G$-fixed) operads $\de\colon \sK\rtarr \sM$.  
For a $\sK$-space $X$, we define a topological monoid  $\LA(X) = B(\BM, \BK, X)$, where the monad $\BM$ is a $\BK$-functor
via $\de$.  We have a  zigzag
\[   
\xymatrix@1{  X & \ar[l]_-{\epz} \ar[r]^-{\overline{\de}} B(\BK, \BK, X)  & B(\BM, \BK, X) = \LA X} \]
in which $\epz$ is a $\sK$-map and a $G$-homotopy equivalence 
and $\overline{\de} = B(\de,\id,\id)$ is an equivalence.  Define $\GA(X) = \OM B \LA (X)$ and $\ga = \et\com \overline{\de} \colon B(\mathbf{K},\mathbf{K},X) \rtarr \GA X$.
We view $\ga$ as a natural choice of a weak group completion. Moreover, $\ga$ is an equivalence if $X$ is grouplike.
If $f\colon X\rtarr Y$ is a weak group completion between $\sK$-spaces, then $\GA f$ is an equivalence.
To see this, note that by the definition of weak group completion, we may assume without loss of generality that $f$ is the map $\eta\colon M \rtarr \OM B M$ for some topological monoid $M$. It suffices to show that $B\LA(\eta):B \LA M \rtarr B \LA (\OM B M)$ is an equivalence. This follows from \cite[Proposition~3.9 and Theorem~3.11]{Thmsn}.

Now consider the following commutative diagram.  
\[ \xymatrix{
Y  &  \ar[l]_-{\epz} B(\BK_V,\BK_V,Y) \ar[r]^-{\overline{\al}_V}  & B(\OM^V\SI^V, \BK_V, Y) \\
B(\mathbf{K},\mathbf{K},Y) \ar[d]_{\ga} \ar[u]^{\epz} & \ar[l]_-{\epz} B( B(\mathbf{K},\mathbf{K},\mathbf{K}_V),\mathbf{K}_V,Y) \ar[r]^-{\overline{\al}_V} \ar[d]^{\overline{\ga}} \ar[u]^{\epz} & B(B(\mathbf{K},\mathbf{K},\OM^V \SI^V),\mathbf{K}_V, Y) \ar[u]^\epz \ar[d]^{\overline{\ga}}\\
\GA Y  &      \ar[l]_-{\epz} B(\GA \BK_V,\BK_V,Y) \ar[r]^-{\overline{\GA\al}_V}  & B(\GA \OM^V\SI^V, \BK_V, Y) \\}   \]
The maps $\epz$ are $G$-homotopy equivalences, hence the middle map $\overline{\ga} = B(\ga,\id,\id)$ is a weak group
completion since $\ga$ is so.  The right map $\overline{\ga}$ and the bottom map $\overline{\GA\al}_V$ are equivalences
since realization preserves levelwise equivalences.   Therefore $\overline{\al}_V$ is a weak group completion. 

In (iii), $\ze$ is an instance of the natural $G$-map $\ze\colon |\OM^V K| \rtarr \OM^V|K|$ for simplicial based $G$-spaces $K$;
suspensions commute with realization, and the adjoint of $\ze$ is the evident evaluation $G$-map $\SI^V |\OM^V K|\iso |\SI^V\OM^V K|  \rtarr |K|$.   The proof of (iii) is due to Hauschild \cite{Haus1} and appears in \cite[pp. 495-496]{CostWan}.  We will not repeat the argument, which reduces
the proof to the nonequivariant case treated in  \cite[\S12]{MayGeo}.  The main equivariant input that allows the reduction is the fact if $S(V)$ 
is the unit sphere in $V$, then the space $Map_H(S(V),K_n)$ of $H$-maps is connected, where $K_n = \SI^V\BK_V^n Y$ is the $G$-space of 
$n$-simplices of the simplicial $G$-space $B_*(\SI^V,\BK_V,Y)$.  This holds since $K_n^J$ is $(dim(V^J) -1)$-connected for each subgroup $J\subset G$, while $S(V)$ regarded as an $H$-CW complex only has cells of type $H/J \times e^n$ where $n<dim(V^J)$.
\end{proof}

\begin{rem} Equivariant homotopy theory often admits varying generalizations of nonequivariant theorems.  A very
different and very interesting equivariant recognition principle was proven by Salvatore and Wahl \cite{SW}.
\end{rem}

\subsection{The pairing $(\sK_V, \sK_W) \rtarr \sK_{V\oplus W}$ and the recognition principle}\label{SecSteinerPair}

The general notion of a pairing of operads is recalled in \S\ref{SecPair}.  In \cite[8.3]{MayGeo}, a pairing 
\[ \boxtimes\colon \BC_mX\sma  \BC_nY \rtarr \BC_{m+n}(X\sma Y) \] 
is defined for based spaces $X$ and $Y$, where $\BC_n$ denotes the monad on based spaces
induced from the little $n$-cubes operad $\sC_n$.  Implicitly,
it comes from a pairing of operads $\boxtimes\colon (\sC_m, \sC_n)\rtarr \sC_{m+n}$.
The Steiner operad analogue appears in \cite[p. 337]{MayPair}, and we recall it here. 

\begin{prop}\mylabel{SteinPair} For finite dimensional real inner product $G$-spaces $V$ and $W$,
there is a unital, associative, and commutative system of pairings 
$$\boxtimes\colon (\sK_V, \sK_W) \rtarr \sK_{V\oplus W}$$
of Steiner operads of $G$-spaces.
\end{prop}
\begin{proof}
The required maps
\[ \boxtimes\colon \sK_V(j) \times \sK_W(k) \rtarr \sK_{V\oplus W}(jk) \]
are given by $(c\otimes d) = e$, where, writing
$c=(f_1,\cdots,f_j)$ and $d=(g_1,\cdots,g_k)$, $e$ is the $jk$-tuple
of Steiner paths 
\[ (f_q, g_r)\colon I\rtarr R_V\times R_W \subset R_{V\oplus W}, \]
$1\leq q\leq j$ and $1\leq r\leq k$, ordered lexicographically.  The formulas required in 
\myref{pairop} are easily verified, as we illustrate in \myref{pairopex}. 

The pairing is unital in the sense that
$\boxtimes\colon \sK_V(j) \iso \sK_0(1)\times \sK_V(j)\rtarr \sK_V(j)$ is the identity map.
It is associative in the sense that the following diagram commutes for a triple $(V,W,Z)$ of inner product $G$-spaces
and a triple $(i,j,k)$. 
\[ \xymatrix{
\sK_V(i) \times \sK_W(j)\times \sK_Z(k) \ar[r]^-{\boxtimes\times \id}
\ar[d]_{\id\times \boxtimes} & \sK_{V\oplus W}(ij)\times \sK_Z(k) \ar[d]^{\boxtimes}\\
\sK_V(i)\times \sK_{W\oplus Z}(jk) \ar[r]_-{\boxtimes} & \sK_{V\oplus W\oplus Z}(ijk) \\}
\]
It is commutative in the sense that the following diagram commutes. 
\[ \xymatrix{
\sK_V(j) \times \sK_W(k) \ar[d]_{t} \ar[r]^-{\boxtimes} & \sK_{V\oplus W}(jk) \ar[d]^{\ta(j,k)}\\
\sK_W(k)\times \sK_V(j) \ar[r]_-{\boxtimes} & \sK_{W\oplus V}(kj)\\} \]
Here $t$ is the interchange map and $\ta(j,k)$ is determined in an evident way by the interchange map for $V$ and $W$ and the 
permutation $\ta(j,k)$ of $jk$-letters.
\end{proof}

Passing to monads as in \myref{pairsm} below,  we 
obtain a unital, associative, and commutative system of pairings
\begin{equation}\label{StMonPair}
 \boxtimes\colon  \mathbf{K}_VX\sma \mathbf{K}_WY
\rtarr \mathbf{K}_{V\oplus W}(X\sma Y).
\end{equation}
For the unit property, when $V=0$ the map 
$\boxtimes\colon X\sma  \mathbf{K}_WY\rtarr \mathbf{K}(X\sma Y)$ 
is induced by the maps $X\times Y^j\rtarr (X\times Y)^j$ obtained from the
diagonal map on $X$ and shuffling. We have the following key observation. 
Its analogue for the little cubes operads is \cite[8.3]{MayGeo}.

\begin{lem}\mylabel{alsma}  The following diagram commutes.
\[ \xymatrix{
\mathbf{K}_V X \sma \mathbf{K}_W Y \ar[r]^-{\boxtimes} \ar[d]_{\al_V\sma \al_W} 
& \mathbf{K}_{V\oplus W} (X\sma Y) \ar[d]^{\al_{V\oplus W}}\\
\OM^V\SI^V X \sma \OM^W\SI^W Y \ar[r]_-{\sma} & 
\OM^{V\oplus W}\SI^{V\oplus W} (X\sma Y)\\} \]
\end{lem}

{The notion of a pairing of a $\sK_V$-space $X$ and a $\sK_W$-space $Y$
to a $\sK_{V\oplus W}$-space $Z$ is defined in \myref{MayPair}, and we have the following 
recognition principle for pairings.  Note that smashing maps out of spheres 
gives a natural map
$$\OM^V X \sma \OM^W Y \rtarr \OM^{V\oplus W} (X\sma Y).$$

\begin{prop}\mylabel{PairVrec}
A pairing $f\colon X\sma Y\rtarr Z$ of a $\sK_V$-space $X$ and a $\sK_W$-space $Y$
to a $\sK_{V\oplus W}$-space $Z$ induces a $G$-map 
\[ \bE f\colon \bE_V X \sma \bE_W Y \rtarr \bE_{V\oplus W} Z \]
such that the following diagram commutes.
\[
\xymatrix{
X\sma Y \ar[r]^-{\xi\sma \xi}  \ar[d]_f &   \OM^V \bE_V X \sma \OM^W \bE_W Y \ar[r] & \OM^{V\oplus W}(\bE_V X \sma \bE_W Y) \ar[d]^{\bE f} \\
Z \ar[rr]_-{\xi} & & \OM^{V\oplus W} \bE_{V\oplus W} Z
\\} \] 
\end{prop}
\begin{proof}   By convention, $\BK_V^0 = \Id$ for any $V$.  Starting at $q=0$ with the identity map on $X\sma Y$,
$\boxtimes$ inductively determines a  pairing $\boxtimes^q$ for all $q$, namely the composite
\[   \xymatrix@1{\BK^q_V X \sma \BK^q_W Y \ar[r]^-{\boxtimes} & \BK_{V\oplus W}(\BK^{q-1}_V X \sma \BK^{q-1}_W Y) 
\ar[rr]^-{\BK_{V\oplus W} \boxtimes^{q-1}}  & & \BK_{V\oplus W}^q (X\sma Y).\\} 
\]
The map $\bE f$ is the geometric realization of a map of simplicial topological spaces that is given on $q$-simplices by
\[
 \xymatrix@1{\SI^V \BK_V^q X \sma \SI^W \BK_W^q Y  \iso  \SI^{V\oplus W} (\BK_V^q X \sma  \BK_W^q Y) 
\ar[rr]^-{\SI^{V\oplus W} \boxtimes^q} & &  \SI^{V\oplus W} \BK_{V\oplus W}^q (X\sma Y). \\} 
\]
Commutation with face and degeneracy operators follows from \myref{pairsm}. The diagram in the statement
commutes by a diagram chase from \myref{alsma}, \myref{MayPair}, and the description of $\xi$ given in (\ref{xi}).
\end{proof}

We have an unstable precursor of the BPQ theorem. 

\begin{thm}[The BPQ theorem for $V$-fold suspensions]\mylabel{unstableBPQ} For based $G$-spaces $X$,  there is a natural $G$-homotopy equivalence 
$$\om \colon \SI^V X \rtarr \bE_V \BK_V X$$ 
such that the following diagram commutes for based $G$-spaces $X$ and $Y$.
\[
\xymatrix{
\SI^V X \sma \SI^W Y \ar[r]^-{\om\sma\om} \ar[d]_{\iso} & \bE_V \BK_V X \sma \bE_W \BK_W Y \ar[d]^-{\bE (\boxtimes)}\\
\SI^{V\oplus W}(X\sma Y) \ar[r]_-{\om}  & \bE_{V\oplus W}\BK_{V\oplus W}(X\sma Y) \\}
\]
Therefore $\bE(\boxtimes)$ is an equivalence.
\end{thm}
\begin{proof} 
Since $\bE_V\BK_V X = B(\SI^V, \BK_V, \BK_V X)$, another extra degeneracy argument explained in 
\cite[9.8]{MayGeo} gives the natural homotopy equivalence $\om$.  For 
the diagram, it suffices to prove commutativity of the adjoint diagram, which features two adjoint maps  
$X\sma Y \rtarr \OM^{V\oplus W}(-)$.
These maps are equal by inspection of definitions.
\end{proof}

\subsection{The geometric recognition principle for orthogonal $G$-spectra}\label{GeoRec}
As in \cite{GM2}, we let $G\sS$ denote the category of orthogonal $G$-spectra.  Briefly, these start with
$\sI_G$-spaces $E$, which are 
continuous functors $E\colon \sI_G\rtarr \sT_G$, where $\sI_G$ is the category of finite dimensional $G$-inner
product spaces and linear isometric isomorphisms, with $G$ acting by conjugation on morphism spaces
$\sI_G(V,V')$.  The continuous $G$-maps $E\colon \sI_G(V,V')\rtarr \sT_G(E(V),E(V'))$ can be specified via
adjoint evaluation $G$-maps  $\sI_G(V,V')_+\sma E(V)\rtarr E(V')$. 

An $\sI_G$-space $E$ is an orthogonal 
$G$-spectrum if there are structure $G$-maps $\SI^W E(V) \rtarr E(V\oplus W)$ 
that give a natural transformation $E\barwedge S_G\rtarr E\com \oplus$ of functors $\sI_G\times \sI_G \rtarr \sT_G$, 
where $S_G = \{S^V\}$ is the sphere $G$-spectrum, $\barwedge$ is the external smash product specified by 
$(D\barwedge E)(V,W) = D(V) \sma E(W)$ for $\sI_G$-spaces $D$ and $E$, and 
$\oplus\colon  \sI_G\times \sI_G \rtarr \sI_G$ is the direct sum of $G$-inner product spaces functor.  See \cite[II\S2]{MM} for details.  

\begin{defn}\mylabel{KIF}  We define a continuous $G$-functor $\sK_*$ from $\sI_G$ to $G$-operads.  It takes a $G$-inner 
product space $V$ to the Steiner operad $\sK_V$.   Linear isometric isomorphisms $i\colon V\rtarr V'$ act
by conjugation of embeddings to send $R_V$ to $R_{V'}$. The action extends pointwise to Steiner paths and 
then applies one at a time to $j$-tuples of Steiner paths to give $G$-maps $\sK_V(j)$ to $\sK_{V'}(j)$.  Compatibility with 
the operad structure is immediate.  Composing with the functor that sends the operad $\sK_V$ to the associated 
monad $\BK_V$ on based $G$-spaces gives a functor $\BK$ from $\sI_G$ to the category of monads in the
category of $\sI_G$-spaces.  In more detail, for an $\sI_G$-space $\sX$ with $V$th space $\sX(V)$, we have based 
evaluation $G$-maps 
$$\sI(V,V')_+\sma \BK_V\sX(V)\rtarr \BK_{V'}\sX(V').$$
Using the diagonal action of $\sI_G(V,V')$, we obtain $G$-maps
$$
\xymatrix{ \sI_G(V,V') \times \sK_V(k)\times \sK_V(j_1)\times \cdots \times \sK_V(j_k) \times \sX(V) \ar[d] \\
\sK_{V'}(k)\times \sK_{V'}(j_1)\times \cdots \times \sK_{V'}(j_k)\times \sX(V')\\} $$
and these give evaluation $G$-maps
$$\sI_G(V,V')_+\sma \BK_V\BK_V \sX(V) \rtarr  \BK_{V'}\BK_{V'} \sX(V'). $$
The product and unit maps are compatible with these maps in the sense that the following diagrams commute,
where the unlabelled arrows are evaluation $G$-maps.

\small{
\begin{equation} 
\xymatrix{
\sI_G(V,V')_+ \sma  \sX(V) \ar[d]_{\id\sma \et}  \ar[r]^-{\pi} & \sX(V') \ar[d]^{\et} \\
\sI_G(V,V')_+ \sma \BK_V \sX(V) \ar[r] & \BK_{V'} \sX(V')  } 
\  
\xymatrix{
\sI_G(V,V')_+ \sma \BK_V \BK_V \sX(V) \ar[r] \ar[d]_{\id\sma \mu} & \BK_{V'} \BK_{V'} \sX(V) \ar[d]^{\mu}  \\
\sI_G(V,V')_+ \sma \BK_V \sX \ar[r] & \BK_{V'} \sX(V')  \\} 
\end{equation}
}

\noindent
Note that we can regard based $G$-spaces $X$ as constant $\sI_G$-spaces, $X(V) = X$; the evaluation $G$-maps 
$\sI_G(V,V')_+\sma X\rtarr X$ are then the projections.
\end{defn}

\begin{defn}\mylabel{KStar}  Define a $\sK_*$-$G$-space $\sY$ to be an $\sI_G$-space $\sY$ with a structure
of $\sK_V$-algebra on  $\sY(V)$ for each $V$
together with $G$-maps $i\colon \sY(V) \rtarr \sY(V\oplus W)$ such that the following diagrams commute,
where the $\tha$ are monad action maps. 
\[
\xymatrix{
\sI_G(V,V')_+\sma \BK_V\sY(V)\ar[r] \ar[d]_{\id\sma \tha} & \BK_{V'}\sY(V') \ar[d]^{\tha}\\
\sI_G(V,V')_+\sma \sY(V)\ar[r] & \sY(V') \\}
\]
In the second diagram, we identify $S^V\sma S^W$ with $S^{V\oplus W}$.
\[  
\xymatrix{ 
\big{(}\sI_G(V,V')\times \sI_G(W,W')\big{)}_+ \sma \sY(V)\sma S^V\sma S^W \ar[r] \ar[d]_{\oplus\sma i\sma \id}
& \sY(V')\sma S^{V'}\sma S^{W'} \ar[d]^{i \sma \id}\\
\big{(}\sI_G(V\oplus W, V'\oplus W')\big{)}_+ \sma \sY(V\oplus W)\sma S^{V\oplus W} 
\ar[r] &  \sY(V'\oplus W')\sma S^{V'\oplus W'} 
\\}
\]
The first diagram says that $\tha$ is a map of $\sI_G$-spaces and, ignoring the sphere coordinates, the second diagram says 
that $i\colon \sY\com \pi_1\Longrightarrow \sY\com \oplus$ is a natural transformation of functors  
$\sI_G\times \sI_G \rtarr \sT_G$. 
 \end{defn} 
  
 \begin{thm}[From $\sK_*$-$G$-spaces to orthogonal $G$-spectra]\mylabel{KStarrec}  For a $\sK_*$-$G$-space $\sY$, the 
 based $G$-spaces $\bE_V\sY(V)$ and the based $G$-maps 
 $$\SI^W\bE_V\sY(V) \rtarr \bE_{V\oplus W}\sY(V\oplus W)$$
 determined by  $i\colon \sY\com \pi_1\Longrightarrow \sY\com \oplus$ specify an orthogonal $G$-spectrum
 $\bE_G^{geo}\sY$. 
 \end{thm}
 \begin{proof} 
 Regarding the $\sI_G(V,V')$ as constant simplicial $G$-spaces, we see by diagram chases from
 the definitions that the data of the previous definitions determine $G$-maps
 \[ \sI_G(V,V')_+\sma \SI^V\BK_V^q\sY(V) \rtarr \SI^{V'} \BK_{V'}^q\sY(V') \]
 and 
 \[  \SI^W  \SI^V\BK_V^q\sY(V) \rtarr \SI^{V\oplus W}\BK_{V\oplus W}^q\sY(V\oplus W). \]
 On passage to geometric realization, these give the required $\sI_G$-space  $\bE_G^{geo}\sY$ and
 the required natural transformation $\bE_G^{geo}\sY\barwedge S_G \rtarr \bE_G^{geo}\sY\com \oplus$. 
 \end{proof}

Of course, the recognition principle of (\ref{myold}) and \myref{Vrec} applies to describe the
relationship between the $G$-spaces $\sY(V)$ and $\OM^V(\bE_G^{geo}\sY)(V)$.  The recognition
principle for pairings also adapts directly. 

\begin{defn}\mylabel{KFP}  Let $\sX$, $\sY$, and $\sZ$ be $\sK_*$-$G$-spaces.  A pairing 
$$f\colon \sX\barwedge \sY \rtarr \sZ\com \oplus$$
is a natural transformation of continuous functors 
$\sI_G \times \sI_G \rtarr \sT_G$ such that each $f\colon \sX(V)\sma \sY(W) \rtarr \sZ(V\oplus W)$ 
is a pairing as in \myref{MayPair} and the following diagram commutes for all $U,V,W$. 
\begin{equation}\label{pairdia}
\xymatrix{
\sX(U)\sma \sY(V)  \ar[rr]^-{\id\sma i}  \ar[dd]^{i\sma \id} \ar[dr]^f
& & \sX(U) \sma \sY(V\oplus W) \ar[dd]^f \\
& \sZ(U\oplus V) \ar[dr]^i & \\
\sX(U\oplus W) \sma \sY(V) \ar[r]_-{f} & \sZ(U\oplus W \oplus V) \ar[r]_{\sZ(\id\oplus t)} & \sZ(U\oplus V\oplus W)\\}
\end{equation}
This diagram expresses that the three composite natural transformations 
of functors $\sI_G^3\rtarr \sT_G$ in sight agree.
\end{defn}

The smash product of orthogonal $G$-spectra is obtained by first applying Day convolution to the
external smash product $\barwedge$ and then coequalizing the action of the sphere $G$-spectrum
on the two variables.  See \cite[II\S3]{MM} for details.

\begin{prop}  A pairing $f\colon \sX\barwedge \sY \rtarr \sZ\com \oplus$  of $\sK_*$-$G$-spaces induces a map
$$\bE_G^{geo} f\colon \bE_G^{geo}\sX\sma \bE_G^{geo}\sY \rtarr \bE_G^{geo}\sZ$$
of orthogonal $G$-spectra that is given levelwise by specialization of \myref{PairVrec}.
\end{prop}
\begin{proof}
The definition of a pairing immediately implies that $f$ induces an external pairing
$$\bE_G^{geo}\sX \barwedge \bE_G^{geo}\sY \rtarr \bE_G^{geo}\sZ\com \oplus, $$
and the diagram (\ref{pairdia}) ensures that the resulting map from the Day convolution to $\bE_G^{geo}\sZ$ 
factors through the coequalizer defining $\bE_G^{geo}\sX\sma \bE_G^{geo}\sY$.
\end{proof}

The suspension $G$-spectrum $\SI^{\infty}_G X$ of a based $G$-space $X$ is given by the
$G$-spaces $\SI^V X$; its structure maps are the evident identifications $\SI^W\SI^V X\iso \SI^{V\oplus W}X$.
The unstable BPQ theorem of \myref{unstableBPQ} leads to the following ``geometric'' version
of the BPQ theorem.

\begin{defn}  For a based $G$-space $X$, define $\BK_* X$ to be the $\sK_*$-$G$-space given
by the $\sK_V$-spaces $\BK_V X$ and the maps $i\colon \BK_V X \rtarr \BK_{V\oplus W} X$
induced by the map of operads $\sK_V \rtarr \sK_{V\oplus W}$ obtained by sending 
embeddings $e\colon V\rtarr V$ to $e\times \id\colon V\times W\rtarr V\times W$.
\end{defn}

It is easily verified that $\BK_*X$ is a $\sK_*$-$G$-space and the pairings $\boxtimes$ 
of (\ref{StMonPair}) prescribe pairings
\begin{equation}\label{StMonPair2}
\boxtimes\colon \bK_* X \barwedge \bK_* Y \rtarr \bK_* (X\sma Y)\com \oplus.
\end{equation}

\begin{thm}[The geometric BPQ theorem for orthogonal suspension $G$-spectra]\mylabel{BPQ2} 
For based $G$-spaces $X$, there is a natural equivalence 
\[ \om\colon \SI^{\infty}_G X \rtarr \bE_G^{geo}\BK_* X \]
such that the following diagram commutes for based $G$-spaces $X$ and $Y$.
\[
\xymatrix{
\SI^{\infty}_G X \sma \SI^{\infty}_G Y \ar[d]_{\iso} \ar[r]^-{\om\sma \om} 
&  \bE_G^{geo}\BK_* X \sma \bE_G^{geo}\BK_* Y \ar[d]^{\bE_G^{geo}(\boxtimes)}\\
\SI^{\infty}_G(X\sma Y) \ar[r]_{\om} & \bE_G^{geo}\BK_* (X\sma Y)\\}
\]
\end{thm}
\begin{proof}
The levelwise equivalence follows from \myref{unstableBPQ}. For the diagram,
the functor $\SI^{\infty}_G$ is left adjoint to the $0$th $G$-space functor, and inspection
of definitions shows that the adjoint diagram starting with $X\sma Y$ commutes. 
\end{proof}

\subsection{A configuration space model for free $\sK_V$-spaces}\label{sec:ConfigModel}

The free $\sK_V$-spaces  $\bK_V X$ can be modelled more geometrically by configuration
spaces.  To explain this, we first record the nonequivariant analogue in terms of the little cubes
operads, since that is relevant folklore which is not in the literature.  

Consider the little 
$n$-cubes operads $\sC_n$ and their associated monads $\mathbf{C}_n$. Let $J=(0,1)$ be
 the interior of $I$. We have the configuration spaces $F(J^n,j)$ of $j$-tuples of distinct points 
 in $J^n$. Sending little $n$-cubes $c\colon J^n\rtarr J^n$ to their center points $c(1/2,\cdots,1/2)$ 
gives a $\SI_n$-homotopy equivalence $f\colon \sC_n(j)\rtarr F(J^n,j)$.  

For based spaces $X$, we construct spaces $\mathbf{F}_n X$ by replacing $\sC_n(j)$ by $F(J^n,j)$ 
in the construction of $\mathbf{C}_n X$ as the quotient of $\amalg  \sC_n(j)\times_{\SI_j} X^j$
by basepoint identifications; we now use the evident omit a point projections $F(j,n)\rtarr  F(j,n-1)$
rather than the analogous maps $\sC_n(j)\rtarr \sC_n(j-1)$. The maps $f$ induce a homotopy equivalence 
$$ f\colon \mathbf{C}_n X\rtarr \mathbf{F}_nX.$$ 
That much has been known since \cite{MayGeo}.  

The folklore observation is that although the $F(J^n,j)$ do not form an operad, $\sC_n$ acts on 
$\mathbf{F}_n X$ in such a way that $f$ is a map of $\sC_n$-spaces.  Indeed, we can evaluate 
little $n$-cubes $J^n\rtarr J^n$ on points of $J^n$ to obtain maps 
$$\sC_n(j)\times F(J^n,j)\rtarr F(J^n,j),$$
and any reader of \cite{MayGeo} will see how to proceed from there.  Moreover, we have pairings 
$$ \boxtimes\colon \BF_mX \sma \BF_n Y \rtarr \BF_{m+n} (X\sma Y)  $$
defined as in  \myref{MayPair} and  \myref{pairsm}, starting from the maps
$$  F(J^n, j) \times F(J^n,k) \rtarr F(J^n,jk) $$
that send $(x,y)$, $x = (x_1,\cdots, x_j)$ and $y = (y_1,\cdots, y_k)$ to the set of pairs 
$(x_q,y_r)$, $1\leq q\leq j$ and $1\leq r\leq k$, ordered lexicographically. 

Nonequivariantly, we put this together to obtain an analogue of \myref{BPQ2}, using the
evident variant of the geometric recognition principle that is obtained from the operads $\sC_n$ as $n$ varies.
Here it is more natural to use symmetric spectra rather than orthogonal spectra, since
it is natural to deal with sequences rather than inner product spaces.  The relationship between the little cubes
operads and symmetric spectra is explained in \cite[\S I.8]{MM},  and we leave
details of the relevant retooling of the previous subsections to the interested reader.

\begin{thm}[The configuration space BPQ theorem for symmetric spectra]\mylabel{BPQ2too} 
For based spaces $X$, there is a natural equivalence 
\[ \om\colon \SI^{\infty} X \rtarr \bE^{geo}\BF_* X \]
such that the following diagram commutes for based spaces $X$ and $Y$.
\[
\xymatrix{
\SI^{\infty}X \sma \SI^{\infty} Y \ar[d]_{\iso} \ar[r]^-{\om\sma \om} 
&  \bE^{geo}\BF_* X \sma \bE^{geo}\BF_* Y \ar[d]^{\bE_G^{geo}(\boxtimes)}\\
\SI^{\infty}(X\sma Y) \ar[r]_{\om} & \bE^{geo}\BF_* (X\sma Y)\\}
\]
\end{thm}
 
For fixed $V$, the discussion generalizes equivariantly to relate $\BD_V  X$ or $\BK_V X$ to $\BF_V X$ for based 
$G$-spaces $X$.  In the case of $\BK_V X$,  we use the time $0$ projections from Steiner paths to embeddings 
$V\rtarr V$ and the centerpoint map from $\mathrm{Emb}_V(j)$ to $F(V,j)$. Letting $V$ vary, we obtain
the following equivariant version of \myref{BPQ2too}.

\begin{thm}[The configuration space BPQ theorem for orthogonal $G$-spectra]\mylabel{BPQ2tri} 
For based $G$-spaces $X$, there is a natural equivalence 
\[ \om\colon \SI^{\infty}_G X \rtarr \bE_G^{geo}\BF_* X \]
such that the following diagram commutes for based $G$-spaces $X$ and $Y$.
\[
\xymatrix{
\SI^{\infty}_G X \sma \SI^{\infty}_G Y \ar[d]_{\iso} \ar[r]^-{\om\sma \om} 
&  \bE_G^{geo}\BF_* X \sma \bE_G^{geo}\BF_* Y \ar[d]^{\bE_G^{geo}(\boxtimes)}\\
\SI^{\infty}_G(X\sma Y) \ar[r]_{\om} & \bE_G^{geo}\BF_* (X\sma Y)\\}
\]
\end{thm}


\section{The recognition principle for infinite loop $G$-spaces}\label{Sec5}

The equivariant recognition principle shows how to recognize (genuine) $G$-spectra
in terms of category or space level information.  It comes in various versions. We shall
give two modernized variants of the machine from \cite{MayGeo}, differing in their 
choice of the output category of $G$-spectra.  In contrast with the previous section,
we are now concerned with infinite loop space machines with input given by  
$E_{\infty}$ $G$-spaces (or $G$-categories) defined over any (genuine) $E_{\infty}$ operad.
A $G$-spectrum $E$ is connective if the negative homotopy groups of each of its fixed
point spectra $E^H$ are zero, and all infinite loop space machines take values in
connective $G$-spectra.

As in \cite{GM2}, we let $\sS$, $\sS\!p$, and $\sZ$ denote the categories of orthogonal
spectra \cite{MMSS}, Lewis-May spectra \cite{LMS}, and EKMM $S$-modules \cite{EKMM}.
Similarly, we let $G\sS$, $G\sS\!p$ and $G\sZ$ denote the corresponding categories
of genuine $G$-spectra from \cite{MM}, \cite{LMS}, and again \cite{MM}.  
We start with a machine that lands in $G\sS$.  It is related to but different from the geometric
machine of the previous section, and it is the choice preferred in \cite{GM2} and in the 
sequels \cite{GMMOAdd, GMMOMult, MMO}. The sphere $G$-spectrum $S_G$ in $G\sS$ 
is cofibrant, and so are the suspension $G$-spectra $\SI_G X$ of cofibrant based $G$-spaces $X$.  
We then give the variant machine that lands in $G\sS\!p$ or $G\sZ$, where every object is fibrant, 
and give a comparison that illuminates homotopical properties of the first machine via its comparison 
with the second.

\subsection{Equivariant $E_{\infty}$ operads}\label{sec:Einfty}

Since operads make sense in any symmetric mon\-oidal category, we have operads 
of categories, spaces, $G$-categories, and $G$-spaces.  Operads in $G\sU$ were
first used in \cite[VII]{LMS}. Although we are only 
interested in finite groups $G$ in this paper, the following definition makes sense 
for any topological group $G$ and is of interest in at least the generality of compact 
Lie groups. 

\begin{defn}\mylabel{EinfGoperad} An $E_{\infty}$ operad $\sC_G$ of $G$-spaces is an
operad in the cartesian monoidal category $G\sU$ such that $\sC_G(0)$ is a contractible $G$-space 
and the $(G\times \SI_j)$-space $\sC_G(j)$ is a 
universal principal $(G,\SI_j)$-bundle for each $j\geq 1$. Equivalently, for a subgroup $\LA$ of $G\times \SI_j$, the $\LA$-fixed point space 
$\sC_G(j)^{\LA}$ is contractible if $\LA\cap \SI_j = \{e\}$ and is empty otherwise. 
We say that $\sC_G$ is reduced if $\sC_G(0)$ is a point. 
\end{defn}

As is usual in equivariant bundle theory, we think of $G$ as acting from the left 
and $\SI_j$ as acting from the right on the spaces $\sC_G(j)$. 
These actions must commute and so define an action of $G\times \SI_j$.
We shall say nothing more about equivariant bundle theory except to note the 
following parallel.  In \cite{MayGeo}, an operad $\sC$ of spaces was defined 
to be an $E_{\infty}$ operad if $\sC(j)$ is a free contractible $\SI_j$-space.  
Effectively, $\sC(j)$ is then a universal principal $\SI_j$-bundle.  If we regard
each $\sC(j)$ as a $G$-trivial $G$-space, such an operad is called a naive 
$E_{\infty}$ operad of $G$-spaces.  Analogously, we have defined genuine $E_{\infty}$ operads  
by requiring the $\sC_G(j)$ to be universal principal $(G,\SI_j)$-bundles. That dictates 
the appropriate homotopical properties of the $\sC_G(j)$, and it is only 
those homotopical properties and not their bundle theoretic consequences that
concern us in the theory of operads. The bundle theory implicitly tells us 
which homotopical properties are relevant to equivariant infinite loop 
space theory.   Our default is that $E_{\infty}$ operads are understood to be
genuine unless otherwise specified.

We give two well-known examples.  Recall that a complete $G$-universe $U$ is a $G$-inner 
product space that contains countably many copies of each irreducible representation of $G$; a canonical choice 
is the sum of countably many copies of the regular representation $\rho_G$.

\begin{exmp}\mylabel{InfSteiner}(The Steiner operad $\sK_U$)
Inclusions $V\subset W$ induce
inclusions of operads $\sK_V\rtarr \sK_W$.  Let $\sK_U$ be the union over $V\subset U$
of the operads $\sK_V$, where $U$ is a complete $G$-universe.\footnote{We denoted the
nonequivariant version as $\sC$ in \cite{Rant1}, but we prefer the notation $\sK_U$ here.}  
This is the infinite Steiner operad of $G$-spaces.  It is an $E_{\infty}$ operad since $\SI_j$-acts 
freely on $\sK_U(j)$ and $\sK_U(j)^{\LA}$ is contractible if $\LA\subset G\times \SI_j$ and
$\LA\cap \SI_j = e$.  Indeed, such a $\LA$ is isomorphic to a subgroup $H$ of $G$ via the projection 
$G\times \SI_j \rtarr G$, and if we let $H$ act on $U$ through the isomorphism, then $U$ is a complete
$H$-universe and $U^H$ is isomorphic to $\bR^{\infty}$.  Therefore, by the proof of Proposition~\ref{SteinerHtpy} $\sK_U(j)^{\LA}$
is equivalent to the configuration space $F(\bR^{\infty},j)$, which is contractible.
\end{exmp}

\begin{exmp}\mylabel{LinIsoms}(Linear isometries operad) 
The equivariant linear isometries operad $\sL_U$ was first 
used in \cite[VII\S1]{LMS} and is defined just as nonequivariantly (e.g. \cite[\S2]{Rant1}).
The $(G\times \SI_j)$-space $\sL_U(j)$ is the space of linear isometries $U^j\rtarr U$, with 
$G$ acting by conjugation, and $\sL_U$ is an $E_{\infty}$ operad of $G$-spaces if $U$ is a complete $G$-universe. Indeed, 
$\SI_j$ acts freely on $\sL_U(j)$ and $\sL_U(j)^{\LA}$ is contractible if $\LA\subset G\times \SI_j$ 
and $\LA\cap \SI_j = e$.  If $\LA\iso H$ and $H$ acts on $U$ through the isomorphism,
then $U$ is a complete $H$-universe and $\sL_U(j)^H$ is isomorphic to the space of 
$H$-linear isometries $U^j\rtarr U$. The usual argument that $\sL(j)$ is contractible 
(e.g. \cite[I.1.2]{MQR}) adapts to prove that this space is contractible.
\end{exmp}

We define $E_\infty$-operads in $G$-categories in \S\ref{sec:CatEin} and give examples in \S\ref{sec:GenPermGCat} and \S\ref{SecPQR}.

\begin{rem}\mylabel{reduced} We will encounter one naturally occuring operad that 
is not reduced.
When an operad $\sC$ acts on a space $X$ via maps $\tha_i$ and we choose points 
$c_i\in \sC(i)$, we have a map $\tha_0\colon \sC(0)\rtarr X$ and the relation 
\[ \tha_2(c_2;\tha_0(c_0),\tha_1(c_1,x)) = \tha_1(\ga(c_2;c_0,c_1),x) \]
for $x\in X$. When the $\sC(i)$ are connected, this says that $\tha_0(c_0)$
is a unit element for the product determined by $c_2$. Reduced operads
give a single unit element. The original definition 
\cite[1.1]{MayGeo} required operads to be reduced.
\end{rem}

\begin{lem}\mylabel{Fixedptop} Let $\sC_G$ be an $E_{\infty}$ operad of $G$-spaces and define
$\sC = (\sC_G)^G$.  Then $\sC$ is an $E_{\infty}$ operad of spaces.  If $Y$ is a 
$\sC_G$-space, then $Y^G$ is a $\sC$-space. 
\end{lem}
\begin{proof}  $(\sC_G)^G$ is an operad since the fixed point functor commutes 
with products, and it is an $E_{\infty}$ operad since the space $\sC_G(j)^G$ is 
contractible and $\SI_j$-free. 
\end{proof}

\subsection{The infinite loop space machine: orthogonal $G$-spectrum version}\label{sec:MachineGS}

In brief, we have a functor $\bE_G = \bE_G^{\sS}$ that assigns an orthogonal $G$-spectrum $\bE_G Y$ 
to a $G$-space $Y$ with an action by some chosen $E_{\infty}$ operad $\sC_G$ of $G$-spaces.  We 
want to start with $\sC_G$-algebras and still exploit the Steiner operads, and we use the product of operads 
trick recalled in \S2.3 to allow this; compare \cite[\S9]{Rant1}.   For simplicity of notation, define 
$\sC_V = \sC_G\times \sK_V$.   We use the following observation.

\begin{lem}\label{CVEV} If $\sC_G$ is an $E_\infty$ operad of $G$-spaces, then the projection 
$$\sC_V(j) = \sC_G(j)\times \sK_V(j) \rtarr \sK_V(j)$$  
is a  $(G\times \SI_j)$-equivalence for each $j$.  
\end{lem}
\begin{proof}
We must show that for each subgroup $\LA\subseteq G\times \SI_j$, the induced map on fixed points 
\[\sC_G(j)^\LA\times \sK_V(j)^\LA \rtarr \sK_V(j)^\LA\]
 is an equivalence. If $\LA\cap \{e\}\times \SI_j = \{e\}$, then $\sC(j)^\LA\simeq \ast$, so the projection 
 is an equivalence. If $\LA$ contains a non-identity permutation, then the fixed points on both sides are 
 empty. Both sides are trivial if $j=0$.
\end{proof}

We view $\sC_G$-spaces as $\sC_V$-spaces for all $V$ via the projections $\sC_V\rtarr \sC_G$, and 
$\sC_V$ acts on $V$-fold loop spaces via its projection to $\sK_V$.  Write $\mathbf{C}_V$ for the monad 
on based $G$-spaces associated to the operad $\sC_V$. 
The categories of $\sC_V$-spaces and $\mathbf{C}_V$-algebras are isomorphic.  As in the $V$-fold
delooping argument, the unit 
$\et\colon \Id\rtarr \OM^V\SI^V$ of the monad $\OM^V\SI^V$ and the action $\tha$ of 
$\mathbf{C}_V$ on the $G$-spaces $\OM^V\SI^V X$ induce a composite natural map 
\[ \xymatrix@1{ \al_V\colon \mathbf{C}_VX \ar[r]^-{\mathbf{C}_V \et} 
&  \mathbf{C}_V\OM^V\SI^V X \ar[r]^-{\tha} &\OM^V\SI^V X,\\} \]
and $\al_V \colon \mathbf{C}_V\rtarr \OM^V\SI^V$ is a map of monads whose adjoint defines
a right action of $\mathbf{C}_V$ on the functor $\SI^V$. 

\begin{defn}[From $\sC_G$-spaces to orthogonal $G$-spectra] \mylabel{Einforth} Let $Y$ be a $\sC_G$-space.  
We define an orthogonal $G$-spectrum 
$\bE_G Y$, which we denote by $\bE_G^{\sS}Y$ when necessary for clarity.  Let
\[ \bE_GY(V) = B(\SI^V,\mathbf{C}_V,Y). \]
Using the action of isometric isomorphisms on the $\sK_V$ and $\SI^V$, as in the
previous section but starting with $Y$ regarded as a constant $\sI_G$-functor, as we 
can since its action by $\sC_G$ is independent of $V$, this defines an $\sI_G$-space.
The structure $G$-map 
\[ \si\colon  \SI^{W}\bE_GY(V) \rtarr \bE_GY(V\oplus W) \]
is the composite
\[ \SI^{W} B(\SI^V,\mathbf{C}_V,Y)\iso B(\SI^{V\oplus W},\mathbf{C}_V,Y)
\rtarr B(\SI^{V\oplus W},\mathbf{C}_{V\oplus W},Y). \]
obtained by commuting $\SI^{W}$ with geometric realization and using the map of monads 
$\BC_V \rtarr \BC_{V\oplus W}$ induced by the inclusion $i\colon \sK_V \rtarr \sK_{V\oplus W}$. 
\end{defn}

Just as in (\ref{myold}), we have the diagram of $\sC_V$-spaces and $\sC_V$-maps
\begin{equation}\label{myold2}
\xymatrix@1{ Y & \ar[l]_-{\epz} \ar[r]^-{\overline{\al}} 
B(\mathbf{C}_V,\mathbf{C}_V,Y) & 
B(\OM^V\SI^V,\mathbf{C}_V,Y) \ar[r]^-{\zeta} & 
\OM^{V} B(\SI^{V},\mathbf{C}_V,Y),\\}
\end{equation}
where $\overline{\al} = B(\al,\id,\id)$. \myref{Vrec} applies verbatim, with the same proof.     We let
$\xi_V= \ze\com \overline{\al} \com \nu$, where $\nu$ is the canonical homotopy inverse to $\epz$.
Then the following diagram commutes, where $\tilde{\si}$ is adjoint to $\si$. 

\[ \xymatrix{
& Y \ar[dl]_{\xi_V} \ar[dr]^{\xi_{V\oplus W}} &\\
 \OM^{V}\bE_GY(V)\ar[rr]_-{\OM^V\tilde{\si}}  & & \OM^{V\oplus W}\bE_GY(V\oplus W)\\} \]
Therefore $\OM^V\tilde{\si}$ is a weak equivalence if $V \supset \bR$.  If we replace 
$\bE_G Y$ by a fibrant approximation $\BR\bE_G Y$, there results a group completion
$\xi\colon Y\rtarr (\BR\bE_G Y)_0$.  
We shall shortly use the category $\sS\!p$ to give an explicit way to think about this.

\begin{rem} Since $\sK_0(0)=\{\ast\}$, $\sK_0(1) = \{\id\}$, and $\sK_0(j) = \emptyset$ 
for $j >1$, $\mathbf{C}_0$ is the identity functor if $\sC_G(0) = \{\ast\}$ and $\sC_G(1) = \{\id\}$.
In that case
\[ \bE_GY(0) = B(\SI^0,\mathbf{C}_0,Y) = B(\Id,\Id,Y)\iso Y. \]
\end{rem}

We comment on an alternative point of view not taken above but relevant below.  
We can use the product of operads trick from \cite{MayGeo} to replace a $\sC_G$-space $Y$ by the equivalent $\sK_U$-space  
$B(\BK_U, \BC_U,Y)$, where $\BC_U$ is the monad associated to the $E_{\infty}$ operad $\sC_U = \sC_G\times \sK_U$ 
and from there only use Steiner operads.  However, there is a catch.  A $\sK_U$-algebra $Y$ is  a $\sK_V$-algebra by
restriction, but the constant $\sI_G$-space $Y$ is {\em not} a $\sK_*$-$G$-space in the sense of \myref{KStar} since 
conjugation by isometries is not compatible with the inclusions used to define $\sK_U$.  Therefore the $B(\SI^V,\sK_V,Y)$ 
do not define an $\sI_G$-space.  However, ignoring isometries, they do define a coordinate free $G$-prespectrum, as defined in
\cite[II.1.2]{MM}.  That can be viewed as the starting point for the alternative machine we construct next.

\subsection{The infinite loop space machine:  Lewis-May $G$-spectrum version}\label{sec:MachineGSp}

A Lewis-May (henceforward LM) $G$-spectrum $E$ consists of $G$-spaces 
$EV$ for each finite dimensional sub $G$-inner product space $V$ in a complete
$G$-universe $U$ together with $G$-homeomorphisms $EV\rtarr \OM^{W-V}EW$ 
whenever $V\subset W$.   For a based $G$-space $X$ we define $Q_G X = \colim \OM^V\SI^V X$.
The suspension LM $G$-spectrum  $\SI^{\infty}_G X$ has $V$th $G$-space 
$Q_G \SI^V X$, and the functor $\SI^{\infty}_G$ is left adjoint to the zeroth space functor $\OM^{\infty}_G$.
We sometimes change notation to $\SI^{\infty}_U$ and $\OM^{\infty}_U$, allowing change of universe.
While $G\sS\!p$ is not symmetric monoidal, that is rectified by passage to the
$S_G$-modules of \cite{EKMM}, at the inevitable price of losing the adjunction;
see \cite[\S11]{Rant1}.

The operad $\sK_U$ acts on $\OM^{\infty}_UE$ for any LM
$G$-spectrum $E$.  One could not expect such precise structure 
when working with orthogonal $G$-spectra.  Nonequivariantly,
such highly structured infinite loop spaces are central to
calculations, and it is to be hoped that the equivariant theory
will eventually reach a comparable state.   Therefore it is natural to 
want an infinite loop space machine that lands in the category $G\sS\!p$ 
of LM $G$-spectra.

The operad $\sK_U$ plays a privileged role. As noted above, if $\sC_G$ is an $E_\infty$ $G$-operad,
we can convert $\sC_G$-spaces to equivalent $\sK_U$-spaces, so that it suffices to build a machine for $\sK_U$-spaces. 
On the other hand,  $\sC_G$ spaces inherit actions of $\sC_U  = \sC_G\times \sK_U$, so that it suffices to build 
a machine for $\sC_U$-spaces.  To encompass both of these approaches in a single machine, we suppose given a map 
(necessarily an equivalence) of $E_\infty$ $G$-operads $\sO_G\rtarr \sK_U$.   We can take $\sO_G = \sC_U$ 
or $\sO_G = \sK_U$, but both here and in \cite{GMMOAdd, GMMOMult, MMO}, our primary interest is in $\sC_U$. 
Formally, the equivariant theory now works in the same way as the nonequivariant theory, and we follow the summary in \cite[\S9]{Rant1}.  An
early version of this machine is in the paper \cite{CostWan} of Costenoble and Waner.  

\begin{sch} We must use the Steiner operads $\sK_V$ and $\sK_U$ rather than the little disks operads
$\sD_V$ and $\sD_U$, which was the choice in \cite{CostWan}, and our notion of an $E_{\infty}$ operad 
of $G$-spaces should replace the notion of a complete operad used there.
\end{sch} 

\begin{defn}[From $\sO_G$-spaces to Lewis-May $G$-spectra]\mylabel{EinfLM}
Let $Y$ be an $\sO_G$-space.  We define a LM $G$-spectrum $\bE_G Y$, which we denote by 
$\bE_G^{\mathrm{Sp}}$ when necessary for clarity, by
\[\bE_G Y = B(\SI^{\infty}_G,\mathbf{O}_G, Y).\]
Here $\mathbf{O}_G$ acts on $\SI^{\infty}_G$ through its projection to  $\bK_U$.
\end{defn}
We have the diagram of $\sO_G$-spaces and $\sO_G$-maps
\begin{equation}\label{myold3}
\xymatrix@1{ Y & \ar[l]_-{\epz} \ar[r]^-{\overline{\al}_U} 
B(\mathbf{O}_G,\mathbf{O}_G,Y) & 
B(Q_G,\mathbf{O}_G,Y) \ar[r]^-{\zeta} & 
\OM^{\infty} _GB(\SI^{\infty}_G,\mathbf{O}_G,Y) = \OM^{\infty}_G \bE_G Y, \\}
\end{equation}
where $\overline{\al}_U = B(\al_U,\id,\id)$.  
As explained nonequivariantly in \cite[\S9]{Rant1}, the following analogue of \myref{Vrec} holds.

\begin{thm}\mylabel{Vrec2} Let $\sO_G$ be an $E_\infty$ $G$-operad  with a map of operads 
$\sO_G \rtarr \sK_U$. The following statements hold
for an $\sO_G$-space $Y$.
\begin{enumerate}[(i)]
\item The map $\epz$ is a $G$-homotopy equivalence with a natural homotopy inverse $\nu$.
\item The map $\overline{\al}_U $ is an equivalence when $Y$ is connected and is a
group completion otherwise.
\item The map $\ze$ is an equivalence.
\end{enumerate}
Therefore the composite 
$$\xi = \ze\com \overline{\al}_U \com \nu\colon Y\rtarr \OM^{\infty}_G \bE_G Y$$ 
is an equivalence if $Y$ is grouplike and is a group completion otherwise. 
\end{thm}

We shall not pursue this variant of the recognition principle in further detail, but we reemphasize
that its much tighter relationship with space level data may eventually aid equivariant calculation.  However, 
it is worth stating the alternative geometric version of the stable BPQ theorem  to which it leads.  Here we specialize to the 
case $\sO_G=\sK_U$. This allows us to use
the pairings of Steiner operads described in \S\ref{SecSteinerPair}, which are not available for  other
$E_{\infty}$ operads. By passage to colimits, we obtain
the following analogue of \myref{SteinPair}.

\begin{prop}  For $G$-universes $U$ and $U'$, there is a unital, associative, and commutative pairing 
$$ \boxtimes\colon (\sK_U,\sK_{U'}) \rtarr \sK_{U\oplus U'} $$
of Steiner operads of $G$-spaces.
\end{prop}

Passing to monads, we obtain a unital, associative, and commutative system of pairings
\begin{equation} 
\boxtimes\colon \bK_U X \sma \bK_{U'} Y \rtarr \bK_{U\oplus U'} (X\sma Y). 
\end{equation}

Passage to colimits from \myref{alsma} gives the following analogue of that result.

\begin{lem}\mylabel{alsma2}  The following diagram commutes.
\[ \xymatrix{
\mathbf{K}_U X \sma \mathbf{K}_U Y \ar[r]^-{\boxtimes} \ar[d]_{\al_U\sma \al_U} 
& \mathbf{K}_{U\oplus U} (X\sma Y) \ar[d]^{\al_{U\oplus U}}\\
\OM^{\infty}_U \SI^{\infty}_U X \sma \OM^{\infty}_U\SI^{\infty}_U Y \ar[r]_-{\sma} & 
\OM^{\infty}_{U\oplus U}\SI^{\infty}_{U\oplus U}(X\sma Y).\\} \]
\end{lem}

The following recognition principle for pairings can by derived from \myref{PairVrec} by passage to colimits
or can be proven by the same argument as there.  We note that our definition of the machine $\bE_G$ depends
on a choice of complete $G$-universe $U$, and we sometimes write $\bE_U$ to indicate that choice.

\begin{prop}\mylabel{PairUrec}
A pairing $f\colon X\sma Y\rtarr Z$ of a $\sK_U$-space $X$ and a $\sK_U$-space $Y$
to a $\sK_{U\oplus U}$-space $Z$ induces a map 
\[ \bE f\colon \bE_U X \barwedge \bE_U Y \rtarr \bE_{U\oplus U} Z \]
of LM $G$-spectra indexed on $U\oplus U$ such that the following diagram commutes.
\[
\xymatrix{
X\sma Y \ar[r]^-{\xi\sma \xi}  \ar[d]_f &   \OM^{\infty}_U \bE_U X \sma \OM^{\infty}_U \bE_U Y \ar[r] & \OM^{\infty}_{U\oplus U}(\bE_U X \barwedge \bE_U Y) 
\ar[d]^{\OM^{\infty}_{U\oplus U}\bE f} \\
Z \ar[rr]_-{\xi} & & \OM^{\infty}_{U\oplus U} \bE_{U\oplus U} Z
\\} \] 
\end{prop}

We can internalize the external smash product, as in \cite{LMS}, by choosing a linear isometry $\ph\colon U\oplus U\rtarr U$.
Then $\ph$ induces a change of universe functor $\ph_*$ which allows us to replace the right arrow by $\OM^{\infty}_U \ph_*\bE f$.
In the following result we can either stick with Lewis-May $G$-spectra or pass to the $S_G$-modules  of 
\cite{EKMM, MM}.  We interpret the smash product according to choice.

\begin{thm}[The $\sK_U$-space BPQ theorem for Lewis-May $G$-spectra]\mylabel{BPQ3} 
For based $G$-spaces $X$, there is a natural equivalence 
\[ \om\colon \SI^{\infty}_U X \rtarr \bE_U \BK_U X \]
such that the following diagram commutes for based $G$-spaces $X$ and $Y$.
\[
\xymatrix{
\SI^{\infty}_U X \sma \SI^{\infty}_U Y \ar[d]_{\iso} \ar[r]^-{\om\sma \om} 
&  \bE_U \BK_UX \sma \bE_U\BK_U Y \ar[d]^{\bE(\boxtimes)}\\
\SI^{\infty}_U(X\sma Y) \ar[r]_-{\om} & \bE_U\BK_U (X\sma Y)\\}
\]
\end{thm}
\begin{proof}[Sketch proof]  The first statement is the usual extra degeneracy argument  \cite[9.8]{MayGeo}.  
We comment on the diagram.  In either  $G\sS\!p$ or $G\sZ$,  it is an internalization of a diagram of 
$G$-spectra indexed on $U\oplus U$.
\[
\xymatrix{
\SI^{\infty}_U X \barwedge \SI^{\infty}_U Y \ar[d]_{\iso} \ar[r]^-{\om\barwedge \om} 
&  \bE_U \BK_UX \barwedge \bE_U\BK_U Y \ar[d]^{\bE(\boxtimes)}\\
\SI^{\infty}_{U\oplus U}(X\sma Y) \ar[r]_-{\om} & \bE_{U\oplus U}\BK_{U\oplus U} (X\sma Y)\\}
\]
The isomorphism on the left is trivial on the prespectrum level (indexing on inner product 
$G$-spaces of the form $V\oplus W$) and follows on the spectrum level.  After passage to
adjoints, to check commutativity it suffices to check starting from $X\sma Y$ on the bottom 
left, where an inspection of definitions gives the conclusion.  If in $G\sS\!p$, this is internalized 
by use of a linear isometry $\phi\colon U\oplus U\rtarr U$.  If in $G\sZ$, this is internalized by 
use of the definition of the smash product in terms of the linear isometries operad $\sL_U$, as in \cite{EKMM, MM}. 
\end{proof}     

In fact, with the model theoretic modernization of the original version of the theory that is given 
nonequivariantly in \cite{Five}, one can redefine the restriction of $\bE_U$ to cofibrant 
$\sK_U$-spaces $Y$ to be 
$$\bE_UY = \SI^{\infty}_G\otimes_{\BK_U}Y,$$ 
where $\otimes_{\BK_U}$ is the evident coequalizer. 
With that reinterpretation and taking $X$ to be a $G$-CW complex,  $\bE_U \BK_U X$ is actually isomorphic to 
$\SI^{\infty}_GX$. 

The nonequivariant statement is often restricted to the case $Y=S^0$. Then 
$\BK_U S^0$ is the disjoint union of operadic models for the classifying 
spaces $B\SI_j$.  Similarly, $\BK_U S^0$ is the disjoint union of 
operadic models for the classifying $G$-spaces $B(G,\SI_j)$.

\subsection{{A comparison of infinite loop space machines}}
We compare the $\sS$ and $\sS\!p$ machines $\bE_G^{\sS}$ and $\bE_G^{\sS\!p}$ by
transporting both of them to the category $G\sZ$ of $S_G$-modules, following \cite{MM}.  
As discussed in
\cite[IV\S4]{MM} with slightly different notations, there is a diagram of Quillen equivalences 
$$\xymatrix{
G\sP \ar[rr]<.5ex>^{L} \ar[dd]<.5ex>^{\bP}
& &G\sS\!p \ar[ll]<.5ex>^{\ell} \ar[dd]<.5ex>^{\bF}\\
& & \\
G\sS  \ar[uu]<.5ex>^{\bU} \ar[rr]<.5ex>^{\bN} 
& & G\sZ. \ar[ll]<.5ex>^{\bN^{\#}} \ar[uu]<.5ex>^{\bV}
}$$
Here $G\sP$ is the category of coordinate-free $G$-prespectra.  
The left adjoint $\bN$ is strong symmetric monoidal, and the unit map 
$\et\colon X \rtarr \bN^{\#}\bN X$ is a weak equivalence for all cofibrant
orthogonal $G$-spectra $X$.  It can be viewed as a fibrant approximation in 
the stable model structure on $G\sS$.  The pair $(\bN,\bN^{\#})$ is a Quillen
equivalence with the positive stable model structure on $G\sS$;
see \cite[III\S\S4,5]{MM}.

We can compare machines using the diagram. In fact, by a direct inspection of
definitions, we see the following result, which is essentially a reinterpretation of 
the original construction of \cite{MayGeo} that becomes visible as soon as one introduces orthogonal spectra.

\begin{lem} The functor $\bE_G^{\sS\!p}$ from $\sC_G$-spaces to the category
$G\sS\!p$ of Lewis-May $G$-spectra is naturally isomorphic to the composite functor 
$L\com \bU\com \bE_G^{\sS}$.
\end{lem}

As explained in \cite[IV\S5]{MM}, there is a monad $\bL$ on $G\sS\!p$ and a category
$G\sS\!p[\bL]$ of $\bL$-algebras.  The left adjoint $\bF$ in the diagram is the composite of left adjoints
\[ \bL\colon G\sS\!p\rtarr G\sS\!p[\bL]\ \ \text{and}\ \  \bJ\colon G\sS\!p[\bL]\rtarr G\sZ. \]  
The functor $L\com \bU\colon G\sS\rtarr G\sS\!p$ lands naturally in $G\sS\!p[\bL]$, 
so that we can define 
\[ \bM = \bJ \com L \com \bU\colon G\sS\rtarr G\sZ. \]
By \cite[IV.5.2 and IV.5.4]{MM}, $\bM$ is lax symmetric monoidal and there is a natural
lax symmetric monoidal map $\al\colon \bN X\rtarr \bM X$ that is a weak equivalence when 
$X$ is cofibrant. Effectively, we have two infinite loop space machines landing 
in $G\sZ$, namely $\bN\com\bE_G^{\sS}$ and $\bJ \com \bE_G^{\sS\!p}$.
In view of the lemma, the latter is isomorphic to $\bM\com \bE_G^{\sS}$, hence
\[ \al\colon \bN \com\bE_G^{\sS}\rtarr \bM \com\bE_G^{\sS} 
\iso \bJ \com \bE_G^{\sS\!p} \] 
compares the two machines, showing that they are equivalent for all 
practical purposes. Homotopically, these categorical distinctions are irrelevant, 
and we can use whichever machine we prefer, deducing properties of one from the other.

\subsection{Examples of $E_\infty$ spaces and $E_\infty$ ring spaces}\label{sec:Exmps}

Many of the examples from the nonequivariant theory generalize directly to the equivariant setting.  To illustrate
the point of using varying $E_{\infty}$ operads and their natural actions on spaces of interest, rather than just 
using $\sK_U$, we focus on actions of the linear isometries operad $\sL_U$.    

Nonequivariantly, taking 
$U\iso \bR^{\infty}$, a systematic account of naturally occurring examples of $\sL_U$-spaces was already given in 
\cite[\S I.1]{MQR}.  It was revisited briefly in more modern language \cite[\S2]{Rant1}.   It includes the infinite classical groups $O$, $SO$, $Spin$,
$U$, $SU$, $Sp$, their classifying spaces, constructed either using Grassmannian manifolds or the standard classifying space functor $B$, and all 
of their associated infinite homogeneous spaces.  All of these examples are grouplike, and all of them are given infinite loop spaces 
by application of the nonequivariant infinite loop space machine.  The discussion in \cite{MQR, Rant1} was in terms of sub 
inner-product spaces $V$ of a universe $U$.  The point to make here is that the entire exposition works verbatim equivariantly, with the $V$ being 
sub inner-product $G$-spaces of our complete $G$-universe $U$.  We give a brief account to show the idea.

As explained in \cite[\S2]{Rant1}, an $\sI_G$-FCP (functor with cartesian product) is a lax symmetric monoidal functor $ \sI_G \rtarr \sT_G$. 
We say that an $\sI_G$-FCP  is monoid-valued if it factors through the category of equivariant topological monoids 
and monoid homomorphisms. The classical groups all give group-valued $\sI_G$-FCPs: 
\[ V\mapsto O(V), \quad V\mapsto SO(V), \quad V\mapsto U(\bC\otimes_{\bR} V), V\mapsto SU(\bC\otimes_{\bR}V), \ \text{etc}. \] 

Any $\sI_G$-FCP $X$ extends to a functor on all isometries (not just isometric isomorphisms) as follows: an isometry $\alpha:V \rtarr W$ yields an 
identification $W\iso \alpha(V) \oplus \alpha(V)^\perp$. Then $X(\alpha)$ is the composite 
\[X(V)\xrightarrow{X(\alpha)\times 0} X(\alpha(V))\times X(\alpha(V)^\perp) \rightarrow X(\alpha(V) \oplus \alpha(V)^\perp).\]
Then the colimit $X(U) = \colim_V X(V)$ inherits an action of $\sL_U$.  The classifying space $BF$ of a 
monoid-valued $\sI_G$-FCP $F$ is an $\sI_G$-space, and the cited sources show that
$F$ is equivalent to $\OM BF$ as an $\sL_U$-space when $F$ is group-valued.  

The formal structure of the operad pair $(\sK_U,\sL_U)$ works the same way equivariantly as nonequivariantly.
It is an $E_{\infty}$ operad pair in the sense originally defined in \cite[VI.1.2]{MQR} and 
reviewed in \cite[\S1]{Rant1} and, in more detail, \cite[4.2]{Rant2}. See \S\ref{sec:MultCatOps} 
below for an example of an operad pair in $G$-categories.  The action of $\sL_U$ on 
$\sK_U$ is defined nonequivariantly in \cite[\S3]{Rant1}, and it works the same way equivariantly.  

From here, multiplicative infinite loop space theory works equivariantly to construct 
$E_{\infty}$ ring $G$-spectra from $(\sK_U,\sL_U)$-spaces, alias $E_{\infty}$-ring 
$G$-spaces, in exactly the same way as nonequivariantly \cite{MQR, Rant1, Rant2}.  
In particular, for any $\sL_U$-algebra $X$,  the free $\sK_U$-algebra $\BK_U X_+$ is an $E_\infty$ ring 
$G$-space, where $X_+$ is obtained from $X$ by adjoining an additive $G$-fixed basepoint $0$.  The 
group completion $\alpha_U\colon \BK_U X_+ \rtarr Q_G X_+$ is a map of $E_{\infty}$ ring
$G$-spaces, and $\bE_G \BK_U X_+$ is equivalent to $\SI^{\infty}_GX_+$ as $E_{\infty}$ ring $G$-spectra. 

As we intend to show elsewhere \cite{AM}, 
the passage from category level data to $E_{\infty}$-ring $G$-spaces, in analogy with 
\cite{MayMult, Rant2}, generalizes to equivariant multicategories.

We remark that the usual construction of Thom $G$-spectra, such as $MO_G$ and $MU_G$, 
already presents them as $E_{\infty}$ ring $G$-spectra, without use of infinite loop space theory, 
as was explained and generalized in \cite[Chapter X]{LMS}.

\subsection{Some properties of equivariant infinite loop space machines}\label{sec:MachineProps}
Many properties of the infinite loop space machine $\bE_G$ follow directly from the
group completion property, independent of how the machine is constructed, but it is
notationally convenient to work with the machine $\bE_G^{\sS\!p}$, for which $\xi$
is a natural group completion without any bother with fibrant approximation. The 
results apply equally well to $\bE_G^{\sS}$.  It is plausible to 
hope that the group completion property actually characterizes the machine up to homotopy, 
as in \cite{MT}, but the proof there fails equivariantly. A direct point-set level comparison of our 
machine with a new version of the Segal-Shimakawa machine will be given in \cite{MMO}. 
We illustrate with the following two results, some version of which must hold for any
equivariant infinite loop space machine $\bE_G$.   The first says that it commutes with 
passage to fixed points and the second says that it commutes with products, both up
to weak equivalence.

\begin{thm}\mylabel{SpFour} For $\sC_G$-spaces $Y$, there is a natural map of spectra
$$\phi\colon \bE(Y^G) \rtarr (\bE_G Y)^G$$
that induces a natural map of spaces under $Y^G$
\[\xymatrix{
& Y^G \ar[dl]_{\xi} \ar[dr]^{\xi^G} & \\ 
\OM^{\infty}\bE(Y^G)  \ar[rr] & &  
(\OM^{\infty}_G\bE_G Y)^G\\} \]
in which the diagonal arrows are both group completions. Therefore the 
horizontal arrow is a weak equivalence of spaces and $\ph$ is a weak
equivalence of spectra.
\end{thm}  
\begin{proof} For based $G$-spaces $X$, we have natural inclusions
$\mathbf C_{U^G}(X^G)\rtarr (\mathbf C_UX)^G$ and $\SI^{\infty}(X^G)
\rtarr (\SI^{\infty}_G X)^G$.
For $G$-spectra $E$, we have a natural isomorphism $\OM^{\infty}(E^G)\iso (\OM_G^{\infty} E)^G$.
This gives the required natural map of spectra
\[ \xymatrix@1{ \bE(Y^G) = B(\SI^{\infty},\mathbf C_{U^G}, Y^G) \ar[r]^-{\phi} & (B(\SI^{\infty}_G,\mathbf C_U, Y))^G 
= (\bE_G Y)^G\\} \]
and the induced natural map of spaces under $Y^G$.  Since the diagonal arrows in the diagram are group completions, 
the horizontal arrow must be a homology isomorphism and hence a weak equivalence.  Since our spectra 
are connective, $\ph$ must also be a weak equivalence.
\end{proof}

\begin{thm}\mylabel{prodcom} Let $X$ and $Y$ be $\sC_G$-spaces. Then the map 
\[ \bE_G(X\times Y) \rtarr \bE_G X\times \bE_GY \]
induced by the projections is a weak equivalence of $G$-spectra.
\end{thm}
\begin{proof}  We are using that the product of $\sC_G$-spaces is a $\sC_G$-space, the proof of which
uses that the category of operads is cartesian monoidal. Working in $G\sS\!p$, the functor $\OM^{\infty}_G$ 
commutes with products and passage 
to fixed points and we have the commutative diagram
\[ \xymatrix{
& (X\times Y)^H\iso X^H\times Y^H \ar[dl]_{\xi^H} \ar[dr]^{\xi^H\times \xi^H} &\\
(\OM^{\infty}_G\bE_G(X\times Y))^H\ar[rr] & & (\OM^{\infty}_G\bE_GX)^H\times (\OM^{\infty}_G\bE_GY)^H.\\} \]
Since the product of group completions is a group completion, the diagonal arrows are both group completions.
Therefore the horizontal arrow is a weak equivalence. Since our spectra are connective, the conclusion follows.
\end{proof} 

\subsection{The recognition principle for naive $G$-spectra}\label{SecNaive}

We elaborate on \myref{SpFour}. The functor $\bE = \bE_{e}$ in that result is the nonequivariant infinite loop space machine, 
which is defined using the product of the nonequivariant Steiner operad $\sK = \sK_{U^G}$ and the
fixed point operad $\sC = (\sC_G)^G$.  We may think of $U^G$ as $\bR^{\infty}$, without reference to $U$, and start
with any (naive) $E_{\infty}$ operad $\sC$ to obtain a recognition principle for naive $G$-spectra, which are just spectra 
with $G$-actions. Again we can use either the category $\sS$ of orthogonal spectra or the category $\sS\!p$ of Lewis-May spectra, 
comparing the two by mapping to the category $\sZ$ of EKMM $S$-modules, but letting $G$ act on objects
in all three.  We continue to write $\bE$ for this construction since it is exactly the same construction as the nonequivariant one, 
but applied to $G$-spaces with an action by the $G$-trivial $E_{\infty}$ operad $\sC$.  

It is worth emphasizing that when working with naive $G$-spectra, there is no
need to restrict to finite groups. We can just as well work with general topological 
groups $G$.  The machine $\bE$  still enjoys the same properties, including the group 
completion property.  Working with Lewis-May spectra, 
the adjunction $(\SI^{\infty},\OM^{\infty})$ relating spaces and 
spectra applies just as well to give an adjunction relating based
$G$-spaces and naive $G$-spectra.  For based $G$-spaces $X$, the map 
$\al\colon \mathbf{C}X\rtarr \OM^{\infty}\SI^{\infty} X$ is a group completion of Hopf
$G$-spaces by the nonequivariant special case since $(\mathbf{C}X)^H = \mathbf{C}(X^H)$ and
$(\OM^{\infty}\SI^{\infty}X)^H =  \OM^{\infty}\SI^{\infty}(X^H)$.   

Returning to finite groups, we work with Lewis-May spectra and $G$-spectra in the rest
of this section in order to exploit the more precise relationship between spaces and spectra
that holds in that context.  However, the conclusions can easily be transported to orthogonal spectra.
We index genuine $G$-spectra on a complete $G$-universe $U$ and we index naive $G$-spectra on the trivial
$G$-universe $U^G\iso \bR^{\infty}$.  The inclusion of universes 
$i\colon U^G\rtarr U$
induces a forgetful functor $i^*\colon G\sS\!p^{U}\rtarr G\sS\!p^{U^G}$ from genuine $G$-spectra 
to naive $G$-spectra. It represents the forgetful 
functor from $RO(G)$-graded  cohomology theories to $\bZ$-graded cohomology theories. 
The functor $i^*$ has a left adjoint $i_*$.  The following observations are trivial but important.

\begin{lem}\mylabel{obvious}  The functors $i_*\SI^{\infty}$ and $\SI^{\infty}_G$ from based $G$-spaces to genuine $G$-spectra are isomorphic. 
\end{lem}
\begin{proof}
Clearly $\OM^{\infty}\io^* = \OM^{\infty}_G$, since both are evaluation at $V=0$,
hence their left adjoints are isomorphic.
\end{proof}

\begin{rem}\mylabel{etc} For $G$-spaces $X$, the unit of the $(i_*,i^*)$ adjunction gives a
natural map $\SI^{\infty}X \rtarr i^*i_*\SI^{\infty}X \iso \io^*\SI_G^{\infty}$ 
of naive $G$-spectra. It is very far from being an equivalence, as the tom Dieck 
splitting theorem shows; see \myref{SpOne}.
\end{rem}

The inclusion of universes $i\colon U^G\rtarr U$ induces an 
inclusion of operads of $G$-spaces $\io\colon \sK_{U^G}\rtarr \sK_{U}$, where $G$ acts
trivially on $\sK_{U^G}$.  The product of this inclusion and
the inclusion $\io\colon \sC = (\sC_G)^G \rtarr \sC_G$ is an inclusion
\[\io\colon \sC_{U^G} \equiv \sC \times \sK_{U^G} \rtarr \sC_G\times \sK_U \equiv \sC_U.\] 
Pulling actions back along $\io$ gives a functor $\io^*$ from $\sC_U$-spaces 
to  $\sC_{U^G}$-spaces. The following consistency statement is important since, 
by definition, the $H$-fixed point spectrum $E^H$ of a genuine $G$-spectrum $E$ is $(i^*E)^H$
and the homotopy groups of $E$ are $\pi_*^H(E) \equiv \pi_*(E^H)$. 

\begin{thm}\mylabel{SpSix}  Let $Y$ be a $\sC_G$-space.  Then there is a natural weak equivalence 
of naive $G$-spectra $\bE\io^*Y \rtarr i^* \bE_G Y$. 
\end{thm}
\begin{proof} Again, although we work with $\bE^{\sS\!p}_G$, the conclusion carries over to $\bE^{\sS}_G$. 
It is easy to check from the definitions that, for $G$-spaces 
$X$, we have a natural commutative diagram of $G$-spaces
\[ \xymatrix{
\mathbf{C_{U^G}}X \ar[r]^-{\al} \ar[d] & \OM^{\infty} \SI^{\infty} X \ar[d] \\
\mathbf{C_U}X \ar[r]_-{\al} & \OM^{\infty}_G \SI^{\infty}_G X. \\} \]
The vertical arrows both restrict colimits over representations to colimits
over trivial representations. 
Passing to adjoints, we obtain a natural commutative diagram 
\[ \xymatrix{
\SI^{\infty}_G\mathbf{C_{U^G}}X \ar[r] \ar[d] &  \SI^{\infty} X \ar[d] \\
\SI^{\infty}_G\mathbf{C_U}X \ar[r] & \SI^{\infty}_G X. \\} \]
The composite gives a right action of $\mathbf{C_{U^G}}$ on $\SI^{\infty}_G$ that  
is compatible with the right action of $\mathbf{C_U}$.  
Using the natural map $\SI^{\infty} \rtarr i^*\SI^{\infty}_G$ of \myref{etc}, there results
a natural map  
\[ \mu\colon \bE\io^*Y = B(\SI^{\infty},\mathbf{C_{U^G}},\io^*Y) \rtarr
B(i^*\SI^{\infty}_G,\mathbf{C_U},Y) \iso i^*\bE_G Y \]
of naive $G$-spectra. The following diagram commutes by a check of definitions.

{\small

\[\xymatrix{ Y\ar[d]_{=} & \ar[l]_-{\epz} \ar[rr]^-{B(\al,\id,\id)} 
B(\mathbf{C_{U^G}},\mathbf{C_{U^G}},\io^*Y) \ar[d] & &
B(Q,\mathbf{C_{U^G}},Y) \ar[r]^-{\zeta} \ar[d] & \OM^{\infty} B(\SI^{\infty},\mathbf{C_{U^G}},Y)\ar[d]^{\OM^{\infty}\mu}\\
Y & \ar[l]_-{\epz} \ar[rr]^-{B(\al,\id,\id)} B(\mathbf{C}_G,\mathbf{C}_G,Y) & &
B(Q_G,\mathbf{C_U},Y) \ar[r]^-{\zeta} & \OM^{\infty}_G B(\SI^{\infty}_G,\mathbf{C_U},Y).\\}\]

}

\noindent
Here the right vertical map is the map of zeroth spaces given by $\mu$.
Replacing the maps $\epz$ with their homotopy inverses, the
horizontal composites become group completions. Therefore $\OM^{\infty}\mu$ is
a weak equivalence, hence so is $\mu$.
\end{proof}

We also have the corresponding statement for the left adjoint $i_*$ of $i^*$. In effect, it gives 
a space level construction of the change of universe functor $i_*$ on connective $G$-spectra.  
We need a homotopically well-behaved version of the left adjoint of the functor $\io^*$ from $\sC$-spaces
to $\sC_G$-spaces, and we define it by $\io_!X = B(\mathbf{C}_G,\mathbf{C},X)$.

\begin{thm}\mylabel{Trouble} Let $X$ be a $\sC$-space. Then there is a natural weak equivalence of genuine 
$G$-spectra $\bE_G(i_!X) \htp i_*\bE(X)$.
\end{thm}
We give the proof in \S\ref{DoubleTrouble}, using a construction that is of independent interest.

\section{Categorical preliminaries on classifying $G$-spaces and $G$-operads}\label{sec:Prelim}

We recall an elementary functor $\sC\!at(\tilde G,-)$  
from $G$-categories to $G$-categories from our paper \cite{GMM} with Mona Merling.  
We explored this functor in detail in the context of 
equivariant bundle theory in \cite{GMM}, and we refer the reader there for proofs.  
In \S\ref{Sec2}, we shall use it to define 
a certain operad $\sP_G$ of $G$-categories. The $\sP_G$-algebras 
will be the genuine permutative $G$-categories.

\subsection{Chaotic topological categories and equivariant classifying spaces}

For (small) categories $\sA$ and $\sB$, we let $\sC\!at(\sA,\sB)$ denote the
category whose objects are the functors $\sA\rtarr \sB$ and whose morphisms
are the natural transformations between them.  
When $\sB$ has a right action by some group $\PI$, then $\sC\!at(\sA,\sB)$ inherits a right $\PI$-action.
When a group $G$ acts from the 
left on $\sA$ and $\sB$, 
$\sC\!at(\sA,\sB)$ inherits a left $G$-action by conjugation on objects and morphisms. Then 
$G\sC\!at\!(\sA,\sB)$ is alternative notation for the $G$-fixed category 
$\sC\!at(\sA,\sB)^G$ of $G$-functors and $G$-natural transformations.  
We have the $G$-equivariant version of the standard adjunction
\begin{equation}\label{adjoint}
 \sC\!at(\sA\times\sB,\sC) \iso \sC\!at(\sA,\sC\!at(\sB,\sC)). 
\end{equation}

\begin{defn} For a space $X$, the chaotic (topological) category $\tilde X$ has
object space $X$, morphism space $X\times X$, and structure maps $I$, $S$, $T$,
and $C$ given by $I(x) = (x,x)$, $S(y,x) = x$, $T(y,x) = y$, and 
$C((z,y),(y,x))  = (z,x)$.  For any point $\ast\in X$, the map $\et\colon X\rtarr X\times X$
specified by $\et(x) = (\ast,x)$ is a continuous natural isomorphism from the identity
functor to the trivial functor $\tilde{X}\rtarr {\ast}\rtarr \tilde{X}$, hence $\tilde{X}$ is equivalent
to $\ast$.   When $X=G$ is a topological group, $\tilde G$ is
isomorphic to the translation category of $G$, but the isomorphism encodes
information about the group action and should not be viewed as an identification; see \cite[1.7]{GMM}. We say that a topological
category with object space $X$ is chaotic if it is isomorphic to $\tilde{X}$.
\end{defn} 

\begin{defn} Without changing notation, we regard a topological group $\PI$ 
as a topological category with a single object $\ast$ and morphism space $\PI$,
with composition given by multiplication.  
Then $\Pi$ is isomorphic to the orbit category $\tilde{\PI}/\PI$, where $\Pi$
acts from the right on $\tilde{\PI}$ via right multiplication on
objects and diagonal right multiplication on morphisms. The resulting
functor $p\colon \tilde{\PI}\rtarr \PI$ is given by 
the trivial map $\PI\rtarr \ast$ of object spaces and the map 
$\xymatrix@1{  p\colon \PI\times\PI \rtarr \PI\times \PI/\PI \iso \PI\\}$
on morphism spaces specified by $p(\ta,\si) = \ta\si^{-1}$.
\end{defn}

\begin{thm}\cite[2.7]{GMM} For a $G$-space $X$ and a topological group 
$\PI$, regarded as a $G$-trivial $G$-space, the functor 
$p\colon \tilde{\PI}\rtarr \PI$ induces an isomorphism of 
topological $G$-categories
\[ \xi\colon \sC\!at(\tilde{X},\tilde{\PI})/\PI \rtarr \sC\!at_G(\tilde{X},\PI). \]
Therefore, passing to $G$-fixed point categories,
\[ (\sC\!at(\tilde{X},\tilde{\PI})/\PI)^G 
\iso \sC\!at(\tilde{X},\PI)^G \iso \sC\!at(\tilde X/G , \PI).\]
\end{thm}  

The last isomorphism is clear since $G$ acts trivially on $\PI$.
Situations where $G$ is allowed to act non-trivially on $\PI$ are of considerable interest,
as we shall see in \S\ref{SecAlgKThy}, but otherwise they will only appear peripherally in this paper.   
The paper \cite{GMM} works throughout in that more general context.  The previous result will not be used directly, 
but it is the key underpinning for the results of the next section.

\subsection{The functor $\sC\!at(\tilde{G},-)$}
The functor $\sC\!at(\tilde{G},-)$ from $G$-categories to
$G$-cate\-gories is a right adjoint (\ref{adjoint}), hence it preserves limits 
and in particular products.  The projection $\tilde{G}\rtarr \ast$ to 
the trivial $G$-category induces a natural map 
\begin{equation}\label{iota}
 \io\colon \sA = \sC\!at(\ast,\sA) \rtarr \sC\!at(\tilde{G},\sA).
\end{equation}
The map $\io$ is not an equivalence of $G$-categories in general \cite[4.19]{GMM}, but
the functor $\sC\!at(\tilde{G},-)$ is idempotent in the sense that the following result holds.

\begin{lem}\mylabel{idem} For any $G$-category $\sA$, 
\[ \io\colon \sC\!at(\tilde{G},\sA) \rtarr  
\sC\!at(\tilde{G},\sC\!at(\tilde{G},\sA))  \]
is an equivalence of $G$-categories.
\end{lem}
\begin{proof} This follows from the adjunction (\ref{adjoint}) using that 
the diagonal $\tilde{G}\rtarr \tilde{G}\times\tilde{G}$ is 
an equivalence with inverse given by either projection and
that the specialization of $\io$ here is induced by the first projection.
\end{proof}

\begin{lem}\cite[3.7]{GMM} Let $\LA$ be a subgroup of $G\times \PI$. The $\LA$-fixed
category $\sC\!at(\tilde G,\tilde\PI)^{\LA}$ is empty if $\LA\cap \PI\neq e$ and is 
nonempty and chaotic if $\LA\cap \PI = e$.
\end{lem}

With $G$ acting trivially on $\PI$, let $H^1(G;\Pi)$ denote the set of isomorphism classes
of homomorphisms $\al\colon G\rtarr \PI$.  Equivalently, it is the set of $\PI$-conjugacy
classes of subgroups $\LA = \{(g,\al(g))\,|\, g\in G\}$ of $G\times \PI $.  Define 
$\PI^{\al}\subset \PI$ to be the subgroup of elements $\si$ that commute with $\al(g)$
for all $g\in G$. 

\begin{thm}\cite[4.14, 4.18]{GMM}\mylabel{fixedCat}  For $H\subset G$, The $H$-fixed category 
$\sC\!at(\tilde G,\PI)^H$ is equivalent to the coproduct of the groups 
$\PI^{\al}$ (regarded as categories), where the coproduct runs over $[\al]\in H^1(H;\PI)$.
\end{thm}

\begin{defn} Define $E(G,\PI) = |\sC\!at(\tilde G,\tilde\PI)|$ and 
$B(G,\PI) = |\sC\!at(\tilde G,\PI)|$.
Let 
$$p\colon E(G,\PI) \rtarr B(G,\PI)$$ be induced by the passage to orbits 
functor $\tilde{\PI}\rtarr \mathbf\PI$.
\end{defn}

\begin{thm}\cite[3.11, 4.23, 4.24]{GMM}\mylabel{UniPrin} Let $\PI$ be a discrete or compact Lie group and let $G$ be a discrete group. 
Then $p \colon E(G,\PI)\rtarr B(G,\PI)$ 
is a  universal principal $(G,\PI)$-bundle. For a subgroup $H$ of $G$,
$$B(G, \PI)^H \htp \coprod B(\Pi^\al),$$
where the union runs over $[\al]\in H^1(H;\PI)$. 
\end{thm}

\subsection{$E_{\infty}$ operads of $G$-categories}\label{sec:CatEin}

The definition of an $E_\infty$-operad of $G$-spaces given in \S\ref{sec:Einfty} has the following categorical analogue.

\begin{defn}\mylabel{EinfGCatoperad} 
An $E_{\infty}$ 
operad $\sO_G$ of (topological) $G$-categories is an operad in the cartesian monoidal category $G\sC\!at$ such that $|\sO_G|$ 
is an $E_{\infty}$ operad of $G$-spaces. We say
that $\sO_G$ is reduced if $\sO_G(0)$ is the trivial category. In practice, the $\sO_G(j)$ are groupoids.
\end{defn}

The proof of \myref{Fixedptop} works just as well to give the following analogue.

\begin{lem}\mylabel{FixedptCatop} 
Let $\sO_G$ be
an $E_{\infty}$ operad of $G$-categories. Then $\sO =(\sO_G)^G$ is an 
$E_{\infty}$ operad of categories. If $\sA$ is an $\sO_G$-category,
then $\sA^G$ is an $\sO$-category. 
\end{lem}

\section{Categorical philosophy: what is a permutative $G$-category?}\label{Sec2}
\subsection{Naive permutative $G$-categories}
We have a notion of a monoidal category $\sA$ internal to a cartesian monoidal
category $\sV$. It is a category internal to $\sV$ together with a coherently
associative and unital product $\sA\times \sA\rtarr \sA$.
It is strict monoidal if the product is strictly associative 
and unital.  It is symmetric monoidal if it has an equivariant symmetry isomorphism 
satisfying the usual coherence 
properties.  A functor $F:\sA\rtarr \sB$ between symmetric monoidal categories is strict 
monoidal if $F(A\otimes A') = FA\otimes FA'$ for $A$,$A'\in \sA$ and $FI = J$, where 
$I$ and $J$ are the unit objects of $\sA$ and $\sB$.

A permutative category is a symmetric strict monoidal category.\footnote{In interesting examples, the product cannot be strictly commutative.}  Taking
$\sV$ to be $\sU$, these are the topological permutative categories. Taking $\sV$
to be $G\sU$, these are the {\em naive} topological permutative $G$-categories.

Nonequivariantly, there is a standard $E_{\infty}$ operad of spaces that is
obtained by applying the classifying space functor to an $E_{\infty}$ 
operad $\sP$ of categories.  The following definition goes back to 
Barratt and Eccles, thought of simplicially \cite{BarrEc}, and to 
\cite{MayPerm}, thought of categorically.

\begin{defn}\mylabel{operadO} We define an $E_{\infty}$ operad $\sP$ of categories. 
Let $\sP(j) = \tilde{\SI}_j$. Since $\SI_j$ acts freely and $\tilde{\SI}_j$ is chaotic, the classifying space $|\sP(j)|$ 
is $\SI_j$-free and contractible, as required of an
$E_{\infty}$ operad. The structure maps
\[ \ga\colon \tilde{\SI}_k\times \tilde{\SI}_{j_1}\times \cdots \times \tilde{\SI}_{j_k} \rtarr \tilde{\SI}_j,\]
where $j = j_1+ \cdots + j_k$, are dictated on objects by the definition of an operad. If we view the object sets of 
the $\sP(j)$ as discrete categories (identity morphisms only), then they form the associativity operad $\sM$.  
\end{defn}

We can define $\sM$-algebras and $\sP$-algebras in $\sC\!at$ or in $G\sC\!at$. In 
the latter case, we regard $\sM$ and $\sP$ as operads with trivial
$G$-action. The following result characterizes naive permutative $G$-categories
operadically. The proof is easy \cite{MayPerm}. 

\begin{prop}\mylabel{RecCat} The category of strict monoidal $G$-categories and strict monoidal $G$-functors is 
isomorphic to the category of $\sM$-algebras in $G\sC\!at$.  The category of naive permutative $G$-categories and 
strict symmetric monoidal $G$-functors is isomorphic to the category of $\sP$-algebras in $G\sC\!at$.
\end{prop}

The term ``naive'' is appropriate since naive permutative $G$-categories give rise to naive $G$-spectra on application 
of an infinite loop space machine. Genuine permutative $G$-categories need more structure, especially precursors of transfer maps, to give rise 
to genuine $G$-spectra. Nonequivariantly, there is no distinction.

\subsection{Genuine permutative $G$-categories}\label{sec:GenPermGCat}

The following observation will play a helpful role in our work.
Recall the natural map $\io\colon \sA\rtarr \sC\!at(\tilde G,\sA)$  of (\ref{iota}).

\begin{lem}  For any space $X$ regarded as a $G$-trivial $G$-space,
$\io\colon \tilde{X} \rtarr \sC\! at(\tilde G,\tilde{X})$ is the inclusion of the 
$G$-fixed category $G\sC\! at(\tilde G,\tilde{X})$.
\end{lem}
\begin{proof} Since $\tilde{X}$ is chaotic, functors $\tilde G\rtarr \tilde{X}$ 
are determined by their object map $G\rtarr X$ and are $G$-fixed if and 
only if the object map factors through $G/G = \ast$.
\end{proof}

\begin{defn} Let $\sP_G$ be the (reduced) operad of $G$-categories whose 
$j^{th}$ $G$-category
is $\sP_G(j) = \sC\! at(\tilde G,\sP(j))$, where $\sP(j) =\tilde{\SI}_j$ is 
viewed as a $G$-category with trivial $G$-action and is given its usual right 
$\SI_j$-action. The unit in $\sP_G(1)$ is the unique functor from $\tilde{G}$ 
to the trivial category $\sP(1)=\sP_G(1)$.  The structure maps $\ga$ of $\sP_G$ 
are induced from those of $\sP$, using that the functor $\sC\! at(\tilde G,-)$ 
preserves products.
By \myref{UniPrin}, $\sP_G$ is an $E_{\infty}$ operad of $G$-categories.
The natural map $\io$  of (\ref{iota}) induces an inclusion
$\io\colon \sP =(\sP_G)^G  \rtarr \sP_G$ of operads of $G$-categories.
\end{defn}

\begin{defn} A {\em genuine} permutative $G$-category is a $\sP_G$-algebra in 
$G\sC\!at$. A map of genuine permutative $G$-categories is a map of $\sP_G$-algebras.
\end{defn}

We usually call these $\sP_G$-categories. We have an immediate source 
of examples. Let $\io^*$ be the functor from genuine permutative $G$-categories 
to naive permutative $G$-categories that is obtained by restricting actions 
by $\sP_G$ to its suboperad $\sP$. 

\begin{prop}\mylabel{ScriptG} The action of $\sP$ on a naive permutative 
$G$-category $\sA$ induces an action of $\sP_G$ on $\sC\!at(\tilde G,\sA)$. Therefore $\sC\!at(\tilde G,-)$ restricts 
to a functor from naive permutative $G$-categories to 
genuine permutative $G$-categories.  
\end{prop}
\begin{proof} This holds since the functor $\sC\!at(\tilde G,-)$ preserves products.
\end{proof}

\begin{prop} The map $\io$ of (\ref{iota})
restricts to a natural map $\sA\rtarr \io^*\sC\!at(\tilde G,\sA)$ of naive permutative 
$G$-categories, and $\io$ is an equivalence when $\sA = \io^*\sC\!at(\tilde G,\sB)$ for 
a naive permutative $G$-category $\sB$. 
\end{prop}
\begin{proof} Since $\io$ is induced by the projection $\tilde G\rtarr \tilde{e} =\ast$, 
the first claim is clear, and the second holds by \myref{idem}.
\end{proof}

As noted before, the map $\io\colon \sA\rtarr \io^* \sC\!at(\tilde G,\sA)$ is not an equivalence 
in general \cite[4.19]{GMM}. The $\sP_G$-categories of interest in this paper are of the form $\sC\!at(\tilde G,\sA)$ for a 
naive permutative $G$-category $\sA$. In fact, we do not yet know how to construct 
other examples, although we believe that they exist.

\begin{rem} Shimakawa \cite[p. 256]{Shim1} introduced the $E_\infty$ operad $\sP_G$ under the 
name $\sD$ and demonstrated the first part of \myref{ScriptG}. 
\end{rem} 

\begin{rem}\mylabel{coend}  One might hope that $(\sC\!at(\tilde G,-),\io^*)$ is an adjoint pair.  However, regarding $\io^*$ 
monadically as the forgetful functor from $\bP_G$-algebras to $\bP$-algebras, its left adjoint 
is the coend that sends a naive permutative $G$-category $\sA$ to the genuine permutative
$G$-category $\bP_G\otimes_{\bP}\sA$, which is the coequalizer in $G\sC\!at$ of the maps 
$\xymatrix{\bP_G\bP\sA\ar[r]<.5ex> \ar[r]<-.5ex> &  \bP_G\sA}$ 
induced by the action map $\bP\sA\rtarr \sA$ and by the map $\bP_G\bP \rtarr \bP_G\bP_G\rtarr \bP_G$ 
induced by the inclusion $\bP\rtarr \bP_G$ and the
product on $\bP_G$.  The universal property of the coequalizer gives a natural map
$$\tilde{\io} \colon \bP_G\otimes_{\bP}\sA\rtarr \sC\!at(\tilde G,\sA)$$ of genuine permutative $G$-categories 
that restricts to $\io$ on $\sA$, but $\tilde{\io}$ is not an isomorphism. We shall say a bit
more about this in \myref{absence}. 
\end{rem}

\subsection{$E_{\infty}$ $G$-categories}\label{Sec2.3}
We can generalize the notion of a genuine permutative $G$-category by
allowing the use of $E_{\infty}$ operads other than $\sP_G$. In fact, thinking 
as algebraic topologists rather than category theorists, there is no
need to give the particular $E_{\infty}$ operad $\sP_G$ a privileged role.

\begin{defn}\mylabel{permcat} An $E_{\infty}$ $G$-category $\sA$ is a $G$-category together with an action of some $E_{\infty}$ operad $\sO_G$ of $G$-categories. The classifying 
space $B\sA =|\sA|$ is then an $|\sO_G|$-space and thus an $E_{\infty}$ $G$-space.
\end{defn}

We may think of $E_{\infty}$ $G$-categories as generalized kinds of genuine permutative $G$-categories. 
The point of the generalization is that we have interesting examples 
of $E_{\infty}$ operads of $G$-categories with easily recognizable algebras. We shall
later define $E_{\infty}$ operads $\sV_G$, $\sV^\times_G$, and $\sW_G$ that are interrelated 
in a way that illuminates the study of multiplicative structures. 

Observe that $\sP_G$-algebras, like nonequivariant permutative categories, have a canonical product, whereas $E_{\infty}$ $G$-categories over other operads do not. 
The general philosophy of operad theory is that algebras over an operad $\sC$ in any suitable category $\sV$ have $j$-fold operations parametrized by the objects 
$\sC(j)$. Homotopical properties of $\sC$ relate these operations. In general, in an $E_{\infty}$ space, there is no preferred choice of a product on its underlying $H$-space, and none is relevant to the applications; $E_{\infty}$ $G$-categories work similarly.

\begin{rem} Symmetric monoidal categories occur more often ``in nature'' than permutative categories. We have not specified a notion of a genuine symmetric monoidal $G$-category in this paper. One approach is to apply the construction $\sC\! at(\tilde{G},-)$ to the tree operad that defines symmetric monoidal categories.  Another approach, which we find more useful, is to define a genuine symmetric monoidal $G$-category to be a pseudoalgebra over
$\sP_G$.   That approach is developed and applied in the categorical sequels \cite{GMMOAdd, GMMOMult}.  We shall not pursue the topic further here.
A first comparison between symmetric monoidal $G$-categories and (genuine) $G$-symmetric monoidal categories, whose definition is a priori quite different, is given in Hill and Hopkins  \cite[Section 3.2]{HH}, but work in progress shows that there is a good deal more to be said about that comparison and about the comparison between these notions and Tambara functors that is given in \cite[Section 5.1]{HH}. 
\end{rem}

Up to homotopy, any two choices of $E_{\infty}$ operads give rise to equivalent 
categories of $E_{\infty}$ $G$-spaces.  To see that, we apply the trick from \cite{MayGeo} of using products of operads to transport operadic algebras from 
one $E_{\infty}$ operad to another.  The product of operads $\sC$ and $\sD$ in any cartesian monoidal category $\sV$ is given by 
\[   (\sC\times \sD)(j) = \sC(j)\times \sD(j), \]
with the evident permutations and structure maps. With the choices of $\sV$ of
interest to us, the product of $E_{\infty}$ operads is an $E_{\infty}$ operad.
The projections 
\[ \sC\ltarr \sC\times \sD \rtarr \sD \] 
allow us to construct
$(\sC\times \sD)$-algebras in $\sV$ from either $\sC$-algebras or $\sD$-algebras 
in $\sV$, by pullback of action maps along the projections.  

More generally, for any map $\mu\colon \sC\rtarr \sD$ of operads in $\sV$, 
the pullback functor $\mu^*$ from $\sD$-algebras to $\sC$-algebras has a 
left adjoint pushforward functor $\mu_!$ from $\sC$-algebras to $\sD$-algebras. 
One can work out a homotopical comparison model 
categorically.
Pragmatically, use of the two-sided bar construction 
as in \cite{MayGeo, Rant1} gives all that 
is needed. One redefines $\mu_! X = B({\bD},{\bC},X)$, where $\bC$ and $\bD$ are 
the monads whose algebras are the $\sC$-algebras and $\sD$-algebras.\footnote{Of course,
this is an abuse of notation, since $\mu_!$ here is really a derived functor.}
In spaces, or equally well $G$-spaces, $\mu^*$ and $\mu_!$ give inverse equivalences of homotopy categories between 
$\sC$-algebras and $\sD$-algebras when $\sC$ and $\sD$ are $E_{\infty}$-operads.

Starting with operads in $\sC\!at$ or in $G\sC\!at$ we can first apply the classifying space functor and then apply this trick.  The conclusion is that all $E_{\infty}$ categories and $E_{\infty}$ $G$-categories give equivalent inputs for infinite loop space machines. In particular, for example, letting $\mathbf{O}_G$, $\mathbf{P}_G$, and 
$\mathbf{O_G\times P_G}$ denote the monads in the category of $G$-spaces whose algebras 
are $|\sO_G|$-algebras, $|\sP_G|$-algebras, and $|\sO_G\times \sP_G|$-algebras, we see that after passage to classifying spaces, every 
$\mathbf{P}_G$-algebra $Y$ determines an $\mathbf{O}_G$-algebra 
$X= B(\mathbf{O}_G,\mathbf{O_G\times P_G},Y)$ such that $X$ and $Y$ are weakly equivalent as $(\mathbf{O_G\times P_G})$-algebras 
(and conversely).  This says that for purposes of equivariant infinite loop space theory, $\sP_G$ and any other $E_{\infty}$ operad $\sO_G$ can be used interchangeably, regardless of how their algebras compare categorically.

\subsection{Equivariant algebraic $K$-theory}\label{SecAlgKThy}

The most interesting non-equivariant permutative categories are given by categories 
$\sA = \coprod{\PI_n}$, where $\{\PI_n|n\geq 0\}$ is a sequence of groups (regarded as
categories with a single object) and where the permutative structure is given by an
associative and unital system of pairings $\PI_m\times \PI_n\rtarr \PI_{m+n}$.  Then
the pairings give the classifying space $B\sA = \coprod B\PI_n$ a structure of topological monoid, 
and one definition of the algebraic $K$-groups of $\sA$ is the homotopy groups of the space $\OM B(B\sA)$.  

Equivariantly, it is sensible to replace the spaces $B\PI_n$
by the classifying $G$-spaces $B(G,\PI_n)$ and proceed by analogy. This definition of 
equivariant algebraic $K$-groups was introduced and studied calculationally in \cite{FHM}.  
It is the equivariant analogue of Quillen's original definition in
terms of the plus construction. With essentially the same level of generality,
the analogue of Quillen's definition in terms of the $Q$-construction has been
studied by Dress and Kuku \cite{DK, Kuku}.  Shimada \cite{Shim2} has given an equivariant 
version of Quillen's ``plus $=$ $Q$'' theorem in this context.

Regarding $\sA$ as a $G$-trivial naive permutative $G$-category, we see that 
the classifying $G$-space of the genuine permutative $G$-category $\sC\!at(\tilde G,\sA)$ is
the disjoint union of classifying spaces $B(G,\PI_n)$. Just as nonequivariantly, 
the functor $\OM B$ can be replaced by the zeroth space functor $\OM^{\infty}_G\bE_G$ 
of an infinite loop $G$-space machine $\bE_G$. The underlying equivariant homotopy type 
is unchanged. Therefore, we may redefine the algebraic $K$-groups 
to be the homotopy groups of the genuine $G$-spectrum $\bK_G\sA\equiv\bE_G B\sC\!at(\tilde G,\sA)$.  
Essentially the same definition is implicit in Shimakawa \cite{Shim1}, who 
focused on an equivariant version of Segal's infinite loop space machine.  
A different equivariant version of Segal's machine is developed and
compared to Shimakawa's in \cite{MMO}.   It is generalized categorically 
in \cite{GMMOAdd, GMMOMult}. 

Applying the functor $\sC\!at(\tilde G,-)$ to naive permutative $G$-categories $\sA$ with 
non-trivial $G$-actions gives more general input for equivariant algebraic $K$-theory 
than has been studied in the literature.  This allows for $G$-actions on the groups 
$\PI_n$, and we then replace $B(G,\PI_n)$ by classifying $G$-spaces $B(G,(\PI_n)_G)$ 
for the $(G,(\PI_n)_G)$-bundles associated to the split extensions $\PI_n\rtimes G$. 
Such classifying spaces are studied in \cite{GMM}.  Alternative but equivalent constructions
of the associated $G$-spectra $\bK_G(\sA)$ are given in \S\ref{SecKGB} and \S\ref{SecKGBtoo}
below. The resulting generalization of equivariant algebraic $K$-theory is studied 
in \cite{Merling}.  

\subsection{The recognition principle for permutative $G$-categories}\label{SecKGB}
We may start with any $E_{\infty}$ operad $\sO_G$ of $G$-categories and apply
the classifying space functor to obtain an $E_{\infty}$ operad $|\sO_G|$ 
of $G$-spaces.  If $\sO_G$ acts on a category $\sA$, then $|\sO_G|$ acts on
$|\sA|=B\sA$.  We can replace $|\sO_G|$ by its product with the Steiner operads 
$\sK_V$ or with the Steiner operad $\sK_U$ and apply the functor $\bE_G^{\sS}$ 
or $\bE_G^{\sS\!p}$ to obtain a (genuine) associated $G$-spectrum, which we denote
ambiguously by $\bE_G(B\sA)$.

\begin{defn} Define the (genuine) algebraic $K$-theory $G$-spectrum of an $\sO_G$-category $\sA$ by $\bK_G(\sA) = \bE_G(B\sA)$.
\end{defn}

We might also start with an operad $\sO$ of categories such that $|\sO|$ is an
$E_{\infty}$ operad of spaces and regard these as $G$-objects with trivial 
action.  Following up the previous section, we then have the following related but less interesting notion.

\begin{defn} Define the (naive) algebraic $K$-theory $G$-spectrum of an $\sO$-category 
$\sA$ by $\bK(\sA) = \bE(B\sA)$.
\end{defn}

Until \S\ref{SecPQR}, we restrict attention to the cases $\sO_G = \sP_G$ and
$\sO = \sP$, recalling that the $\sP_G$-categories are the genuine permutative 
$G$-categories, the $\sP$-categories are the naive permutative $G$-categories,
and the inclusion $\io\colon \sP\rtarr \sP_G$ induces a forgetful functor 
$\io^*$ from genuine to naive permutative $G$-categories.
Since the classifying space functor commmutes with products, passage to 
fixed points, and the functors $\io^*$, Theorems \ref{SpFour}, \ref{prodcom}, and 
\ref{SpSix} have the following immediate corollaries. The first was promised in 
\cite[Thm 2.2]{GM2}.

\begin{thm}\mylabel{SpFour2} For $\sP_G$-categories $\sA$, there is a natural weak equivalence of spectra
$$\bK(\sA^G) \rtarr (\bK_G \sA)^G.$$
\end{thm}

\begin{thm}\mylabel{prodcom2} Let $\sA$ and $\sB$ be $\sP_G$-categories. Then the map 
\[ \bK_G(\sA\times \sB) \rtarr \bK_G \sA \times \bK_G\sB \]
induced by the projections is a weak equivalence of $G$-spectra.
\end{thm}

\begin{thm}\mylabel{SpSix2} For $\sP_G$-categories $\sA$, there is a natural
weak equivalence of naive $G$-spectra $\bK \io^*\sA \rtarr i^*\bK_G\sA$. 
\end{thm}

The algebraic $K$-groups of $\sA$ are defined to be the groups 

\begin{equation} K_*^H\sA = \pi_*^H(\bK\io^*\sA) \iso \pi_*^H(\bK_G\sA).
\end{equation}

We are particularly interested in examples of the form $\sC\!at(\tilde G,\sA)$, where
$\sA$ is a naive permutative $G$-category.  As noted in \myref{ScriptG}, we then
have a natural map $\io\colon \sA \rtarr \io^*\sC\!at(\tilde G,\sA)$ of naive permutative 
$G$ categories. We can pass to classifying spaces and apply the functor $\bE$ 
to obtain a natural map
\begin{equation}
\xymatrix@1{ \bK \sA \ar[r]^-{\bK\io} 
& \bK \io^*\sC\!at(\tilde G,\sA) \ar[r]^-{\mu}_-{\htp} & i^*\bK_G\sC\!at(\tilde G,\sA).\\} 
\end{equation}
This map is a weak equivalence when $\io^H\colon \sA^H \rtarr (\io^*\sC\!at(\tilde G,\sA))^H$
is an equivalence of categories for all $H\subset G$.   
The following example where this holds is important in equivariant algebraic 
$K$-theory.

\begin{exmp} Let $E$ be a Galois extension of $F$ with Galois group $G$
and let $G$ act entrywise on $GL(n,E)$ for $n\geq 0$.  The disjoint union
of the $GL(n,E)$ is a naive permutative $G$-category that we denote by $GL(E_G)$.
Its product is given by the block sum of matrices. Write $GL(R)$ for the 
nonequivariant permutative general linear category of a ring $R$.  
As we proved in \cite[4.20]{GMM}, Serre's version of Hilbert's Theorem 90 implies that
\[ \io^H\colon GL(E^H) \iso GL(E_G)^H \rtarr (\io^* \sC\!at(\tilde G,GL(E_G))^H \] 
is an equivalence of categories for $H\subset G$.  This identifies the
equivariant algebraic $K$-groups of $E$ with the nonequivariant
algebraic $K$-groups of its fixed fields $E^H$.
\end{exmp}   

\begin{rem}\mylabel{absence} In the list above of theorems about permutative categories, 
a consequence of \myref{Trouble} is conspicuous by its absence.  Letting
$\io_!\sA \equiv \bP_G\otimes_{\bP}\sA$ denote the left adjoint of $\io^*$, as defined in \myref{coend}, 
one might hope that $B\io_!\sA$ is equivalent as an $|\sP_G|$-space to $\io_!B\sA$ for a naive 
permutative $G$-category $\sA$. We do not know whether or not that is true.    
\end{rem}

\section{The free $|\sP_G|$-space generated by a $G$-space $X$}\label{Sec3}

The goal of this section is to obtain a decomposition of the fixed point categories of free 
permutative G-categories. This decomposition will be the crux of the proof of the tom Dieck 
splitting theorem given in \S5.2.

\subsection{The monads $\bP_G$ and $\mathbf{P}_G$ associated to $\sP_G$}
Recall that $\sP_G$ is reduced. In fact, both $\sP_G(0)$ and 
$\sP_G(1)$ are trivial categories. As discussed for spaces in \cite[\S4]{Rant1}, 
there are two monads on $G$-categories whose algebras are the genuine 
permutative $G$-categories. The unit object of an $\sP_G$-category can be preassigned, 
resulting in a monad $\bP_{G}$ on based $G$-categories, or it can be viewed as part of 
the $\sP_G$-algebra structure, resulting in a monad $\bP_{G+}$ on unbased $G$-categories.  
Just as in \cite{Rant1}, these monads are related by 
$$\bP_{G}(\sA_+) \iso \bP_{G+}\sA,$$
where $\sA_+ = \sA\amalg \ast$ is obtained from an unbased $G$-category $\sA$ by adjoining
a disjoint copy of the trivial $G$-category $\ast$.  Explicitly,  
\begin{equation}\label{opAone}
\bP_G(\sA_+) = \coprod_{j\geq 0} \sP_G(j)\times_{\SI_j} \sA^j. 
\end{equation}
The term with $j=0$ is $\ast$ and accounts for the copy of $\ast$ on the left. 
The unit $\et\colon \sA\rtarr \bP_G(\sA_+)$
identifies $\sA$ with the term with $j=1$. The product
$\mu\colon \bP_G\bP_G \sA_+\rtarr \bP_G\sA_+$ is induced by the operad structure 
maps $\ga$. We are only concerned with based $G$-categories that can be written
in the form $\sA_+$.

Since we are concerned with the precise point-set relationship between an infinite 
loop space machine defined on $G$-categories and suspension $G$-spectra, it is useful
to think of (unbased) $G$-spaces $X$ as categories.  Thus we also let $X$ denote the topological 
$G$-category  with object and morphism $G$-space $X$ and with $I$, $S$, $T$, 
and $C$ all given by the identity map $X\rtarr X$; this makes sense for $C$ since we can
identify $X\times_X X$ with $X$.  We can also identify the classifying $G$-space $|X|$ with $X$. 

By specialization of (\ref{opAone}), we have an identification of (topological) $G$-categories
\begin{equation}\label{opXone}
\bP_G(X_+) = \coprod_{j\geq 0} \sP_G(j)\times_{\SI_j} X^j.
\end{equation}

The following illuminating result gives another description of $\bP_G(X_+)$.

\begin{prop} For $G$-spaces $X$, there is a natural isomorphism of genuine
permutative $G$-categories
\[  \bP_G(X_+) = \coprod_j \sC\!at(\tilde G,\tilde{\SI}_j)\times_{\SI_j} X^j
\rtarr \coprod_j \sC\! at(\tilde G,\tilde{\SI}_j\times_{\SI_j} X^j) = \sC\!at(\tilde G,\bP(X_+)).\]
\end{prop}
\begin{proof} For each $j$ and for $(G\times \SI_j)$-spaces $Y$, such as $Y= X^j$,
we construct a natural isomorphism of $(G\times \SI_j)$-categories 
\[ \sC\!at(\tilde G,\tilde{\SI}_j) \times Y
\rtarr \sC\!at(\tilde G,\tilde{\SI}_j\times Y).\]
Here $Y$ is viewed as the constant $(G\times \SI_j)$-category at $Y$. The target is
\[ \sC\!at(\tilde G,\tilde{\SI}_j) \times \sC\!at(\tilde G,Y). \]
Since there is a map between any two objects of $\tilde{G}$ but the only maps
in $Y$ are identity maps $i_y\colon y\rtarr y$ for $y\in Y$, the only functors $\tilde G\rtarr Y$
are the constant functors $c_y$ at $y\in Y$ and the only natural transformations between 
them are the identity transformations $\id_y\colon c_y\rtarr c_y$.  Sending $y$ to 
$c_y$ on objects and $i_y$ to $\id_y$ on morphisms specifies an identification of 
$(G\times \SI_j)$-categories $Y\rtarr \sC\!at(\tilde G,Y)$. The product of the identity 
functor on $\sC\!at(\tilde G,\tilde{\SI}_j)$ and this identification gives the desired natural
equivalence. With $Y=X^j$, passage to orbits over $\SI_j$ gives the $j^{th}$ component of the 
claimed isomorphism of $G$-categories.  It is an isomorphism of $\sP_G$-categories since on 
both sides the action maps are induced by the structure maps of the operad $\sP$.
\end{proof}

Recall that we write $\mathbf{P}_G$ for the monad on based $G$-spaces associated to the operad $|\sP_G|$.  
Thus $\mathbf{P}_G(X_+)$ is the free $|\sP_G|$-space generated by the $G$-space $X$.

\begin{prop} For $G$-spaces $X$, there is a natural isomorphism
\[
\mathbf{P}_G(X_+) = \coprod_{j\geq 0} |\sP_G(j)|\times_{\SI_j} X^j \iso 
|\bP_G X_+|.
\]
\end{prop}
\begin{proof} For a $(G\times \SI_j)$-space $Y$ viewed as a $G$-category,
the nerve $NY$ can be identified with the constant simplicial space $Y_*$ with 
$Y_q=Y$.   The nerve functor $N$ does not commute with passage to orbits 
in general, but arguing as in \cite[\S2.3]{GMM} we see that 
\[ N(\sP_G(j)\times_{\SI_j} Y)\iso  (N\sP_G(j))\times_{\SI_j} Y_* = 
N(\sP_G(j)\times_{\SI_j} NY).\]
Therefore the classifying space functor commutes with coproducts, products,  
and the passage to orbits that we see here.
\end{proof}

\subsection{The identification of $(\bP_GX_+)^G$}  
The functor $|-|$ commutes with passage to $G$-fixed points, and we 
shall prove the following identification. Let $\bP$ denote the monad on 
nonequivariant based categories associated to the operad $\sP$ that defines 
permutative categories.

\begin{thm}\mylabel{ident} For $G$-spaces $X$, there is a natural equivalence 
of $\sP$-categories
\[ \bP_G(X_+)^G\htp \prod_{(H)}\bP(\widetilde{WH}\times_{WH}X^H)_+, \]
where $(H)$ runs over the conjugacy classes of subgroups of $G$ and $WH = NH/H$. 
\end{thm}

We are regarding $\sP$ as the suboperad $(\sP_G)^G$ of $\sP_G$, and the identification
of categories will make clear that the identification preserves the action by $\sP$.
Of course,
\begin{equation}\label{catone}
\bP_G( X_+)^G = \coprod_{j\geq 0} (\sP_G(j)\times_{\SI_j} X^j)^G 
\end{equation}
and
\begin{equation}\label{cattwo}
\bP(\widetilde{WH}\times_{WH}X^H)_+ = 
\coprod_{k\geq 0}\ \tilde{\SI}_k\times_{\SI_k}(\widetilde{WH}\times_{WH}X^H)^k.
\end{equation}

We shall prove \myref{ident} by identifying both (\ref{catone}) and (\ref{cattwo})
with a small (but not skeletal) model $\sF_G(X)^G$ for the category of finite $G$-sets 
over $X$ and their isomorphisms over $X$. We give the relevant definitions and 
describe these identifications here, and we fill in the 
easy proofs in \S\ref{Sec3.3} and \S\ref{Sec3.4}. 

A homomorphism $\al\colon G\rtarr \SI_j$ is equivalent to the left 
action of $G$ on the set $\mathbf{j}=\{1,\cdots,j\}$ specified by
$g\cdot i = \al(g)(i)$ for $i\in \mathbf{j}$. Similarly,
an anti-homomorphism $\al\colon G\rtarr \SI_j$ is equivalent 
to the right action of $G$ on $\mathbf{j}$ specified by $i\cdot g = \al(g)(i)$ or,
equivalently, the left action specified by $g\cdot i = \al(g^{-1})(i)$; of course,
if we set $\al^{-1}(g) = \al(g)^{-1}$, then $\al^{-1}$ is a homomorphism.  
We focus on homomorphisms and left actions, and we denote such $G$-spaces 
by $(\mathbf{j},\al)$. When we say that $A$ is a finite $G$-set, we agree to 
mean that $A = (\mathbf{j},\al)$ for a given homomorphism $\al\colon G\rtarr \SI_j$.  
That convention has the effect of fixing a small groupoid $G\sF$ equivalent to the 
groupoid of all finite $G$-sets and isomorphisms of finite $G$-sets. By a $j$-pointed $G$-set, we mean a $G$-set
with $j$ elements. 

\begin{defn}\mylabel{Gsets}  Let $X$ be a $G$-space and $j\geq 0$. 
\begin{enumerate}[(i)]
\item Let $\sF_G(j)$ be the $G$-groupoid whose objects are the $j$-pointed $G$-sets 
$A$ and whose morphisms $\si\colon A\rtarr B$ are the bijections, 
with $G$ acting by conjugation.  Then $\sF_{G}(j)^G$ is the category with the same objects 
and with morphisms the isomorphisms of $G$-sets $\si\colon A\rtarr B$. 

\vspace{2mm}

\item Let $\sF_G(j,X)$ be the $G$-groupoid whose objects are
the maps (not $G$-maps) $p\colon A\rtarr X$ and whose morphisms $f\colon p\rtarr q$,
$q\colon B\rtarr X$, are the bijections $f\colon A\rtarr B$ such that 
$q\com f = p$; $G$ acts by conjugation on all maps $p$, $q$, and $f$.
We view $\sF_G(j,X)^G$ as the category of $j$-pointed $G$-sets over 
$X$ and isomorphisms of $j$-pointed $G$-sets over $X$. 

\vspace{2mm}

\item Let $\sF_G= \coprod_{j\geq 0} \sF_G(j)$ and  $\sF_G(X) = \coprod_{j\geq 0} \sF_G(j,X)$. 

\vspace{2mm}

\item Let $\sF_G^{\sP}(j)$ be the full $G$-subcategory of $G$-fixed objects of 
$\sP_G(j)/\SI_j$ and let $\sF_G^{\sP}(j,X)$ be the full $G$-subcategory of $G$-fixed objects of $\sP_G(j)\times_{\SI_j} X^j$.
Then 
$$\sF_G^{\sP}(j)^G =(\sP_G(j)/\SI_j)^G \ \ \text{and} \ \ \sF_{G}^{\sP}(j,X)^G 
= (\sP_G(j)\times_{\SI_j} X^j)^G.$$ 
\end{enumerate}
\end{defn}

In \S\ref{Sec3.3}, we prove that the right side of (\ref{catone}) 
can be identified with $\sF_G(X)^G$.

\begin{thm}\mylabel{CATone} There is a natural isomorphism of permutative categories
\[  (\bP_G(X_+))^G = \coprod_{j\geq 0} \sF_G^{\sP}(j,X)^G 
\iso \coprod_{j\geq 0} \sF_G(j,X)^G = \sF_G(X)^G.\]
\end{thm}

We will prove an equivariant variant of this result, before passage to fixed points,
in \myref{identtoo}. In \S\ref{Sec3.4}, we prove that the right side of (\ref{cattwo}) 
can also be identified with $\sF_G(X)^G$.  At least implicitly, this identification of fixed point
categories has been known since the 1970's; see for example Nishida \cite[App. A]{Nishida}.

\begin{thm}\mylabel{CATtwo} There is a natural equivalence of categories
\[ \prod_{(H)}\coprod_{k\geq 0}\ \tilde{\SI}_k\times_{\SI_k}(\widetilde{WH}\times_{WH}X^H)^k\rtarr \coprod_{j\geq 0} \sF_G(j,X)^G = \sF_G(X)^G.\]
\end{thm}

These two results prove \myref{ident}.

\begin{rem}\mylabel{subtle} With our specification of finite $G$-sets as 
$A=(\mathbf{j},\al)$, the disjoint union of $A$ and $B =(\mathbf{k},\be)$ 
is obtained via the obvious identification of $\mathbf{j}\coprod \mathbf{k}$ 
with $\mathbf{j+k}$.  The disjoint union of finite $G$-sets over a $G$-space 
$X$ gives $\sF_G(X)$ a structure of naive permutative $G$-category.
By \myref{CATone}, its fixed point category $\sF_G(X)^G$ is a $\sP$-category 
equivalent to $(\bP_G(X_+))^G$.  One might think that $\sF_G(X)$ is a genuine
permutative $G$-category equivalent 
to the free $\sP_G$-category $\sP_G(X_+)$. However, its $H$-fixed subcategory for 
$H\neq G$ is not equivalent to  $\sF_H(X)^H$, and one cannot expect an
action of $\sP_G$ (or any other $E_{\infty}$ $G$-operad) on $\sF_G(X)$. 
To see the point, let $G$ be the quaternion group of order $8$, $Q=\{\pm 1, \pm i, \pm j, \pm k\}$, 
and let $X=\ast$. Every nontrivial subgroup of $G$ contains the center $H=Z={\pm 1}$. Therefore 
the $H$-set $H$ cannot be obtained by starting with a $G$-set (a disjoint
union of orbits $G/K$) and restricting along the inclusion $H\rtarr G$. 
\end{rem}

To compare with our paper \cite{GM2}, we offer some alternative notations.

\begin{defn}\mylabel{EGO}  For an unbased $G$-space $X$, let 
$\sE_G(X) = \sE_G^{\sP}(X) = \bP_G(X_+)$.  
It is a genuine permutative $G$-category, and its
$H$-fixed subcategory $\sE_G(X)^H$ is equivalent to $\sE_H(X)^H$ and
therefore to $\sF_H(X)^H$.
\end{defn}

\begin{rem}\mylabel{altnot} In \cite{GM2}, we gave a more 
intuitive definition of a $G$-category $\sE_G(X)$. It will reappear in \S\ref{SecEGP},
where it will be given the alternative notation $\sE^{U}_G(X)$. It is acted on  
by an $E_{\infty}$ operad $\sV_G$ of $G$-categories, and, again, its fixed point 
category $\sE^{U}_G(X)^H$ is equivalent to $\sE^{U}_H(X)^H$ and therefore to $\sF_H(X)^H$.
\end{rem}

\subsection{The proof of \myref{CATone}}\label{Sec3.3}

We first use \myref{fixedCat} to identify (\ref{catone}) when $X$ is a point.  
The proof of \myref{fixedCat} compares several equivalent categories, and
anti-homomorphisms appear naturally.  To control details of equivariance,
it is helpful to describe the relevant categories implicit in our operad $\sP_G$ 
in their simplest forms up to isomorphism.  Details are in \cite[\S\S2.1, 2.2, 4.1, 4.2]{GMM}.  

\begin{lem} The objects of the chaotic $(G\times \SI_j)$-category $\sP_G(j)$ are the 
functions $\ph\colon G\rtarr \SI_j$.  The (left) action of $G$ on $\sP_G(j)$ is given 
by $(g\ph)(h) = \ph(g^{-1} h)$
on objects and the diagonal action on morphisms.
The (right) action of $\SI_j$ is given by $(\ph\si)(h) = \ph(h)\si$ on
objects and the diagonal action on morphisms. 
\end{lem}

\begin{lem} The objects of the $G$-category $\sP_{G}(j)/\SI_j$ are the functions \linebreak
$\al\colon G\rtarr \SI_j$ such that $\al(e) = e$. The morphisms $\si\colon \al \rtarr \be$ 
are the elements $\si\in \SI_j$, thought of as
the functions $G\rtarr \SI_j$ specified by $\si(h) = \be(h)\si\al(h)^{-1}$.
The composite of $\si$ with $\ta\colon \be\rtarr \ga$ is $\ta\si\colon \al\rtarr \ga$.
The action of $G$ is given on objects by 
\[ (g\al)(h) = \al(g^{-1}h)\al(g^{-1})^{-1}. \]
In particular, $(g\al)(e) = e$. The action on morphisms
is given by 
\[g(\si\colon \al\rtarr \be) = \si\colon g\al\rtarr g\be. \]
\end{lem}

\begin{lem} For $\LA\subset G\times \SI_j$, $\sP_G(j)^{\LA}$ is empty if 
$\LA\cap \SI_j\neq e$. It is a nonempty and hence chaotic subcategory of 
$\sP_{G}(j)$ if $\LA\cap \SI_j=e$. 
\end{lem} 

\begin{lem}\mylabel{pedantic} The objects of $(\sP_G(j)/\SI_j)^G$ are the anti-homomorphisms \linebreak
$\al\colon G\rtarr \SI_j$. Its morphisms $\si\colon \al\rtarr \be$ are the conjugacy relations
$\be = \si\al\si^{-1}$, where $\si\in \SI_j$. For $H\subset G$, restriction of functions gives 
an equivalence of categories $$(\sP_G(j)/\SI_j)^H\rtarr (\sP_H(j)/\SI_j)^H. $$
\end{lem}

Now return to a general $G$-space $X$.  To prove \myref{CATone}, it suffices to
prove that $(\sP_G(j)\times_{\SI_j} X^j)^G$ is isomorphic to $\sF_G(j,X)^G$ 
for all $j$. 
Passage to orbits here means that for $\ph\in \sP_G(j)$, $y\in X^j$, and $\si\in \SI_j$ 
(thought of as acting on the left on $\mathbf{j}$ and therefore on $j$-tuples of elements of $X$), 
$(\ph\si,  y) = (\ph,\si y)$ in $\sP_G(j)\times_{\SI_j} X^j$.
Observe that an object $(\ph, z_1,\cdots, z_j)\in \sP_G(j)\times_{\SI_j} X^j$
has a unique representative in the same orbit under $\SI_j$ of the form 
$(\al,x_1,\cdots,x_j)$ where $\al(e) = e$. It is obtained by replacing $\ph$ 
by $\ph \tau$, where $\ta=\ph(e)^{-1}$, and replacing $z_i$ by $x_i = z_{\ta(i)}$. 

\begin{lem}\mylabel{picky} 
An object $(\al,y) \in \sP_G(j)\times_{\SI_j} X^j$, where $\al(e) = e$ 
and $y\in X^j$, is $G$-fixed if and only if $\al\colon G\rtarr \SI_j$ is an anti-homomorphism and $\al(g^{-1})y = gy$ for all $g\in G$. 
\end{lem}
\begin{proof} Assume that $(\al,y) = (g\al,gy)$ for all $g\in G$. Then each 
$g\al$ must be in the same $\SI_j$-orbit as $\al$, where $\al$ is regarded as an 
object of $\sP_G(j)$ and not $\sP_G(j)/\SI_j$, so that $(g\al)(h) = \al(g^{-1}h)$. 
Then $(g\al)(h) = \al(h)\si$ for all $h\in G$ and some $\si\in \PI$. Taking 
$h=e$ shows that $\si = \al(g^{-1})$. The resulting formula $\al(g^{-1}h) = \al(h)\al(g^{-1})$ implies that $\al$ is an anti-homomorphism.  Now 
\[ (\al,y) = (g\al,gy) = (\al \al(g^{-1}), gy) = (\al, \al(g)gy),\]
which means that $\al(g) gy = y$ and thus $gy = \al(g^{-1})y$. 
\end{proof}

Use $\al^{-1}$ to  define a left 
action of $G$ on $\mathbf{j}$ and define $p\colon \mathbf{j}\rtarr X$ by 
$p(i) = x_i$.  Then the lemma shows that the $G$-fixed elements $(\al,y)$ are in
bijective correspondence with the maps of $G$-sets $p\colon A \rtarr X$, where
$A$ is a $j$-pointed $G$-set.
Using \myref{pedantic}, we see similarly that maps $f\colon A\rtarr B$ 
of $j$-pointed $G$-sets over $X$ correspond bijectively to morphisms in  
$(\sP_G(j)\times_{\SI_j} X^j)^G$.  These bijections specify the
required isomorphism between $\sF_G(j,X)^G$ and $(\sP_G(j)\times_{\SI_j} X^j)^G$.

\subsection{The proof of \myref{CATtwo}}\label{Sec3.4}
This decomposition is best proven by a simple thought exercise. Every finite 
$G$-set $A$ decomposes non-uniquely as a disjoint union of orbits $G/H$, and 
orbits $G/H$ and $G/J$ are isomorphic if and only if $H$ and $J$ are conjugate.  
Choose one $H$ in each conjugacy class.  Then $A$ decomposes uniquely as the 
disjoint union of the $G$-sets $A_H$, where $A_H$ is the set of elements of $A$ 
with isotropy group conjugate to $H$.  This decomposes the category $G\sF \equiv (\sF_G)^G$ as the 
product over $H$ of the categories $G\sF(H)$ of finite $G$-sets all of whose isotropy groups are conjugate to $H$.

In turn, $G\sF(H)$ decomposes uniquely as the coproduct over $k\geq 0$ of the 
categories $G\sF(H,k)$ whose objects are isomorphic to the disjoint union, denoted 
$kG/H$, of $k$ copies of $G/H$.  Up to isomorphism, $kG/H$ is the only object of
$G\sF(H,k)$. The automorphism group of the $G$-set $G/H$ is $WH$, hence the 
automorphism group of $kG/H$ is the wreath product $\SI_k\int WH$. Viewed
as a category with a single object, we may identify this group with the category
$\tilde{\SI}_k\times_{\SI_k} (WH)^k$. This proves the following result.

\begin{prop} The category $G\sF$ is equivalent to the category
\[ \prod_{(H)} \coprod_{k\geq 0} \tilde{\SI}_k\times_{\SI_k}(WH)^k. \]
\end{prop}

The displayed category is a skeleton of $G\sF$. As written, its objects are sets
of numbers $\{k_H\}$, one for each $(H)$, but they are thought of as the
finite $G$-sets $\coprod_H k_H G/H$. Its morphism groups specify the automorphisms
of these objects.  On objects, the equivalence sends a finite $G$-set $A$ to the 
unique finite $G$-set of the form $\coprod_{(H)} kG/H$ in the same isomorphism 
class as $A$.  Via chosen isomorphisms, this specifies the inverse 
equivalence to the inclusion of the chosen skeleton in $G\sF$.

We parametrize this equivalence to obtain a description of the category
$G\sF(X)$ of finite $G$-sets over $X$.  Given any $H$ and $k$, a $k$-tuple of
elements $\{x_1,\cdots,x_k\}$ of $X^H$ determines the $G$-map $p\colon kG/H \rtarr X$ 
that sends $eH$ in the $i$th copy of $G/H$ to $x_i$, and it is clear that every 
finite $G$-set $A$ over $X$ is isomorphic to one of this form.  Similarly, for
a finite $G$-set $q\colon B\rtarr X$ over $X$ and an isomorphism $f\colon A\rtarr B$,
$f$ is an isomorphism over $X$ from $q$ to $p=q\com f$, and every isomorphism
over $X$ can be constructed in this fashion.  Since we may as well choose $A$ and $B$
to be in our chosen skeleton of $G\sF$, this argument proves \myref{CATtwo}.


\section{The Barratt-Priddy-Quillen and tom Dieck splitting theorems}\label{Sec6}
\subsection{The Barratt-Priddy-Quillen theorem revisited}\label{Sec5.1}
The BPQ theorem shows how to model suspension $G$-spectra in 
terms of free $E_{\infty}$ $G$-categories and $G$-spaces. It is built 
tautologically into the equivariant infinite loop space machine in the 
same way as it is nonequivariantly \cite[2.3(vii)]{MayPerm} or \cite[\S10]{Rant1}.
The following result works for either $\bE_G = \bE_G^{\sS\!p}$ or $\bE_G =\bE_G^{\sS}$,
but note that the interpretation of both the source and target are different in the two cases.
The proof shows consistency with the versions of the BPQ theorem in 
Theorems \ref{BPQ2} and \ref{BPQ3}.

\begin{thm}[The $E_{\infty}$ operad BPQ theorem]\mylabel{SpThree} 
For an $E_{\infty}$ operad $\sC_G$ of $G$-spaces and based $G$-spaces $X$, there 
is a natural weak equivalence of $G$-spectra
\[ \SI^{\infty}_G X\rtarr \bE_G \mathbf{C_G}X. \]
\end{thm} 
\begin{proof} 
For $ \bE_G^{\sS\!p}$, recall that $\sC_U = \sK_U\times \sC_G$.  The same formal argument 
as for \myref{BPQ3} and use of the projections to $\sC_G$ and to $\sK_U$ give equivalences
of LM $G$-spectra
\[ \xymatrix{\SI^{\infty}_G X \ar[r] \ar[dr] & B(\SI^{\infty}_G,\BC_U,\BC_U X) \ar[d] 
\ar[r] & B(\SI^{\infty}_G,\BC_U,\BC_G X). \\
& B(\SI^{\infty}_G,\BK_U,\BK_U X) & \\} \]
For $ \bE_G^{\sS}$, recall that $\sC_V = \sK_V\times \sC_G$.  Analogously to 
\myref{BPQ2}, there is an orthogonal $G$-spectrum with $V$th space 
$B(\SI^V,\BC_V,\BC_VX)$. The usual formal argument and the projections to 
$\sC_G$ and $\sK_V$ give diagrams
\[ \xymatrix{\SI^V X \ar[r] \ar[dr]& B(\SI^V,\BC_V,\BC_V X) \ar[d] 
\ar[r] & B(\SI^V,\BC_V,\BC_G X). \\
& B(\SI^V,\BK_V,\BK_V X) & \\} \]
for all $V$ in which the left horizontal arrow and the vertical arrow are level equivalences of orthogonal
$G$-spectra, and the right horizontal arrow is a weak equivalence ($\pi_*$-isomorphism) of orthogonal
$G$-spectra, as we see my forgetting to $G$-prespectra and passing to colimits over $V\subset U$,
where $U$ is a complete $G$-universe.
\end{proof}

Taking $Y= X_+$ for an unbased $G$-space $X$ and using (\ref{opXone}), we can
rewrite this version of the BPQ theorem using the infinite loop space machine defined 
on permutative $G$-categories.

\begin{thm}[The categorical BPQ theorem: first version]\mylabel{SpThree2}  For unbased $G$-spaces $X$, there is a natural weak
equivalence of $G$-spectra
\[ \SI^{\infty}_G X_+\rtarr \bK_G \bP_G(X_+). \]
\end{thm} 

\begin{rem}  Diagrams showing compatibility with smash products, like those in
Theorems \ref{BPQ2} and \ref{BPQ3} are conspicuous by their absence from
Theorems \ref{SpThree} and \ref{SpThree2}.  A previous version of this article 
erroneously claimed that the operad $\sP$ has a self pairing $(\sP,\sP) \rtarr \sP$
induced by the homomorphisms
\begin{equation}\label{otimes}
 \otimes: \SI_j\times \SI_k \rtarr \SI_{jk},
 \end{equation}
which are made precise in \S\ref{SecPair} by use of  lexicographic ordering.
However, these do not satisfy the condition in \myref{pairop}(iii); see \myref{oops}.
For a conceptual understanding of why $\sP$ cannot have a self-pairing, consider the 
free $\sP$-algebra $\bP(S^0)$. This is a model for the groupoid of finite sets. As explained in 
\cite[Appendix A]{MayMult}, a self-pairing on $\sP$ would give strict distributivity on both sides
in $\bP(S^0)$. But the lexicographic ordering on $\mathbf{j}\times(\mathbf{k} \amalg \mathbf{m})$ does not agree with the lexicographic ordering on $(\mathbf{j}\times \mathbf{k}) \amalg (\mathbf{j}\times \mathbf{m})$.

As we explain in \cite{GMMOMult}, the homomorphisms $\otimes$ exhibit a product that exists in any operad. 
The categorical operads $\sP$ and $\sP_G$ are ``pseudo-commutative'', meaning that certain diagrams of
functors defined using these products commute up to natural isomorphism.    Putting together \myref{SpThree2}, the comparison of operadic and Segalic machines in \cite{MMO}, and $2$-category machinery developed in \cite{GMMOAdd}, we will
obtain multicategorical generalizations of the missing diagrams in \cite{GMMOMult}, where we complete the proofs from equivariant infinite loop space theory promised in \cite{GM2}.
\end{rem}

\subsection{The tom Dieck splitting theorem}\label{BQtD1}\label{Sec5.2}
The $G$-fixed point spectra of suspension $G$-spectra have a well-known 
splitting. It is due to tom Dieck \cite{tDOrbitII} on the level of homotopy groups
and was lifted to the spectrum level in \cite[\S V.11]{LMS}. The tom Dieck splitting 
actually works for all compact Lie groups $G$, but we have nothing helpful to add in 
that generality. Our group $G$ is always finite.  In that case, we have already
given the ingredients for a new categorical proof, as we now explain.

\begin{thm}\mylabel{SpOne} For a based $G$-space $Y$,
\[ (\SI^{\infty}_G Y)^G \htp \bigvee_{(H)} \SI^{\infty}(EWH_+\sma_{WH} Y^H). \]
The wedge runs over the conjugacy classes of subgroups $H$ of $G$, and
$WH = NH/H$.  
\end{thm}

\myref{SpOne} and the evident natural identifications 
\begin{equation}\label{sillyplus} 
EWH_+\sma_{WH} X_+^H \iso (EWH\times_{WH} X^H)_+
\end{equation}
imply the following version for unbased $G$-spaces $X$.

\begin{thm}\mylabel{SpTwo} For an unbased $G$-space $X$,
\[ (\SI^{\infty}_G X_+)^G \htp \bigvee_{(H)} \SI^{\infty}(EWH\times_{WH} X^H)_+. \]
\end{thm}

Conversely, we can easily deduce \myref{SpOne} from \myref{SpTwo}.  Viewing $S^0$ as
$\{1\}_+$ with trivial $G$ action, our standing assumption that basepoints are nondegenerate 
gives a based $G$-cofibration $S^0\rtarr Y_+$ that sends $1$ to the basepoint of $Y$,
and $Y = Y_+/S^0$.  The functors appearing in \myref{SpTwo} preserve cofiber
sequences, and the identifications (\ref{sillyplus}) imply identifications
\begin{equation}\label{billyplus} 
(EWH\times_{WH} Y^H)_+/(EWH\times_{WH} \{1\})_+ \iso EWH_+\sma_{WH} Y^H.
\end{equation}
Therefore \myref{SpTwo} implies \myref{SpOne}.

We explain these splittings in terms of the categorical BPQ theorem. 
We begin in the based setting. The nonequivariant case $G=e$ of the
BPQ theorem relates to the equivariant case through \myref{SpFour}. 
Explicitly, Theorems \ref{SpFour} and \ref{SpThree} give a pair of weak equivalences 
\begin{equation}\label{alpha} (\SI^{\infty}_G Y)^G \rtarr (\bE_G \mathbf C_G Y)^G \ltarr \bE((\mathbf C_G Y)^G).
\end{equation}
Since the functor $\SI^{\infty}$ commutes with wedges, the nonequivariant BPQ theorem gives a weak equivalence
\begin{equation}\label{beta} 
\bigvee_{(H)} \SI^{\infty} (EWH_+\sma_{WH} Y^H) \rtarr \bE\mathbf{C} (\bigvee_{(H)} (EWH_+\sma_{WH} Y^H).
\end{equation}
If we could prove that there is a natural weak equivalence of $\sC$-spaces
\[ (\mathbf C_GY)^G\htp \mathbf{C}(\bigvee_{(H)} (EWH_+\sma_{WH} Y^H), \]
that would imply a natural weak equivalence
\begin{equation}\label{gamma} 
\bE((\mathbf C_GY)^G)\htp \bE \mathbf{C} (\bigvee_{(H)} (EWH_+\sma_{WH} Y^H)
\end{equation}
and complete the proof of \myref{SpOne}.  However, the combinatorial study of the behavior 
of $C$ on wedges is complicated by the obvious fact that wedges of based spaces do not 
commute with products. 

We use the following consequence of \myref{ident} and the relationship between 
wedges and products of spectra to get around this.  Recall that $\mathbf{P}_G$ 
is the monad on based $G$-spaces obtained from the operad $|\sP_G|$ of $G$-spaces. 

\begin{thm}\mylabel{identtwo} For unbased $G$-spaces $X$, there is a natural 
equivalence of $|\sP|$-spaces
\[ (\mathbf{P}_G X_+)^G\htp \prod_{(H)}\mathbf{P}(EWH\times_{WH}X^H)_+, \]
where $(H)$ runs over the conjugacy classes of subgroups of $G$ and $WH = NH/H$. 
\end{thm}
\begin{proof} Remembering that $|\tilde G|= EG$, we see that the classifying 
space of the category $\widetilde{WH}\times_{WH}X^H$ can be identified with 
$EWH\times_{WH}X^H$. The commutation relations between $|-|$ and the constituent
functors used to construct the monads $\mathbf{P}_G$ on $G$-spaces and
$\bP_G$ on $G$-categories make the identification clear.
\end{proof}

\begin{rem}\mylabel{pickyprod}
Of course, we can and must replace $\sP_G$ and $\sP$ by their products with the equivariant and nonequivariant 
Steiner operad to fit into the  infinite loop space machine.  There is no harm in doing so since if we denote the product operads
by $\sO_G$ and $\sO$, as before, the projections $\sO_G\rtarr \sP_G$ and $\sO\rtarr \sP$ 
induce weak equivalences of monads that fit into a commutative diagram
\[ \xymatrix{
(\mathbf{O_G} X_+)^G \ar[r]^-{\htp} \ar[d]_{\htp} 
&\prod_{(H)}\mathbf{O}(EWH\times_{WH}X^H)_+ \ar[d]^{\htp}\\
(\mathbf{P}_G X_+)^G \ar[r]^-{\htp} & \prod_{(H)}\mathbf{P}(EWH\times_{WH}X^H)_+.\\}  \]
\end{rem}

The functor $\SI^{\infty}_G$ commutes with wedges, and the 
natural map of $G$-spectra 
$$E\wed F\rtarr E\times F$$ is a weak equivalence. 
Theorems \ref{prodcom} and \ref{SpThree} have the following implication.
We state it equivariantly, but we shall apply its nonequivariant special case.
\begin{prop}\mylabel{CWP} For based $G$-spaces $X$ and $Y$, the natural map 
\[ \bE_G \mathbf{O}_G(X\wed Y)\rtarr \bE_G (\mathbf{O}_GX\times \mathbf{O}_GY) \]
is a weak equivalence of $G$-spectra.
\end{prop}
\begin{proof} The following diagram commutes by the universal property of products.
\[ \xymatrix{
\SI^{\infty}_G(X\wed Y)\ar[d]_{\iso} \ar[r] & \bE_G \mathbf{O}_G (X\wed Y) \ar[d] \\
\SI^{\infty}_G X \wed \SI^{\infty}_G Y \ar[d] 
&\bE_G(\mathbf{O}_G X\times\mathbf{O}_G Y) \ar[d]\\
\SI^{\infty}_GX\times \SI^{\infty}_G Y \ar[r] 
& \bE_G \mathbf{O}_G X\times \bE_G \mathbf{O}_G Y.\\}\]
All arrows except the upper right vertical one are weak equivalences, hence that
arrow is also a weak equivalence.
\end{proof}

For any nonequivariant $E_{\infty}$ operad $\sC$, we therefore have a weak equivalence
\begin{equation}\label{delta} 
 \bE \mathbf C (\bigvee_{(H)} (EWH_+\sma_{WH} Y^H)\rtarr  \bE \prod_{(H)} \mathbf{C}(EWH_+\sma_{WH} Y^H).
\end{equation}

Together with (\ref{delta}), \myref{identtwo} and \myref{pickyprod} give a weak equivalence (\ref{gamma}) 
in the case $Y=X_+$.  Together with (\ref{alpha}) and (\ref{beta}), this completes the proof of \myref{SpTwo}, and \myref{SpOne} follows. 

\section{The $E_{\infty}$ operads $\sV_G$, $\sV^\times_G$, and $\sW_G$}\label{SecPQR}

The operad $\sP_G$ has a privileged conceptual role, but there are other
categorical $E_{\infty}$ $G$-operads with different good properties. We define 
three interrelated examples. The objects of the chaotic category $\sP_G(j)$ are functions 
$G\rtarr \SI_j$.  We give analogous chaotic $G$-categories in which the objects are suitable 
functions between well chosen infinite $G$-sets, with $G$ again acting by conjugation. 
Their main advantage over $\sP_G$ is that it is easier to
recognize $G$-categories on which they act.

\subsection{The definitions of $\sV_G$ and $\sV^\times_G$}

We start with what we would like to take as a particularly natural choice for 
the $j^{th}$ category of an $E_{\infty}$ $G$-operad.  It is described in more
detail in \cite[\S6.1]{GMM}. 

\begin{defn} 
Let $U$ be a countable ambient $G$-set that contains countably many copies of each orbit $G/H$. 
Let $U^j$ be the product of $j$ copies of $U$ with diagonal action by $G$, and let 
$^{j}U$ be the disjoint union of $j$ copies of the $G$-set $U$. Here $U^0$ is a one-point set, 
sometimes denoted $1$, and $^0\!U$ is the empty set, sometimes denoted 
$\emptyset$ and sometimes denoted $0$. 
\end{defn}

Let $\mathbf{j}=\{1,\cdots,j\}$ with its natural left action by $\SI_j$, written
$\si\colon \mathbf{j} \rtarr \mathbf{j}$.

\begin{defn}\mylabel{EGJ} 
For $j\geq 0$, let $\tilde{\sE}^U_G(j)$ be the
chaotic $G\times \SI_j$-category whose objects are the pairs $(A,\io)$, where $A$ is a $j$-element subset of $U$ and $\io\colon \mathbf{j}\rtarr A$ is a bijection. The group $G$ acts on objects by $g(A,\io) = (gA,g\io)$, where $(g\io)(i) = g\cdot \io(i)$. 
The group $\SI_j$ acts on objects by $(A,\io)\si = (A,\io\com \si)$ for $\si\in\SI_j$.
Since $\tilde{\sE}^U_G(j)$ is chaotic, this determines the actions on morphisms.
\end{defn}

\begin{prop}\mylabel{GSIHtpyType}\cite[6.3]{GMM} For each $j$, the classifying space 
$|\tilde{\sE}^U_G(j)|$ is a universal principal $(G,\SI_j)$-bundle. 
\end{prop}

Therefore $\tilde{\sE}_G^U(j)$ satisfies the properties required of the $j^{th}$ category 
of an $E_\infty$ $G$-operad. However, these categories as $j$ varies do not form an operad. 
The problem is a familiar one.   These categories can be thought of as analogous to configuration spaces.
Just as we fattened up the configuration space models of \S\ref{sec:ConfigModel} to the little discs operads of \S\ref{sec:DisksSteiner}, we must fatten up these categories
to provide enough room for an operad structure.

\begin{defn}\mylabel{operadPG} We define a reduced operad $\sV_G$ of $G$-categories. Let $\sV_G(j)$ be the 
chaotic $G$-category whose set of objects is the set of
injective functions $^j\!U\rtarr U$. Let $G$ act by conjugation and let $\SI_j$ have the right action 
induced by its left action on $^j\!U$. Let $\id\in\sV_G(1)$ be the identity 
function $U\rtarr U$. Define
\[ \ga\colon\sV_G(k)\times\sV_G(j_1)\times \cdots \times\sV_G(j_k)\rtarr\sV_G(j), \]
where $j = j_1+\cdots + j_k$, to be the composite 
\[\sV_G(k) \times\sV_G(j_1)\times \cdots \times\sV_G(j_k)\rtarr 
\sV_G(k)\times\sV_G(^j\!U,^k\!U) \rtarr\sV_G(j) \]
obtained by first taking coproducts of maps and then composing. Here $\sV(^j\!U,^k\!U)$ 
is the set of injections $^j\!U\rtarr ^k\!U$. The operad axioms \cite[1.1]{MayGeo} are easily verified.
\end{defn}

Remembering that taking sets to the free 
$\bR$-modules they generate gives a coproduct-preserving functor from sets 
to $\bR$-modules, we see that $\sV_G$ is a categorical 
analogue of the linear isometries operad $\sL_U$. 

There is a parallel definition that uses products instead of coproducts.  

\begin{defn} We define an unreduced operad $\overline{\sV}^\times_G$ of $G$-categories. Let 
$\overline{\sV}^\times_G(j)$ be 
the chaotic $G$-category whose set of objects is the set of
injective functions $U^j\rtarr U$. Let $G$ act by conjugation and let $\SI_j$ have the right action 
induced by its left action on $U^j$. Let $\id\in \overline{\sV}^\times_G(1)$ be the identity function. Define
\[ \ga\colon \overline{\sV}^\times_G(k)\times \overline{\sV}^\times_G(j_1)\times\cdots \times \overline{\sV}^\times_G(j_k)\rtarr \overline{\sV}^\times_G(j), \]
where $j = j_1+\cdots + j_k$, to be the composite 
\[ \overline{\sV}^\times_G(k) \times \overline{\sV}^\times_G(j_1)\times\cdots \times \overline{\sV}^\times_G(j_k)\rtarr \overline{\sV}^\times_G(k)\times \overline{\sV}^\times_G(U^j,U^k) \rtarr \overline{\sV}^\times_G(j) \]
obtained by first taking products of maps and then composing. Here $\overline{\sV}^\times_G(U^j,U^k)$ is the set of
injections $U^j\rtarr U^k$. Again, the operad axioms are easily verified.
\end{defn}

Observe that the objects of $\overline{\sV}^\times_G(0)$ are the injections from the point $U^0$ 
into $U$ and can be identified with the set $U$, whereas $\sV_G(0)$ is the trivial category 
given by the injection of the empty set $^{0}\! U$ into $U$.  As in
\myref{reduced}, the objects of the zeroth category give unit objects for operad
actions, and it is convenient to restrict attention to a reduced variant of 
$\overline{\sV}^\times_G$.  

\begin{defn} Choose a $G$-fixed point $1\in U$ (or, equivalently, adjoin a $G$-fixed
basepoint $1$ to $U$) and also write $1$ for the single point in $U^0$.  Give $U^j$, 
$j\geq 0$, the basepoint whose coordinates are all $1$. The reduced variant of $\overline{\sV}^\times_G$ 
is the operad $\sV^\times_G$ of $G$-categories that is obtained by restricting the objects of the $\overline{\sV}^\times_G(j)$ to 
consist only of the basepoint preserving injections $U^j\rtarr U$ for all $j\geq 0$.
\end{defn}

\begin{rem} If $\overline{\sV}^\times_G$ acts on a category $\sA$, then $\sV^\times_G$ acts on $\sA$ by
restriction of the action.  However, $\sV^\times_G$ can act even though
$\overline{\sV}^\times_G$ does not.  This happens when the structure of $\sA$ encodes a particular
unit object and the operad action conditions fail for other choices of objects in $\sA$.
\end{rem}

\begin{prop}\mylabel{UniPQ} The classifying spaces $|\sV_G(j)|$, $|\overline{\sV}^\times_G(j)|$, and $|\sV^\times_G(j)|$ 
are universal principal $(G,\SI_j)$-bundles, hence $\sV_G$, $\overline{\sV}^\times_G$, and $\sV^\times_G$ are $E_{\infty}$ operads.
\end{prop}
\begin{proof}
Since the objects of our categories are given by injective functions, $\SI_j$ acts freely on the
objects of $\sV_G(j)$ and $\sV^\times_G(j)$. Since our categories are chaotic, it suffices to show that
if $\LA\cap \SI_j = \{e\}$, where $\LA\subset G\times \SI_j$, then the object sets
$\sV_G(j)^{\LA}$ and $\sV^\times_G(j)^{\LA}$ are nonempty. This means that there are $\LA$-equivariant
injections $^{j}U\rtarr U$ and $U^j\rtarr U$, and in fact there are $\LA$-equivariant
bijections.  We have $\LA =\{(h,\al(h))|h\in H\}$ for a subgroup $H$ of $G$ and a
homomorphism $\al\colon H\rtarr \SI_j$, and we may regard $U$ as an 
$H$-set via the canonical isomorphism
$H\iso \LA$.  Since countably many copies of every orbit of $H$ 
embed in $U$, $^{j}U$, and $U^j$ for $j\geq 1$, these sets are all isomorphic 
as $H$-sets and therefore as $\LA$-sets.
\end{proof}

\subsection{The definition of $\sW_G$ and its action on $\sV_G$}\label{sec:MultCatOps}

This section is parenthetical, aimed towards work in progress on a new 
version of multiplicative infinite loop space theory. The notion of an 
action of a ``multiplicative'' operad $\sG$ on an ``additive'' 
operad $\sC$ was defined in \cite[VI.1.6]{MQR}, and $(\sC,\sG)$ was then 
said to be an ``operad pair''. This notion was redefined and discussed in 
\cite{Rant1, Rant2}. Expressed in terms
of diagrams rather than elements, it makes sense for operads in any cartesian
monoidal category, such as the categories of $G$-categories and of $G$-spaces. As is 
emphasized in the cited papers, although this notion is the essential starting 
point for the theory of $E_{\infty}$ ring spaces, the only interesting nonequivariant example 
we know is $(\sK,\sL)$, where $\sK$ is the Steiner operad. As pointed out in 
in \S\ref{sec:Exmps}, this example works equally well equivariantly. 

The pair of operads $(\sV_G,\sV^\times_G)$ very nearly gives another example, but 
we must shrink $\sV^\times_G$ and drop its unit object to obtain this.

\begin{defn} Define $\sW_G\subset\sV^\times_G$ to be the suboperad such that $\sW_G(j)$ 
is the full subcategory
of $\sV^\times_G(j)$ whose objects are the based bijections $U^j\rtarr U$. In particular, 
$\sW_G(0)$ is the empty
category, so that the operad $\sW_G$ does not encode unit object information.  By
the proof of \myref{UniPQ}, for 
$j\geq 1$ $\sW_G(j)$ 
is again a universal principal $(G,\SI_j)$-bundle.  We view $\sW_G$ as a restricted $E_{\infty}$ operad, 
namely one without unit objects.
\end{defn}

\begin{prop}\mylabel{WVaction}  The restricted operad $\sW_G$ acts on the operad $\sV_G$.  
\end{prop}
\begin{proof} We must specify action maps
\[ \la\colon\sW_G(k) \times\sV_G(j_1)\times \cdots \times\sV_G(j_k) \rtarr\sV_G(j),\]
where $j=j_1\cdots j_k$ and $k\geq 1$.  To define them, consider the set of sequences $I = \{i_1,\cdots, i_k\}$, 
ordered lexicographically, where 
$1\leq i_r\leq j_r$ and $1\leq r\leq k$. For an injection $\ph_r\colon ^{j_r}\!U\rtarr U$,
let $\ph_{i_r}\colon U\rtarr U$ denote the restriction of $\ph_r$ to the $i_r^{th}$ copy of $U$ in $^{j_r}\!U$.
Then let 
\[ \ph_I = \phi_{i_1}\times\cdots \times \ph_{i_k} \colon U^k\rtarr U^k. \]  
For a bijection $\ps\colon U^k\rtarr U$, define
\[ \la(\ps; \ph_1,\cdots, \ph_k)\colon ^{j}U\rtarr U \]
to be the injection which restricts on the $I^{th}$ copy of $U$ to the composite 
\[  \xymatrix@1{
U \ar[r]^-{\ps^{-1}} & U^k \ar[r]^-{\phi_I} & U^k \ar[r]^-{\ps} & U.\\} \]
It is tedious but straightforward to verify that all conditions 
specified in \cite[VI.1.6]{MQR}, \cite[4.2]{Rant2} that make sense are 
satisfied\footnote{In fact, with the details of \cite[4.2]{Rant2}, the only condition that does not make sense would 
require $\la(1) = \id\in\sV_G(1)$, where $\{1\}=\sW(0)$, and that condition lacks force since it does not interact with the remaining conditions.}.
\end{proof}

\begin{rem} When all $j_i = 1$, so that there is only one sequence $I$, we can 
define $\la$ more generally, 
with $\sV^\times_G(k)$ replacing $\sW_G(k)$, by letting 
\[ \la(\ps; \ph_1,\cdots, \ph_k)\colon U\rtarr U \]
be the identity on the complement of the image of the injection $\ps\colon U^k\rtarr U$ and
\[  \xymatrix@1{
\ps(U) \ar[r]^-{\ps^{-1}} & U^k \ar[r]^-{\phi_I} & U^k \ar[r]^-{\ps} & \ps(U)\\} \]
on the image of $\ps$. Clearly we can replace $\sV_G(1)$ by $\sV^\times_G(1)$ here.
\end{rem}

This allows us to give the following speculative analogue of \myref{permcat}. 
An $E_{\infty}$ ring space is defined to be a $(\sC,\sG)$-space, where $(\sC,\sG)$ is an 
operad pair such that $\sC$ and $\sG$ are $E_{\infty}$ operads of spaces. Briefly, a $(\sC,\sG)$-space $X$ is a $\sC$-space and a $\sG$-space with respective basepoints $0$ and $1$ such that $0$ is a zero element for the $\sG$-action and the action $\mathbf{C} X\rtarr X$ is a map of $\sG$-spaces with zero, where $\mathbf{C}$ denotes the monad associated to the operad $\sC$. Here the action 
of $\sG$ on $\sC$ induces an action of $\sG$ on the free $\sC$-spaces $\mathbf{C}X$, 
so that $\mathbf{C}$ restricts to a monad in the category of $\sG$-spaces. 
These notions are redefined in the more recent papers \cite{Rant1, Rant2}. The definitions are formal and apply equally well to spaces, $G$-spaces, categories, 
and $G$-categories.  

\begin{defn}\mylabel{bipermcat} An $E_{\infty}$ ring $G$-category $\sA$ is a $G$-category 
together with an action by the $E_{\infty}$ operad pair $(\sV_G,\sW_G)$ such that the multiplicative 
action extends from the restricted $E_{\infty}$ operad $\sW_G$ to an action of the $E_{\infty}$ operad $\sV^\times_G$. 
\end{defn}

The notion of a bipermutative category, or symmetric strict bimonoidal category, was specified in \cite[VI.3.3]{MQR}.  With the standard skeletal model, the direct sum and tensor product on the category of finite dimensional free modules over a commutative ring $R$ gives a typical example. Without any categorical justification, we allow ourselves to think of $E_{\infty}$ ring $G$-categories as an $E_{\infty}$ version of genuine operadic bipermutative $G$-categories.  A less concrete but more general version of this notion is defined and developed in \cite{GMMOMult}. 

Our notion of an $E_{\infty}$ $G$-category $\sA$ implies that $B\sA$ is an $E_{\infty}$ $G$-space. We would like to say that our notion of an $E_{\infty}$ ring $G$-category $\sA$ implies that $B\sA$ is an $E_{\infty}$ ring $G$-space, but that is not quite true. 
However, we believe there is a way to prove the following conjecture that avoids the categorical work of \cite{EM, GMMOAdd, GMMOMult, MayMult, Rant2}.  However, that proof is work in progress. 

\begin{conj} There is an infinite loop space machine that carries $E_{\infty}$ ring
$G$-categories to $E_{\infty}$ ring $G$-spectra. 
\end{conj}

\section{Examples of $E_{\infty}$  and $E_{\infty}$ ring $G$-categories}\label{sec:Einf}
We have several interesting examples.  We emphasize that these particular constructions
are new even when $G=e$.  In that case, we may take $U$ to be the set of positive integers, with 
$1$ as basepoint. 

We have the notion of a genuine permutative $G$-category, which comes with a 
preferred product, and the 
notion of a $\sV_G$-category, which does not.  It seems plausible that 
the latter notion is more general, but to verify that we would have to show how to regard a permutative category as a $\sV_G$-algebra. One natural way to do so would 
be to construct a map of operads $\sV_G\rtarr \sP_G$, but we do not know how to do 
that. Of course, the equivalence of $\sV_G$-categories and $\sP_G$-categories shows that genuine permutative categories give a plethora of examples of $\sV_G$-algebras up to homotopy. However, the most important examples can easily be displayed directly, 
without recourse to the theory of permutative categories.  

\subsection{The $G$-category $\sE^U_G = \sE_G^{\sV}$ of finite sets}
Recall \myref{altnot}. Intuitively, we would like to have a genuine permutative $G$-category
whose product is given by disjoint unions of finite sets, with $G$ relating
finite sets (not $G$-sets) by translations. Even nonequivariantly, this is 
imprecise due to both size issues and the fact that categorical coproducts are not strictly asssociative. We make it precise by taking coproducts of finite subsets of our 
ambient $G$-set $U$, but we must do so without assuming that our given finite 
subsets are disjoint.  We achieve this by using injections $^j\!U\rtarr U$ to 
separate them.  We do not have canonical choices for the injections,
hence we have assembled them into our categorical $E_{\infty}$ operad $\sV_G$.  
Recall \myref{EGJ} and \myref{GSIHtpyType}.  

\begin{defn}\mylabel{ExamE} The $G$-category $\tilde{\sE}^U_G$ of finite ordered sets is the coproduct 
over $n\geq 0$ of the $G$-categories $\tilde{\sE}^U_G(n)$. The $G$-category $\sE^U_G = \sE_G^{\sV}$ of 
finite sets is the coproduct
over $n\geq 0$ of the orbit categories $\tilde{\sE}^U_G(n)/\SI_n$.  By \myref{GSIHtpyType}, $B\sE^U_G$ is 
the coproduct over $n\geq 0$ of classifying spaces $B(G,\SI_n)$. Explicitly, by \cite[6.5]{GMM},
the objects of $\sE^U_G$ are the finite subsets (not $G$-subsets) $A$ of $U$. 
Its morphisms are the bijections $\nu\colon A\rtarr B$; if $A$ has $n$ points,
the morphisms $A\rtarr A$ give a copy of the set $\SI_n$. The group $G$ acts by translation on objects, so that 
$gA = \{ga | a\in A\}$, and by conjugation on morphisms, so that  $g\nu\colon gA\rtarr gB$ is given by 
$(g\nu)(g\cdot a) = g\cdot\nu(a)$. 
\end{defn}

\begin{prop} The $G$-categories $\tilde{\sE}^U_G$ and $\sE^U_G$ are $\sV_G$-categories,
and passage to orbits over symmetric groups defines a map $\tilde{\sE}^U_G\rtarr \sE^U_G$
of $\sV_G$-categories.
\end{prop}
\begin{proof}
Define a $G$-functor
\[ \tha_j\colon\sV_G(j) \times (\sE^U_G)^j\rtarr \sE^U_G \]
as follows. On objects, for $\ph\in\sV_G(j)$ and $A_{i}\in O\!b\,\sE^U_G$, 
$1\leq i\leq j$, define 
\[ \tha_j(\ph; A_1,\cdots,A_j) = \ph(A_1\amalg \cdots\amalg A_j), \]
where $A_i$ is viewed as a subset of the $i^{th}$ copy of $U$ in $^{j}U$.
For a morphism 
\[ (\ze; \nu_1, \cdots, \nu_j)\colon (\ph; A_1, \cdots, A_j)\rtarr (\ps; B_1, \cdots, B_j), \]
where $\ze\colon \ph\rtarr \ps$ is the unique morphism, 
define $\tha_j(\ze; \nu_1, \cdots, \nu_j)$ to be the unique bijection that makes the
following diagram commute.
\[ \xymatrix{
A_1 \amalg \cdots \amalg  A_j  \ar[d]_{\nu_1\amalg\cdots \amalg \nu_j} \ar[r]^-{\ph}
& \ph(A_1\amalg \cdots\amalg A_j)  \ar[d]^{\tha_j(\ze; \nu_1, \cdots, \nu_j)}\\
B_1 \amalg \cdots \amalg B_j \ar[r]_-{\ps} & \ps(B_1\amalg \cdots\amalg B_j) \\} \]
Then the $\tha_j$ specify an action of $\sV_G$ on $\sE^U_G$.  

Since the $\tilde{\sE}^U_G(n)$
are chaotic, to define an action of $\sV_G$ on $\tilde{\sE}^U_G$ we need only specify
the required $G$-functors 
\[ \tilde{\tha}_j\colon\sV_G(j) \times (\tilde{\sE}^U_G)^j\rtarr \tilde{\sE}^U_G \]
on objects.  A typical object has the form $(\ph;(A_1,\io_1), \cdots, (A_j,\io_j))$,
$\io_i\colon \mathbf{n_i}\rtarr A_i$. We have the canonical isomorphism 
$\mathbf{n_1}\amalg\cdots \amalg \mathbf{n_1}\iso \mathbf{n}$, $n = n_1+\cdots n_j$, 
and $\tilde{\tha}_j$ sends our typical object to
\[ \big{(}\ph(A_1\amalg\cdots, \amalg A_j),\ph\com (\io_1\amalg\cdots \amalg \io_j)\big{)}.\]  
Again, the $\tilde{\tha}_j$ specify an action. The compatibility with passage
to orbits is verified by use of canonical orbit
representatives for objects $A$ that are obtained by choosing fixed reference maps 
$\et_A\colon \mathbf{n}\rtarr A$ for each $n$-point set $A\subset U$; compare
\cite[Proposition 6.3 and Lemma 6.5]{GMM}. 
\end{proof}

\begin{rem}\mylabel{EvsF} If we restrict to the full $G$-subcategory of $\sE^U_G$ of 
$G$-fixed sets $A$ of cardinality $n$, we obtain an equivalent analogue 
of the category $\sF_G(n)$ of \myref{Gsets}: these are two small models 
of the $G$-category of all $G$-sets with $n$ elements and the bijections 
between them, and they have 
isomorphic skeleta. Thus the restriction of $\sE^U_G$ to its full $G$-subcategory 
of $G$-fixed sets $A$ is an equivalent analogue of $\sF_G$.  Remember from 
\myref{subtle} that no $E_{\infty}$ operad can be expected to act on $\sF_G$. 
The $\sV_G$-category $\sE^U_G$ gives a convenient substitute. 
\end{rem}

\subsection{The $G$-category $\sG\!\sL_G(R)$ for a $G$-ring $R$}\label{SecKGBtoo}
Let $R$ be a $G$-ring, that is a ring with an action
of $G$ through automorphisms of $R$.  We have analogues of Definitions
\ref{EGJ} and \ref{ExamE} that can be used in equivariant algebraic $K$-theory. 
For a set $A$, let $R[A]$ denote the free $R$-module on the basis $A$. Let $G$ 
act entrywise on the matrix group $GL(n,R)$ and diagonally on $R^n$.  Our conventions 
on semi-direct products
and their universal principal $(G,GL(n,R)_G)$-bundles are in \cite{GMM}, and
\cite[\S6.3]{GMM} gives more details on the following definitions.

\begin{defn}\mylabel{ExamR0} We define the chaotic general linear category 
$\widetilde{\sG\!\sL}_G(n,R)$. Its objects are the monomorphisms of (left) $R$-modules 
$\io\colon R^n\rtarr R[U]$.  The group $G$ acts on objects by 
$g\io = g\com \io \com g^{-1}$.  The group $GL(n,R)$ acts on objects by
$\io \ta = \io \com \ta\colon R^n\rtarr R[U]$.  Since $\widetilde{\sG\!\sL}_G(n,R)$ is chaotic, 
this determines the actions on morphisms.
\end{defn}

\begin{prop}\mylabel{GLHtpyType}\cite[6.18]{GMM} 
The actions of $G$ and $GL(n,R)$ on $\widetilde{\sG\!\sL}_G(n,R)$
determine an action of $GL(n,R) \rtimes G$, and the classifying space 
$|\widetilde{\sG\!\sL}_G(n,R)|$ is a universal principal $(G,GL(n,R)_{G})$-bundle.
\end{prop}

\begin{defn}\mylabel{ExamR} The general linear $G$-category $\sG\!\sL_G(R)$ 
of finite dimensional free $R$-modules is the coproduct over $n\geq 0$ of
the orbit categories $\widetilde{\sG\!\sL}_G(n,R)/GL(n,R)$. By \myref{GLHtpyType},
$B\sG\!\sL_G(R)$ is the coproduct over $n\geq 0$ of classifying spaces $B(G,GL(n,R)_G)$.
Explicitly, by \cite[6.20]{GMM}, the objects of $\sG\!\sL_G(R)$ are the finite dimensional
free $R$-submodules $M$ of $R[U]$. 
The morphisms $\nu\colon M\rtarr N$ are the isomorphisms of $R$-modules.
The group $G$ acts by translation on objects, so that 
$gM = \{gm\,|\,m\in M\}$, and by conjugation on morphisms, 
so that $(g\nu)(gm) = \nu(m)$ for $m\in M$ and $g\in G$. 
\end{defn}

\begin{prop} The $G$-categories $\widetilde{\sG\!\sL}_G(R)$ and $\sG\!\sL_G(R)$ are 
$\sV_G$-categories and passage to orbits over general linear groups defines a 
map $\widetilde{\sG\!\sL}_G(R)\rtarr \sG\!\sL_G(R)$ of $\sV_G$-categories.
\end{prop}
\begin{proof}
Define a functor
\[ \tha_j\colon\sV_G(j) \times \sG\!\sL_G(R)^j\rtarr \sG\!\sL_G(R) \]
as follows. On objects, for $\ph\in\sV_G(j)$ and $M_i\in \sP b\sG\!\sL_G(R)$, 
$1\leq i\leq j$, define
\[ \tha_j(\ph; M_1, \cdots, M_j) = R[\ph](M_1\oplus \cdots\oplus M_j), \]
where $R[\ph]\colon R[^j\!U]\rtarr R[U]$ is induced by $\ph\colon ^j\!U \rtarr U$
and $M_i$ is viewed as a submodule of the $i^{th}$ copy of $R[U]$ in 
$R[^j\!U] = \oplus_j R[U]$. 
For a morphism 
\[ (\ze; \nu_1, \cdots, \nu_j)\colon (\ph; M_1, \cdots, M_j)
\rtarr (\ps; N_1, \cdots, N_j), \]
define $\tha_j(\io; \nu_1, \cdots, \nu_j)$ to be the unique isomorphism of $R$-modules that makes the following diagram commute.
\[ \xymatrix{
M_1 \oplus \cdots \oplus M_j  \ar[d]_{\nu_1\oplus\cdots \oplus \nu_j} \ar[r]^-{R[\phi]}  
& R[\ph](M_1\oplus \cdots\oplus M_j)  \ar[d]^{\tha_j(\ze; \nu_1, \cdots, \nu_j)} \\
N_1 \oplus \cdots \oplus N_j \ar[r]_-{R[\ps]}
& R[\ps](N_1\oplus \cdots\oplus N_j)\\} \]
Then the $\tha_j$ specify an action of $\sV_G$ on $\sG\!\sL_G(R)$. Since the 
$\widetilde{\sG\!\sL}_G(R,n)$ are chaotic, to define an action of $\sV_G$ on 
$\widetilde{\sG\!\sL}_G(R)$, we need only specify the required $G$-functors
\[ \tilde{\tha}_j\colon\sV_G(j) \times \widetilde{\sG\!\sL}_G(R)^j
\rtarr \widetilde{\sG\!\sL}_G(R) \]
on objects.  A typical object has the form $(\ph;\io_1,\cdots,\io_j)$,
$\io_i\colon R^{n_i}\rtarr R[U]$, and, with $n=n_1+\cdots +n_j$,  $\tilde{\tha}_j$ 
sends it to 
$$R[\ph]\com (\io_1\oplus\cdots \oplus \io_j)\colon R^n\rtarr R[U].$$
Again, the $\tilde{\tha}_j$ specify an action. The compatibility with passage
to orbits is verified by use of canonical orbit
representatives for objects that are obtained by choosing reference maps 
$\et_M\colon R^n \rtarr M$ for each $M$ dimensional free $R$-module $M\subset R[U]$; compare \cite[6.18, 6.20]{GMM}. 
\end{proof}

On passage to classifying spaces and then to $G$-spectra via our infinite loop space machine  $\bE_G$, we obtain a 
model $\bE_G B\sG\!\sL_G(R)$ for the $K$-theory spectrum 
$\bK_G(R)$ of $R$. The following result compares the two evident models in sight.

\begin{defn} Define the naive permutative $G$-category $GL_G(R)$ to be the 
$G$-groupoid whose objects are the $n\geq 0$ and whose set of morphisms 
$m\rtarr n$ is empty if $m\neq n$ and is the $G$-group $GL(n,R)$ if $m=n$, 
where $G$ acts entrywise.  The product is given by block sum of matrices. 
Applying the chaotic groupoid functor to the groups $GL(n,R)$ we obtain 
another naive permutative $G$-category $\widetilde{GL}_G(R)$ and a map 
$\widetilde{GL}_G(R)\rtarr GL_G(R)$ of naive permutative $G$-categories.
Applying the functor $\sC\!at(\tilde G,-)$ from \myref{ScriptG}, we obtain a map of 
genuine permutative $G$-categories $\sC\!at(\tilde G, (\widetilde{GL}_G(R)))\rtarr \sC\!at(\tilde G,(GL_G(R)))$.  
\end{defn}

It is convenient to write $\sG\!\sL_G^{\sP}(R)$ for the $\sP_G$-category 
$\sC\!at(\tilde G,(GL_G(R)))$ and $\sG\!\sL_G^{\sV}(R)$ for the $\sV_G$-category $\sG\!\sL_G(R)$,
and similarly for their total space variants  $\sC\!at(\tilde G,(\widetilde{GL}_G(R)))$ 
and $\widetilde{\sG\!\sL}_G(R)$.  We have the following comparison theorem.

\begin{thm}  The $G$-spectra $\bK_G \sG\!\sL_G^{\sP}(R)$ and
$\bK_G \sG\!\sL_G^{\sV}(R)$ are weakly equivalent, functorially
in $G$-rings $R$. 
\end{thm}
\begin{proof} We again use the product of operads trick from \cite{MayGeo}. 
Projections and quotient maps give a commutative diagram of 
$(\sP_G\times\sV_G)$-categories
\[ \xymatrix{
\widetilde{\sG\!\sL}_G^{\sP}(R) \ar[d] &
\widetilde{\sG\!\sL}_G^{\sP}(R)\times \widetilde{\sG\!\sL}_G^{\sV}(R)
\ar[d] \ar[l] \ar[r]  &
\widetilde{\sG\!\sL}_G^{\sV}(R) \ar[d] \\
\sG\!\sL_G^{\sP}(R)   & 
\sG\!\sL_G^{\sP\times\sV}(R) \ar[l] \ar[r] &
\sG\!\sL_G^{\sV}(R).\\} \]
The middle term at the top denotes the diagonal product, namely
\[ \coprod_n \ \widetilde{\sG\!\sL}_G^{\sP}(n,R) \times \widetilde{\sG\!\sL}_G^{\sV}(n,R).\] 
The middle term on the bottom is the coproduct over $n$ of the orbits of these products
under the diagonal action of $GL(n,R)$.
The product of total spaces of universal principal $(G,GL(R,n)_G)$-bundles is 
the total space of another universal principal $(G,GL(R,n)_G)$-bundle.  Therefore,
after application of the classifying space functor, the horizontal projections 
display two equivalences between universal principal $(G,GL(R,n)_G)$-bundles.
The conclusion follows by hitting the resulting diagram with the functor $\bK_G$ 
defined with respect to $(\sP_G\times\sV_G)$-categories and using evident equivalences 
to the functors $\bK_G$ defined with respect to $\sP_G$-categories and $\sV_G$-categories when the input is given by $\sP_G$ or $\sV_G$-categories.
\end{proof}

\subsection{Multiplicative actions on $\sE^U_G$ and $\sG\!\sL_G(R)$}
 
We agree to think of $\sV^\times_G$-categories as ``multiplicative'', whereas we think of 
$\sV_G$-categories as ``additive''.  

\begin{prop}\mylabel{ExamEM} The $G$-category $\sE^U_G$ is a $\sV^\times_G$-category.
\end{prop}
\begin{proof} Define a $G$-functor
\[ \xi_j\colon\sV^\times_G(j) \times (\sE^U_G)^j\rtarr \sE^U_G \]
as follows. On objects, for $\ph\in\sV^\times_G(j)$ and $A_i\in \sE_G$, $1\leq i\leq j$, 
define 
\[ \xi_j(\ph; A_1,\cdots,A_j) = \ph(A_1\times \cdots\times A_j). \]
For a morphism 
\[ (\ze; \nu_1, \cdots , \nu_j)\colon (\ph; A_1, \cdots , A_j)\rtarr (\ps; B_1, \cdots , B_j) \]
define $\xi_j(\ze; \nu_1, \cdots , \nu_j)$ to be the unique bijection that makes the
following diagram commute.
\[ \xymatrix{
A_1 \times \cdots \times  A_j  \ar[d]_{\nu_1\times \cdots \times \nu_j}
\ar[r]^-{\ph} &
\ph(A_1\times \cdots\times A_j)  \ar[d]^{\xi_j(\ze; \nu_1, \cdots , \nu_j)} \\
B_1 \times \cdots \times B_j \ar[r]_-{\ps} 
& \ps(B_1\times \cdots\times B_j)\\} \]
Then the $\xi_j$ specify an action of $\sV^\times_G$ on $\sE^U_G$. 
\end{proof}

\begin{prop} If $R$ is a commutative $G$-ring, then $\sG\!\sL_G(R)$ 
is a $\sV^\times_G$-category.
\end{prop}
\begin{proof} Define a functor
\[ \xi_j\colon\sV^\times_G(j) \times \sG\!\sL(R)_G^j\rtarr \sG\!\sL_G(R) \]
as follows. Identify $R[U^j]$ with $\otimes_jR[U]$, where $\otimes = \otimes_R$. On objects, for $\ph\in\sV_G(j)$ and $R$-modules $M_i\subset R[U]$, 
$1\leq i\leq j$, define 
\[ \xi_j(\ph; M_1, \cdots , M_j) = R[\ph](M_1\times \cdots\times M_j).\]
For a morphism 
\[ (\ze; \nu_1, \cdots , \nu_j)\colon (\ph; M_1, \cdots , M_j)
\rtarr (\ps; N_1, \cdots , N_j) \]
define $\xi_j(\ze; \nu_1, \cdots , \nu_j)$ to be the unique isomorphism of $R$-modules that makes the following diagram commute.
\[ \xymatrix{
M_1 \otimes \cdots \otimes M_j 
\ar[d]_{\nu_1\otimes\cdots \otimes \nu_j} \ar[r]^-{R[\ph]}
& \ph(M_1\otimes \cdots\otimes M_j)] 
\ar[d]^{\xi_j(\ze; \nu_1, \cdots , \nu_j)} \\
N_1 \otimes \cdots \otimes N_j  \ar[r]^-{R[\ps]}&  
\ps(N_1\otimes \cdots\otimes N_j).\\} \]
Then the $\xi_j$ specify an action of $\sV^\times_G$ on $\sG\!\sL_G(R)$.
\end{proof}

Restricting the action from $\sV^\times_G$ to $\sW_G$, the examples above and easy diagram chases 
prove that the operad pair $(\sV_G,\sW_G)$ acts on the categories $\sE_G$ and 
$\sG\!\sL_G(R)$.  This proves the following result. 

\begin{thm} The categories $\sE^U_G$ and $\sG\!\sL_G(R)$ for a commutative $G$-ring 
$R$ are $E_{\infty}$ ring $G$-categories in the sense of \myref{bipermcat}. 
\end{thm}

Although  we have a 
definition of a genuine permutative $G$-category, we do not have a comparably
simple definition of a genuine bipermutative $G$-category.  The previous examples 
show that we do have examples of $E_{\infty}$ ring $G$-categories.  In \cite{GMMOMult}, 
we will show how to construct $E_{\infty}$ ring $G$-categories from general naive 
bipermutative $G$-categories, in particular nonequivariant bipermutative categories,
and we will show how to construct genuine commutative ring $G$-spectra from them.

\section{The $\sV_G$-category $\sE^U_G(X)$ and the BPQ-theorem}\label{SecEGP}

We now return to the categorical BPQ-theorem, but thinking
in terms of $\sV_G$-categories rather than $\sP_G$-categories.  
This gives a more intuitive approach to the $G$-category of finite sets
over a $G$-space $X$.  

\subsection{The $G$-category $\sE^U_G(X)$ of finite sets over $X$}\label{Sec9.1}
\begin{defn}\mylabel{E(G,X)} 
Let $X$ be a $G$-space. We define the $G$-groupoid $\sE^U_G(X)=\sE_G^{\sV}(X)$ of 
finite sets over $X$. Its objects are the functions $p\colon A\rtarr X$, where 
$A$ is a finite subset of our ambient $G$-set $U$. For a second function 
$q\colon B\rtarr X$, a map $\nu \colon p\rtarr q$ is a bijection $\nu \colon A\rtarr B$ 
such that $q\com \nu= p$. Composition is given by composition of functions over $X$.
The group $G$ acts by translation of $G$-sets and conjugation on all maps in sight. 
Thus, for an object $p\colon A\rtarr X$, $gp\colon gA\rtarr X$ is given by 
$(gp)(ga) = g(p(a))$.  For a map 
$\nu \colon p\rtarr q$, $g\nu \colon gA\rtarr gB$ is given by $(g\nu)(ga) = g(\nu(a))$. 

To topologize $\sE^U_G(X)$, give $U$ and $X$ disjoint basepoints $\ast$.\footnote{These basepoints 
are just a convenience for specifying the topology; they play no other role.}
View the set $\sO\!b$ of objects of $\sE^U_G(X)$ as the set of based
functions $p\colon U_+\rtarr X_+$ such that $p^{-1}(\ast)$ is the complement 
of a finite set $A\subset U$. Topologize $\sO\!b$ as a 
subspace of $X_+^{U_+}$.  View the set $\sM\!or$ of morphisms of $\sE^U_G(X)$ 
as a subset of the set of functions $\mu \colon U_+\rtarr U_+$ that send the 
complement of some finite set $A\subset U$ to $\ast$ and map $A$ bijectively 
to some finite set $B\subset U$. Topologize $\sM\!or$ as the subspace of points 
$(p,\mu,q)$ in $\sO\! b\times U_+^{U_+}\times \sO\!b$, where $U_+^{U_+}$ is discrete. 
When $X$ is a finite set and thus a discrete space (since points are closed in spaces 
in the category $\sU$), $\sE^U_G(X)$ is discrete. 

Let  $\sE^U_G(n,X)$ denote the full subcategory of $\sE^U_G(X)$ of maps 
$p\colon A\rtarr X$ such that $A$ has $n$ elements. Then $\sE^U_G(X)$ is
the coproduct of the groupoids $\sE^U_G(n,X)$.
\end{defn}

\begin{prop} The operad $\sV_G$ acts naturally on the categories $\sE^U_G(X)$. 
\end{prop}
\begin{proof}
For $j\geq 0$, we must define functors
\[ \tha_j\colon\sV_G(j)\times \sE^U_G(X)^j\rtarr \sE^U_G(X). \]
To define $\tha_j$ on objects, let $\ph\colon ^{j}U\rtarr U$ be an injective function and $p_i\colon A_i\rtarr X$ be a function, $1\leq i\leq j$, where $A_i$ is a finite subset of $U$.  We define $\tha_j(\ph; p_1,\cdots, p_j)$ to be the 
composite
\[ \xymatrix@1{
 \ph(A_1\amalg\cdots \amalg A_j) \ar[r]^-{\ph^{-1}} & A_1\amalg\cdots \amalg A_j \ar[r]^-{\amalg p_i}
 & \, ^{j}X \ar[r]^-{\nabla} & X,\\} \]
where $\nabla$ is the fold map, the identity on each of the $j$ copies of $X$. To define $\tha$ on
morphisms, let $\ps\colon \colon ^{j}U\rtarr U$ be another injective function, 
and let $\ze\colon \ph\rtarr \ps$ be the unique map in $\sV_G(j)$. For functions 
$q_i\colon B_i\rtarr X$ and bijections $\nu_i\colon A_i\rtarr B_i$ such that 
$q_i\nu_i = p_i$, define $\tha_j(\ze; \nu_1,\cdots, \nu_j)$ to be the unique dotted arrow bijection that makes the following diagram commute.
\[ \xymatrix{
\ph(A_1\amalg\cdots \amalg A_j) \ar[r]^-{\ph^{-1}}
 \ar@{-->}[dd]_{\tha(\ze; \nu_1,\cdots, \nu_j)} 
& A_1\amalg\cdots \amalg A_j \ar[dr]^-{\amalg p_i} 
\ar[dd]^-{\amalg \nu_i}  & &\\
& & ^{j}X \ar[r]^-{\nabla} & X.\\
\ps(B_1\amalg\cdots \amalg B_j)  \ar[r]_-{\ps^{-1}}
& B_1\amalg\cdots \amalg B_j  \ar[ur]_-{\amalg q_i} 
&  &  \\} \]
Then the maps $\tha_j$ specify an action of $\sV_G$ 
on the category $\sE^U_G(X)$.
\end{proof}

We have a multiplicative elaboration, which is similar to \cite[VI.1.9]{MQR} but curiously restricted.  Regarding a 
$G$-space $X$ as a constant $G$-category with object and morphism space both $X$, it makes
sense to speak of an action of the operad $\sV^\times_G$ on the $G$-category $X$.  For example, 
$\sV^\times_G$ acts on $X$ if $X$ is a commutative topological $G$-monoid.  The following result
is closely related to \myref{WVaction}. It has the minor advantage that restriction from $\sV^{\times}_G$
to $\sW_G$ is unnecessary but the major limitation that it only applies to commutative $G$-monoids, not
to general $\sV^\times_G$-algebras.

\begin{prop} If $X$ is a commutative topological $G$-monoid, then $\sE^U_G(X)$ 
is an $E_{\infty}$ ring $G$-category.
\end{prop} 
\begin{proof}
By analogy with the previous proof, for $k\geq 0$, we have functors
\[ \xi\colon\sV^\times_G(k)\times \sE^U_G(X)^k\rtarr \sE^U_G(X^k). \]
With notations as in the previous proof, on objects 
$(\ph;p_1,\cdots,p_k)$, $\xi(\ph;p_1,\cdots,p_k)$, $p_r\colon A_r\rtarr X$, $\xi(\ph; p_1,\cdots p_k)$  is defined to 
be the composite
\[ \xymatrix@1{
 \ph(A_1\times \cdots \times A_k) \ar[r]^-{\ph^{-1}} & A_1\times \cdots \times A_k \ar[r]^-{\times p_k}
 & X^k  \ar[r]^-{\pi} & X,  \\} \]
 where $\pi$ is the $k$-fold product on $X$.
On morphisms $(\ze;\nu_1,\cdots, \nu_k)$, $\nu_r\colon p_r\rtarr q_r$, where the $\nu_i$ are understood to be
order-preserving, $\xi(\ze;\nu_1,\cdots,\nu_k)$ is defined
to be the unique dotted arrow that makes the the following diagram commute.
\[ \xymatrix{
\ph(A_1\times \cdots \times A_k) \ar[r]^-{\ph^{-1}} \ar@{-->}[dd]_{\xi(\ze; \nu_1,\cdots, \nu_k)}
& A_1\times \cdots \times A_k\ar[dr]^-{\times p_i} \ar[dd]^-{\times \nu_i}  &&\\
& & X^k \ar[r]^{\pi} & X.\\
\ps(B_1\times \cdots \times B_k) \ar[r]_-{\ps^{-1}} & B_1\times \cdots \times B_k  \ar[ur]_-{\times q_i} 
&   &   \\} \]
Further details are similar to those in the proof of \cite[VI.1.9]{MQR}
or \cite[4.9]{Rant1}. 
\end{proof}

\subsection{Free $\sV_G$-categories and the $\sV_G$-categories $\sE^U_G(X)$}
The categories $\sE^U_G(X)$ are conceptually simple, and they allow us to 
give the promised genuinely equivariant variant of \myref{CATone}. To see that, we give
a reinterpretation of $\sE^U_G(X)$. Regarding $X$ as a topological
$G$-category as before, we have the topological $G$-category 
$\tilde{\sE}^U_G(j)\times_{\SI_j} X^j$. 

\begin{lem} The topological $G$-categories $\sE^U_G(j,X)$ and 
$\tilde{\sE}^U_G(j)\times_{\SI_j}X^j$ are naturally isomorphic.
\end{lem}
\begin{proof}
For an ordered set $A=(a_1,\cdots,a_j)$ of points of $U$, let a point 
$(A;x_1,\cdot,x_j)$ of
$\sO\!b(\tilde{\sE}_G(j) \times_{\SI_j} X^j)$ correspond to the function 
$p\colon A\rtarr X$ given by $p(a_i)=x_i$.
Similarly, let a point $(\nu\colon A\rtarr B;x_1,\cdots,x_j)$ of 
$\sM\!or(\tilde{\sE}_G(j) \times_{\SI_j} X^j)$ correspond to the 
bijection $\nu\colon p\rtarr q$ over $X$, where $q\nu(a_i) = p(a_i) = x_i$. 
Since we have passed to orbits over $\SI_j$, our specifications are independent
of the ordering of $A$.  These correspondences identify the two categories. 
\end{proof}

Recall that we write $\bV_G$ for the monad on based $G$-categories associated
to the operad $\sV_G$, $|\sV_G|$ for the operad of $G$-spaces obtained by applying 
the classifying space functor $B$ to $\sV_G$, and $\mathbf{V}_G$ for the monad on 
based $G$-spaces associated to $|\sV_G|$.  Recall too that $X_+$ denotes the union 
of the $G$-category $X$ with a disjoint trivial basepoint category 
$\ast$ and that 
\begin{equation}\mylabel{opXtwo}
\bV_G(X_+) = \coprod_{j\geq 0}\sV_G(j)\times_{\SI_j} X^j.
\end{equation}

\begin{thm}\mylabel{identtoo} There is a natural map 
\[ \om\colon \bV_G(X_+)\rtarr \sE_G^U(X)\]
of $\sV_G$-categories, and it induces a weak equivalence 
\[ B\om\colon \mathbf{V}_G(X_+) \rtarr B\sE_G^U(X) \]
of $|\sV_G|$-spaces on passage to classifying spaces.
\end{thm}

\begin{proof} Pick any $G$-fixed point $1\in U$.\footnote{This must not be confused with the convenience basepoint $\ast$
used to define the topology.}  Define an inclusion 
$i\colon X_+\rtarr \sE_G^U(X)$ of based $G$-categories by identifying 
$\ast$ with $\sE_G^U(0,X)$ and mapping $X$ to $\sE_G^U(1,X)$ by sending 
$x$ to the map $1\rtarr x$ from the $1$-point subset $1$ of $U$ to $X$. 
Since $\bV_G(X_+)$ is the free (based) $\sV_G$-category generated by $X_+$, 
$i$ induces the required natural map $\om$. Explicitly, it is the composite
\[  \xymatrix@1{ \bV_G(X_+) \ar[r]^-{\bV_Gi} 
&  \bV_G\sE_G^U(X)) \ar[r]^-{\tha} & \sE_G^U(X).\\} \]
More explicitly still, let $\ul{1}\subset \, ^j\!U$ be the $j$-point subset consisting of the elements $1$ in the 
$j$ summands.  Then $\om$  is the coproduct of the maps
\[ \om_j = i_j\times_{\SI_j}\id\colon\sV_G(j)\times_{\SI_j}X^j \rtarr \tilde{\sE}^U_G(j)\times_{\SI_j}X^j, \]
where $i_j\colon\sV_G(j)\rtarr \tilde{\sE}^U_G(j)$ is the $(G\times \SI_j)$-functor that sends an object 
$\ph\colon ^j\!U\rtarr U$ to the set $\ph(\ul{1})\subset U$ and sends the morphism 
$\nu \colon \ph\rtarr\ps$ to the bijection
\[ \xymatrix@1{ 
\ph(\ul{1} )\ar[r]^-{\ph^{-1}} 
& \ul{1}  \ar[r]^-{\ps} & \ps(\ul 1).\\} \] 
Passing to classifying spaces, $|i_j|$ is a map between universal principal
$(G,\SI_j)$-bundles, both of which are $(G\times \SI_j)$-CW complexes. 
Therefore $|i_j|$ is a $(G\times \SI_j)$-equivariant homotopy equivalence.  
The conclusion follows.
\end{proof}
  
\subsection{The categorical BPQ theorem: second version}\label{subSecbait}
We begin by comparing \myref{identtoo}, which is about $G$-categories, with
Theorems \ref{ident}, \ref{CATone}, and \ref{CATtwo}, which are about $G$-fixed categories. 
Clearly $\sE_G(X)^G$ is a $\sV$-category,
where $\sV = (\sV_G)^G$. By \myref{identtoo}, it is weakly equivalent (in the
homotopical sense) to the $\sV$-category $(\bV_GX_+)^G$.  We also have the
$\sP$-category $\sF_G(X)^G$, which by \myref{CATone} and \myref{subtle}
is equivalent (in the categorical sense) to the $\sP$-category $(\bP_GX_+)^G$. Elaborating \myref{EvsF},
$\sE^U_G(X)^G$ and $\sF_G(X)^G$ are two small models for the category of all finite
$G$-sets and $G$-isomorphisms over $X$ and are therefore equivalent.  To take
the operad actions into account, recall the discussion in \S\ref{Sec2.3}.  We say that a map
of topological $G$-categories is a weak equivalence if its induced map of classifying 
$G$-spaces is a weak equivalence.

\begin{lem} The $\sP_G$-category $\bP_GX_+$ and the $\sV_G$-category $\bV_GX_+$
are weakly equivalent as $(\sP_G\times\sV_G)$-categories. Therefore the $\sP$-category
$(\bP_GX_+)^G$ and the $\sV$-category $(\bV_GX_+)^G$-categories are weakly equivalent
\end{lem}
\begin{proof} The projections 
\[\xymatrix@1{ \bP_GX_+ &  (\bP_G\times\bV_G)(X_+)  \ar[l] \ar[r] & \bV_GX_+\\} \]
are maps of $(\sP_G\times\sV_G)$-categories that induce weak equivalences of 
$|\sP_G\times\sV_G|$-spaces on passage to classifying spaces.
\end{proof}

\begin{thm} The classifying spaces of the $\sP$-category $\sF_G(X)^G$ and the 
$\sV$-category $\sE^U_G(X)^G$ are weakly equivalent as $|\sP\times\sV|$-spaces.
\end{thm}

The conclusion is that, on the $G$-fixed level, the categories $\sE^U_G(X)^G$ and 
$\sF_G(X)^G$ can be used interchangeably as operadically structured versions of
the category of finite $G$-sets over $X$.  On the equivariant level, $\sE^U_G(X)$
but not $\sF_G(X)$ is operadically structured. It is considerably more convenient
than the categories $\bP_G(X_+)$ or $\bV_G(X_+)$.   With the notations
$\bK_G\bV_G(X_+) = \bE_G B\bV_G(X_+) = \bE_G \BV _G(X_+)$ and  
$\bK_G\sE^U_G(X) = \bE_GB\sE^U_G(X)$, we have the following immediate consequence 
of Theorems \ref{SpThree2} and \ref{identtoo}. It is our preferred version of the
categorical BPQ theorem, since it uses the most intuitive categorical input.

\begin{thm}[Categorical Barratt-Priddy-Quillen theorem]\mylabel{GBPQ} 
For $G$-spaces $X$, there is a composite natural weak equivalence
\[ \al\colon \SI^{\infty}_G X_+\rtarr \bK_G\bV_GX_+\rtarr \bK_G\sE^U_G(X). \]
\end{thm} 

\begin{rem}
It is not known how the tom Dieck splitting theorem behaves with respect 
to the Mackey functor structure on homotopy groups.  It seems likely to us
that  this could be analyzed using this version of the BPQ theorem and 
our categorical proof of the splitting.
\end{rem}

\section{Appendix: pairings of operads}\label{SecPair}

We recall the following definition from \cite[1.4]{MayPair}.
It applies equally well equivariantly. We write it element-wise,
but written diagrammatically it applies to operads in any symmetric
monoidal category $\sV$. Write $\mathbf{j} = \{1, \cdots, j\}$ and let 
\[ \otimes \colon \SI_j\times \SI_k \rtarr \SI_{jk} \]
be the homomorphism obtained by identifying $\mathbf{j}\times \mathbf{k}$
with $\mathbf{jk}$ by ordering the set of $jk$ elements $(q,r)$, $1\leq q\leq j$
and $1\leq r\leq k$, lexicographically.  
More precisely, let $\lambda_{j,k}:\mathbf{jk}\rtarr \mathbf{j}\times \mathbf{k}$ be the lexicographic ordering.
Then, given $\rho\in \SI_j$ and $\sigma\in \SI_k$, $\rho\otimes\sigma$ is defined by
 \[ \mathbf{jk} \xrightarrow{\lambda_{j,k}} \mathbf{j}\times \mathbf{k} \xrightarrow{\rho\times \sigma} \mathbf{j}\times \mathbf{k} \xrightarrow{\lambda_{j,k}^{-1}} \mathbf{jk}.\]

For nonnegative integers $h_q$ and $i_r$, let 
\[ \de\colon \coprod _{(q,r)} (\mathbf{h}_q \times \mathbf{i}_r)\rtarr 
(\coprod_q\mathbf{h}_q)\times (\coprod_r\mathbf{i}_r) \]
be the distributivity isomorphism viewed as a permutation via block and lexicographic identifications of the source and target sets with the appropriate set $\mathbf{n}$.  A little more precisely, we define the permutation $\delta$ to be the composite
\[\begin{split}
 \sum_{q,r}\mathbf{h_qi_r}& \xrightarrow{\iso} \mathbf{h_1i_1}\amalg \mathbf{h_1i_2}\amalg \dots \amalg \mathbf{h_j i_k} \xrightarrow{\lambda\amalg \dots \amalg\lambda} \mathbf{h_1}\times \mathbf{i_1} \amalg \dots\amalg\mathbf{h_j}\times \mathbf{i_k} \\
&   \xrightarrow{\mathrm{dist}} (\mathbf{h_1}\amalg\dots\amalg\mathbf{h_j})\times (\mathbf{i_1}\amalg \dots\amalg\mathbf{i_k}) \xrightarrow{\iso} \mathbf{h}\times \mathbf{i} \xrightarrow{\la_{h,i}^{-1}} \mathbf{hi}
\end{split}\]

\begin{defn}\mylabel{pairop} Let $\sC$, $\sD$, and $\sE$ be operads in a symmetric
monoidal category $\sV$ (with product denoted $\otimes$). A pairing of operads
\[ \boxtimes\colon (\sC, \sD)  \rtarr \sE \]
consists of maps 
\[ \boxtimes\colon \sC(j) \otimes \sD(k) \rtarr \sE(jk) \]
in $\sV$ for $j\geq 0$ and $k\geq 0$ such that the diagrammatic versions of the following properties hold.  Let $c\in \sC(j)$ and $d\in \sD(k)$.
\begin{enumerate}[(i)]
\item If $\mu\in \SI_j$ and $\nu\in \SI_k$, then
\[  c\mu\boxtimes d\nu = (c\boxtimes d)(\mu\otimes \nu) \]
\item With $j = k = 1$, $\id\boxtimes \id = \id$.
\item If $c_q\in \sC(h_q)$ for $1\leq q\leq j$ and $d_r\in \sD(i_r)$ 
for $1\leq r\leq k$, then\footnote{The original definition in \cite{MayPair} had $\de$ on the other side in this condition.}

\[\ga(c\boxtimes d;\times_{(q,r)}c_q\boxtimes d_r) 
= \Big[\ga(c;\times_q c_q)\boxtimes \ga(d;\times_r d_r)\Big]\de. \]
\end{enumerate}
\end{defn}

When specialized to spaces, the following definition (which is a variant of \cite[1.2]{MayPair}) gives one
possible starting point for multiplicative infinite loop space theory.

\begin{defn}\mylabel{MayPair}  Let  $\boxtimes\colon (\sC, \sD)  \rtarr \sE$ be a pairing of operads in $\sV$. 
A pairing of a $\sC$-algebra $X$ and a $\sD$-algebra $Y$
to an $\sE$-algebra $Z$ is a map $f\colon X\otimes Y \rtarr Z$ such that the following diagram commutes
for all $j$ and $k$, where $X^j$ denotes the $j$th tensor power in $\sV$ and we write $\tha$ generically for action maps.
\[
\xymatrix{ 
\sC(j)\otimes X^j \otimes \sD(k)\otimes Y^k \ar[d]_{\boxtimes}  \ar[rr]^-{\tha\otimes \tha} & & X\otimes Y \ar[d]^{f} \\
\sE(jk) \otimes (X\otimes Y)^{jk} \ar[r]_-{\id\otimes f^{jk}} & \sE(jk) \otimes Z^{jk} \ar[r]_-{\tha} & Z \\} 
\] 
On the left, $\boxtimes$ denotes the composite 
$$ \xymatrix@1{\sC(j)\otimes X^j \otimes \sD(k)\otimes Y^k \ar[r]^-{\id\otimes t\otimes \id} & 
\sC(j) \otimes \sD(k)\otimes  X^j \otimes Y^k \ar[r]^-{\boxtimes\otimes \la} & \sE(jk) \otimes Z^{jk}. \\} $$
Here, in elementwise notation,  
$$\la((x_1\otimes \cdots \otimes x_j)\otimes (y_1\otimes \cdots \otimes y_k)) = ((x_1\otimes y_1) \otimes \cdots \otimes (x_j\otimes y_k)), $$
where we order the  pairs $(x_q\otimes y_r)$,  $1\leq q \leq j$ and $1\leq r\leq k$, lexicographically. 
\end{defn}

Letting $\sV$ be the category of unbased $G$-spaces, with $\otimes = \times$, 
but then passing to monads on based $G$-spaces, we obtain the following observations.

\begin{prop}\mylabel{pairsm}  For based $G$-spaces $X$ and $Y$,
a pairing $\boxtimes\colon (\sC_G,\sD_G)\rtarr \sE_G$ of operads of $G$-spaces induces a natural pairing
$$\boxtimes\colon \mathbf{C}_G X\sma \mathbf{D}_G Y \rtarr \mathbf{E}_G(X\sma Y)$$
such that the following diagrams commute.
\[ \xymatrix{
X\sma Y \ar[r]^-{\et\sma \et} \ar[dr]_{\et} & \mathbf{C}_GX\sma \mathbf{D}_GY\ar[d]^{\boxtimes}\\
& \mathbf{E}_G(X\sma Y) \\} \]
\[\xymatrix{ 
\mathbf{C}_G \mathbf{C}_G X\sma \mathbf{D}_G\mathbf{D}_GY \ar[rr]^-{\mu\sma \mu} \ar[d]_{\boxtimes} &  &
\mathbf{C}_G X\sma \mathbf{D}_GY \ar[d]^{\boxtimes}\\
\mathbf{E}_G(\mathbf{C}_G X\sma \mathbf{D}_GY) \ar[r]_-{\mathbf{E}_G\boxtimes} & \mathbf{E}_G\mathbf{E}_G(X\sma Y) \ar[r]_{\mu} & \mathbf{E}_G(X\sma Y) \\} \]
The following diagram commutes for any pairing $f\colon X\otimes Y \rtarr Z$ of a 
$\sC_G$-algebra $X$ and a $\sD_G$-algebra $Y$ to an $\sE_G$-algebra $Z$.
\[
\xymatrix{ 
\mathbf{C_G}  X\sma \mathbf{D_G} Y  \ar[rr]^-{\tha\sma \tha} \ar[d]_{\boxtimes}   & &
X\sma Y \ar[d]^{f}\\
\mathbf{E_G} (X\sma Y) \ar[r]_-{\mathbf{E_G} f} &  \mathbf{E_G} Z \ar[r]_-{\tha} & Z. \\} 
\]
\end{prop}
\begin{proof}  The map $\boxtimes$ is induced from the map $\boxtimes$ of the previous definition and the
commutativity of the first two diagrams is checked by chases from \myref{pairop}.  The commutativity of 
the second implies that $\boxtimes$ is a pairing in the sense of \myref{MayPair}.  The commutativity of the
third follows from  \myref{MayPair}. 
\end{proof} 

\begin{exmp}\mylabel{pairopex}  The following commutative diagram, in which we ignore the path space variable for simplicity, 
shows that condition (iii) is satisfied by the pairing $(\sK_V,\sK_W) \rtarr \sK_{V\oplus W}$ defined in  \myref{SteinPair}. This
completes the proof of that result.
{\tiny
\[ \xymatrix{
\mathbf{hi}\times V\times W \ar[r] \ar[d]_{\delta\times \id} & \Big(\coprod\limits_{q,r}\mathbf{h_q i_r}\Big) \times V \times W \ar[r]^{\mathrm{dist}} \ar[d]_{\amalg \la \times \id} & \coprod\limits_{q,r} \Big( \mathbf{h_qi_r}\times V\times W\Big) \ar[d]_{\amalg \la} \ar[r]^-{\amalg c_q\otimes d_r} & \mathbf{j}\times\mathbf{k}\times V\times W \ar[d]^{\mathrm{twist}}
\\
\mathbf{hi}\times V\times W \ar[d]_{\la\times \id} & \Big(\coprod\limits_{q,r}\mathbf{h_q}\times \mathbf{i_r}\Big) \times V\times W \ar[r]^{\mathrm{dist}} \ar[d]_{\overline{\delta}\times \id} & \coprod\limits_{q,r}\Big(\mathbf{h_q}\times \mathbf{i_r}\times V \times W\Big) \ar[d]_{\amalg\, \mathrm{twist}} & \mathbf{j}\times V \times \mathbf{k}\times W \ar[d]^{c\times d}
\\
\mathbf{h}\times \mathbf{i}\times V \times W \ar[r] \ar[d]_{\mathrm{twist}} & (\coprod\limits_q \mathbf{h_q}) \times (\coprod\limits_r \mathbf{i_r}) \times V \times W \ar[d]_{\mathrm{twist}} & \coprod\limits_{q,r} (\mathbf{h_q}\times V\times \mathbf{i_r}\times W) \ar[d]_{\overline{\delta}} \ar[uur]^(.6){\amalg(c_q\times d_r)} & V\times W
\\ 
\mathbf{h}\times V \times \mathbf{i}\times W \ar[r] & (\coprod\limits_q \mathbf{h_q})\times V \times (\coprod\limits_r \mathbf{i_r}) \times W \ar[r]^{\mathrm{dist}} & \Big(\coprod\limits_q \mathbf{h_q}\times V\Big) \times \Big(\coprod\limits_r \mathbf{i_r}\times W\Big) \ar[r]^-{\amalg c_q\times \amalg d_r} & (\coprod\limits_q V) \times (\coprod\limits_r W) \ar[u]_{c\times d} \ar@/^7ex/[uu]^{\id}\\
}\]
}
\end{exmp}

The following counterexample was pointed out to us by Anna Marie Bohmann and Angelica Osorno.   Using a 
more sophisticated categorical framework, we shall explain how to get around the difficulty  in \cite{GMMOAdd, GMMOMult}.

\begin{noexmp}\mylabel{oops}  We show that the pairing (\ref{otimes}) is not a self-pairing of $\sP$. 
Letting $\tau\in \sP(2)$ be the transposition $\tau = (12)$, we calculate
\[ \gamma(\tau\otimes \tau; \id_2\otimes\id_1,\id_2\otimes\id_1,\id_1\otimes\id_1,\id_1\otimes\id_1) = (1526)(3)(4)\]
whereas
\[ \big[\gamma(\tau;\id_2,\id_1)\otimes \gamma(\tau;\id_1,\id_1)\big]\delta = (14526)(3).\]
In this case $\delta$ is the transposition $(23)$.  Thus condition (iii) fails.
\end{noexmp}

\section{Appendix: the double bar construction and the proof of \myref{Trouble}}\label{DoubleTrouble}

The proof of \myref{Trouble} is based on a construction that the senior author has
used for decades in unpublished work and whose algebraic analogue has also 
long been used.  Heretofore he has always found alternative 
arguments that avoid its use in published work, and the topological version seems not to have 
appeared in print.  The construction works 
in great generality with different kinds of 
bar constructions, as described in \cite{Meyer1, Meyer2, Shul} for example.  We restrict attention
to the monadic bar construction used in this paper.  We shall be informal, since 
it is routine to fill in the missing details.

We assume given two monads $\mathbf C$ and 
$\mathbf D$ in some reasonable category $\sU$, and we assume given a morphism
of monads $\io\colon \mathbf C\rtarr \mathbf D$.  We also assume given a right $\mathbf D$-functor 
$\SI\colon \sU\rtarr \sV$ for some other reasonable category $\sV$. Then $\SI$ is
a right $\mathbf D$-functor with the pullback action
\[  \SI\mathbf C \rtarr \SI \mathbf D \rtarr \SI. \]
Let $X$ be a $\mathbf C$-algebra in $\sU$.  Reasonable means in particular 
that we can form ``geometric realizations'' of simplicial objects $X$ as usual, 
tensoring $X$ over the category $\DE$ with a canonical (covariant) simplex functor from $\DE$ to $\sU$ or $\sV$. 

We assume that the functor $\mathbf D$ commutes with geometric realization, so that the realization
of a simplicial $\mathbf D$-algebra is a $\mathbf D$-algebra.  Then the bar 
construction 
$$ \io_! X = B(\mathbf D,\mathbf C,X)$$ 
in $\sU$ specifies an ``extension of scalars'' functor that converts $\mathbf C$-algebras $X$ to 
$\mathbf D$-algebras in a homotopically well-behaved fashion. Since 
$D$ acts on $\SI$, we have the bar construction $B(\SI,\mathbf D,\io_! X)$, and we also have the
bar construction $B(\SI,\mathbf C, X)$, both with values in $\sV$.  Under these assumptions, 
we have the following result.

\begin{thm}\mylabel{Double} There is a natural equivalence $B(\SI,\mathbf D,\io_! X) \htp B(\SI,\mathbf C,X)$.
\end{thm}  

\begin{proof}[Proof of \myref{Trouble}]  We replace $\sU$ by $G\sU$ and $\sV$ by $G\sS\!p$.   We take
$\mathbf C$ to be the monad associated to the operad $\sC_{U^G} = (\sC_G)^G\times \sK_{U^G}$ and $\mathbf D$ to 
be the monad associated to $\sC_U = \sC_{G}\times \sK_U$.  We take $\SI$ to be $\SI^{\infty}_G$, and we recall that 
$\SI^{\infty}_G = i_*\SI^{\infty}$ by \myref{obvious}.
By inspection or a commutation of left adjoints argument, the functor $i_*$ commutes with geometric realization.  Therefore
\[ \mathbf E_G(\io_!X) \equiv B(\SI^{\infty}_G,\mathbf {C_{U}}, \io_! X) \htp B(\SI^{\infty}_G,\mathbf{C_{U^G}},X)\iso i_*B(\SI^{\infty},\mathbf{C_{U^G}},X) \equiv i_*\mathbf E X, \]
where \myref{Double} gives the equivalence.
\end{proof}

\begin{proof}[Proof of \myref{Double}]  We construct the double bar construction
\[ B(\SI,\mathbf D,\mathbf D, \mathbf C,X) \]
as the geometric realization of the bisimplicial object $B_{\bullet,\bullet}(\SI,\mathbf D,\mathbf D, \mathbf C,X)$
in $\sV$ whose $(p,q)$-simplex object is $\SI \mathbf{D}^p\mathbf{D}\mathbf{C}^qX$.  The horizontal 
face and degeneracy operations are those obtained by applying the simplicial bar construction
$B_{\bullet}(\SI,\mathbf D, Y)$ to the $\mathbf D$-algebras $ Y=\mathbf{D}\mathbf{C}^qX$.  The vertical face
and degeneracy operations are those obtained by applying the simplicial bar construction $B_{\bullet}(\UP, \mathbf C,X)$
to the $\mathbf C$-functors $\UP = \SI \mathbf{D}^p\mathbf{D}$.  The geometric realization of a bisimplicial object
is obtained equivalently as the realization of the diagonal simplicial object, the horizontal realization of its 
vertical realization, and the vertical realization of its horizontal realization.   Realizing first vertically 
and then horizontally, we obtain
\[ B(\SI,\mathbf D, B(\mathbf D,\mathbf C, X)) = B(\SI,\mathbf D, i_!X).\] 
Realizing first horizontally and then vertically, we  
obtain the bar construction
\[ B(B(\SI,\mathbf D,\mathbf D), \mathbf C, X) \htp B(\SI, \mathbf C, X). \] 
Here $B(\SI,\mathbf D,\mathbf D)$ is the right $\mathbf C$-functor whose value on a $\mathbf C$-algebra $Y$ is
$B(\SI,\mathbf D,\mathbf DY)$ with right $\mathbf C$-action induced by the $\mathbf C$-action $\mathbf CY \rtarr Y$.  
The equivalence is induced by the standard natural equivalence $B(\SI,\mathbf D,\mathbf DY) \rtarr \SI Y$.
\end{proof}

\begin{rem} The double bar construction can be defined more generally and symmetrically.  Dropping the assumption
that there is a map of monads $\mathbf C\rtarr \mathbf D$, $B(\SI,\mathbf D,\mathbf F,\mathbf C,X)$ is defined 
if $F$ is a left $\mathbf D$-functor and a right $\mathbf C$-functor $\bU\rtarr \bU$ such that the 
following diagram commutes.
\[ \xymatrix{
\mathbf D \mathbf F \mathbf C \ar[r] \ar[d] & \mathbf D\mathbf F \ar[d] \\
\mathbf F \mathbf C \ar[r] & \mathbf F.\\} \]
This can even work when the domain and target categories of $\mathbf F$ differ but agree
with the categories on which $\mathbf C$ and $\mathbf D$ are defined.
\end{rem}

\end{document}